\documentclass[10pt,leqno]{amsart}

\usepackage[margin=1in]{geometry}
\usepackage[T1]{fontenc}
\usepackage{lmodern}
\ifdefined\pdfgentounicode
  \input{glyphtounicode}
  \pdfglyphtounicode{integraldisplay}{222B}
  \pdfglyphtounicode{integraltext}{222B}
  \pdfglyphtounicode{summationdisplay}{2211}
  \pdfglyphtounicode{summationtext}{2211}
  \pdfglyphtounicode{radicalbig}{221A}
  \pdfglyphtounicode{radicalBig}{221A}
  \pdfglyphtounicode{radicalbigg}{221A}
  \pdfglyphtounicode{radicalBigg}{221A}
\fi
\usepackage{microtype}
\usepackage{amsmath,amssymb,amsthm,mathtools}
\allowdisplaybreaks[1]
\usepackage{mathrsfs}
\usepackage{array}
\usepackage{booktabs}
\usepackage{enumitem}
\usepackage{needspace}
\usepackage{comment}
\usepackage[colorlinks=true,linkcolor=blue,citecolor=blue,urlcolor=blue]{hyperref}
\hypersetup{pdftitle={Sharp Aubin inequalities: geometric explanations, 24 November 2026},
pdfauthor={Juncheng Wei, Zetian Yan, Zikai Ye}}

\numberwithin{equation}{section}

\newtheorem{theorem}{Theorem}[section]
\newtheorem{lemma}[theorem]{Lemma}
\newtheorem{proposition}[theorem]{Proposition}

\theoremstyle{remark}
\newtheorem{remark}[theorem]{Remark}

\DeclareMathOperator{\tr}{tr}

\newcommand{\Sn}{\mathbb{S}^n}
\newcommand{\R}{\mathbb{R}}

\newcommand{\Tu}{\mathcal{T}_u}
\newcommand{\Hc}{\mathcal{H}}
\newcommand{\Erel}{\mathcal{E}_{\mathrm{rel}}}
\newcommand{\Lrel}{\mathcal{L}_{\mathrm{rel}}}
\newcommand{\op}{\mathrm{op}}

\date{Revised draft, 24 November 2026}

\begin{document}

\title[Sharp Aubin inequalities on $\mathbb{S}^n$]{Sharp Aubin inequalities on $\mathbb{S}^n$}

\author{Juncheng Wei}
\address{(J. Wei) Department of Mathematics, The Chinese University of Hong Kong, Shatin, N.T., Hong Kong}
\email{wei@math.cuhk.edu.hk}

\author{Zetian Yan}
\address{(Z. Yan) Department of Mathematics, The Chinese University of Hong Kong, Shatin, N.T., Hong Kong}
\email{zetianyan@cuhk.edu.hk}

\author{Zikai Ye}
\address{(Z. Ye) Department of Mathematics, Uppsala University, 751 06 Uppsala, Sweden}
\email{zikai.ye@math.uu.se}

\subjclass[2020]{46E35, 26D10, 47A60}
\keywords{Sharp Sobolev inequalities, zero-barycenter condition, concentration-compactness, second variation, moment estimates}

\begin{abstract}
We prove that for every $n\geq3$ and every $u\in H^1(\mathbb{S}^n)$ satisfying the zero-barycenter condition $\int_{\mathbb{S}^n}\xi|u|^{\frac{2n}{n-2}}\,d\sigma=0$, we have
\begin{equation*}
\left(\int_{\mathbb{S}^n}|u|^{\frac{2n}{n-2}}\,d\sigma\right)^{\frac{n-2}{n}} \leq \int_{\mathbb{S}^n}|u|^2d\sigma
+\frac{4\,2^{-2/n}}{n(n-2)} \int_{\mathbb{S}^n}|\nabla u|^2\,d\sigma,
\end{equation*}
where $d\sigma$ is the normalized volume measure on the unit sphere. Both coefficients on the right-hand side are optimal and the equality holds if and only if $u$ is constant. The proof is mainly based on establishing instability of non-constant critical points along conformal Killing fields by a detailed analysis of the concentration and distribution of the measure $u^{\frac{2n}{n-2}}d\sigma$. In the three-dimensional case, the instability is established using conformal variations along transverse directions.
\end{abstract}
\maketitle

\section{Introduction}\label{sec:intro}

Let $n\geq3$ and equip the unit sphere $\Sn\subset\R^{n+1}$ with its round metric and normalized volume measure $d\sigma=\frac{dV_{\Sn}}{\mathrm{Vol}(\Sn)}$. We write $\xi=(\xi_1,\dots,\xi_{n+1})$ for the position vector and set $p=\frac{2n}{n-2}$, $\alpha=2^{-\frac{2}{n}}$ and $\beta=\frac{4\alpha}{n(n-2)}$. For notational convenience, in this paper, all norms and scalar products on $\Sn$ are taken with respect to $d\sigma$. 

The classical sharp Sobolev inequality on the sphere \cite{Beckner,Case} states that
\begin{equation}\label{eq:sobolev}
\|u\|_p^2\leq\langle u,P_2u\rangle=\|u\|_2^2+\frac{4}{n(n-2)}\|\nabla u\|_2^2
\end{equation}
for every $u\in H^1(\Sn)$, where $P_2=1-\frac{4}{n(n-2)}\Delta$ is the conformal Laplacian on $\mathbb{S}^n$. The Equality holds precisely for constant multiples of conformal factors, see \cite{Beckner,FrankLieb}. The Euclidean version of \eqref{eq:sobolev} is due to Aubin \cite{Aubin76} and Talenti \cite{Talenti}. More generally, sharp Sobolev inequalities of fractional order follow by duality from the sharp Hardy--Littlewood--Sobolev inequality of Lieb \cite{Lieb}. Their spherical versions were obtained by Beckner \cite{Beckner}. The associated conformally covariant operators include the conformal Laplacian $P_2$, the GJMS operators \cite{GJMS} and their fractional-order analogues.

Aubin \cite{Aubin79} proved that for every $\epsilon>0$ there exists $C_\epsilon>0$ such that
\begin{equation}\label{eq:aubin-improved-sobolev}
\|u\|_p^2\leq\frac{4(\alpha+\epsilon)}{n(n-2)}\|\nabla u\|_2^2+C_\epsilon\|u\|_2^2
\end{equation}
for every $u\in H^1(\Sn)$ satisfying the zero-barycenter condition
\begin{equation}\label{eq:zero-barycenter}
\int_{\Sn}\xi|u|^p\,d\sigma=0.
\end{equation}
Hang and Wang \cite{HangWang} gave a new proof of \eqref{eq:aubin-improved-sobolev} using concentration-compactness and extended it to higher-order moment constraints. In \cite{HangWang,Putterman}, optimal leading coefficients were determined for several higher-order constraints. Sun and Zou \cite{SunZou} obtained related inequalities on compact manifolds with an additional $\epsilon$ in the leading coefficient.

The natural sharp form of \eqref{eq:aubin-improved-sobolev} is
\begin{equation}\label{eq:sharp-improved-sobolev}
\|u\|_p^2\leq\alpha\langle u,P_2u\rangle+(1-\alpha)\|u\|_2^2=\|u\|_2^2+\beta\|\nabla u\|_2^2.
\end{equation}
The optimality of the two coefficients follows by considering constant functions and two equally weighted Sobolev extremals concentrating at antipodal points, see \cite[Remark~2.4]{WYY}. However, \eqref{eq:aubin-improved-sobolev} does not give a uniform bound for $C_\epsilon$ as $\epsilon\to0$. Without moment constraints, the existence of a finite lower-order constant in the Sobolev inequality with the sharp leading coefficient on compact Riemannian manifolds was established by Hebey and Vaugon \cite{HebeyVaugon} for the Dirichlet energy and by Aubin and Li \cite{AubinLi} and Druet \cite{Druet99} for $L^q$ norms of the gradient.

Under \eqref{eq:zero-barycenter}, Chang and Yang \cite{ChangYang91} showed that the coefficient of the gradient term in \eqref{eq:sobolev} can be decreased while the coefficient of the $L^2$ term remains one. Li \cite{Li96} gave a simpler proof. For $n\geq5$, Wang and Zhang \cite{WZ} obtained an explicit improved gradient coefficient larger than $\beta$. While revising this paper, we received the manuscript \cite{FJWX} by Frank, Jin, Wang and Xiong, who proved \eqref{eq:aubin-improved-sobolev} with $\epsilon=0$ and a finite lower-order coefficient in every dimension $n\geq3$. They also determined the range of $q$ for which the analogous inequality for $L^q$ norms of the gradient holds with the sharp leading coefficient. Their proof uses stereographic averaging and the sharp gradient stability estimate of Figalli and Zhang \cite{FigalliZhang}. An earlier connection between improved spherical Sobolev inequalities under nonlinear moment constraints and the stability estimate of Bianchi and Egnell \cite{BianchiEgnell} was obtained by Dolbeault, Esteban and Loss \cite{DEL} via stereographic projection. Recent advances in the quantitative stability of the Sobolev inequality include the work of K\"onig \cite{Konig23,Konig25} on the optimal stability constant and that of Dolbeault, Esteban, Figalli, Frank and Loss \cite{DEFFL} on its optimal dimensional dependence.

In \cite{WYY}, we established \eqref{eq:sharp-improved-sobolev} for antipodally even functions and for axially symmetric functions satisfying \eqref{eq:zero-barycenter}. The antipodally even case is precisely the conjecture of Dolbeault \cite[Section~4]{Dolbeault}. Its proof relies on a strict comparison of the isoperimetric profiles of $\mathbb{RP}^n$ and $\Sn$. In the axially symmetric case, we used a compactness reduction and proved the instability of every nonconstant smooth constrained critical point arising in the minimization problem.

A related problem is the Chang--Yang conjecture on the sharp Moser--Trudinger--Onofri inequality under the zero-barycenter condition on the exponential density. Gui and Moradifam \cite{GuiMoradifam} proved the conjecture on $\mathbb{S}^2$. The higher-dimensional version was established by Gui, Li, Wei and Ye \cite{GLWY}. The corresponding critical points were classified by Gui, Hu, Li, Wei, Xie and Ye \cite{GHLWXY}.

In this paper, we prove \eqref{eq:sharp-improved-sobolev} without any symmetry assumption.

\begin{theorem}\label{thm:sharp-improved-sobolev}
Let $u\in H^1(\Sn)$ satisfy \eqref{eq:zero-barycenter}. Then we have
\begin{equation}\label{eq:centered-sobolev}
\|u\|_p^2\leq\|u\|_2^2+\beta\|\nabla u\|_2^2.
\end{equation}
Equality holds if and only if $u$ is constant. Both coefficients on the right-hand side are optimal.
\end{theorem}

In Section~\ref{sec:variational-reduction}, we consider a relaxed minimization problem in which the critical measure is allowed to have at most countably many atoms. We show that every nonnegative minimizer either consists of two atoms of mass $\frac{1}{2}$ at antipodal points, in which case its energy is one, or is a smooth positive stable solution of the Euler--Lagrange equation. Thus Theorem~\ref{thm:sharp-improved-sobolev}follows from the following classification result.

\begin{theorem}\label{thm:stable-critical-point-classification}
Suppose that $u>0$ is a smooth function on $\Sn$ satisfying
\begin{equation}\label{eq:critical-equation}
u-\beta\Delta u=(a+b\cdot\xi)u^{p-1},\qquad \int_{\Sn}u^p\,d\sigma=1,\qquad \int_{\Sn}\xi u^p\,d\sigma=0,\qquad a\leq1,
\end{equation}
where $a\in\R$ and $b\in\R^{n+1}$. Let
\begin{equation*}
Q(f,h)=\int_{\Sn}[fh+\beta\nabla f\cdot\nabla h-(p-1)(a+b\cdot\xi)u^{p-2}fh]\,d\sigma
\end{equation*}
and
\begin{equation*}
\Tu=\left\{f\in H^1(\Sn):\int_{\Sn}u^{p-1}f\,d\sigma=0,\quad \int_{\Sn}\xi u^{p-1}f\,d\sigma=0\right\}.
\end{equation*}
If $Q(f,f)\geq0$ for all $f\in\Tu$, then $u=1$, $a=1$ and $b=0$.
\end{theorem}

Testing \eqref{eq:critical-equation} against $u$ gives $a=\|u\|_2^2+\beta\|\nabla u\|_2^2>0$. For a nonconstant solution, write
\begin{equation}\label{intro:coordinates}
r=\frac{\|u\|_2^2}{a}\in(0,1),\qquad M=\int_{\Sn}\xi\xi^Tu^p\,d\sigma.
\end{equation}
Using \eqref{eq:sobolev}, we have
\begin{equation}\label{intro:energy-lower}
\frac{\alpha}{1-(1-\alpha)r}\leq a\leq1,\qquad \beta\|\nabla u\|_2^2=a(1-r).
\end{equation}
The ratio $r$ measures the proportion of the energy contributed by the $L^2$ term and $M$ is the matrix of second moments of $u^p\,d\sigma$. For positive functions normalized in $L^p$, we use the same definitions with $a=\|u\|_2^2+\beta\|\nabla u\|_2^2$. If $r\to1$ along such a sequence with bounded energy, the Poincar\'e and Sobolev inequalities imply convergence to $1$ in $H^1(\Sn)$ and hence $M\to\frac{I}{n+1}$. If $r\to0$ along such a sequence satisfying \eqref{eq:zero-barycenter} with limiting energy at most one, the concentration-compactness argument in Section~\ref{sec:variational-reduction} shows that along a subsequence
\begin{equation*}
u^p\,d\sigma\rightharpoonup\frac{1}{2}(\delta_e+\delta_{-e}),\qquad M\to e\otimes e
\end{equation*}
for some $e\in\Sn$.

To illustrate the difference between the three-dimensional case and the higher-dimensional cases, we consider a family of antipodally even functions connecting the constant function to two antipodal bubbles. Fix $e\in\Sn$, write $x=e\cdot\xi$ and define
\begin{equation*}
v_\tau=(\cosh\tau-\sinh\tau\,x)^{-\frac{n-2}{2}},\qquad u_\tau=\frac{v_\tau+v_{-\tau}}{\|v_\tau+v_{-\tau}\|_p},\qquad \tau\geq0.
\end{equation*}
We have $u_0=1$ and $u_\tau^p\,d\sigma\rightharpoonup\frac{1}{2}(\delta_e+\delta_{-e})$ as $\tau\to\infty$. By the antipodally even case of \cite{WYY}, the energy of $u_\tau$ is strictly greater than one for every $\tau>0$. As $\tau\to0$, a direct calculation gives
\begin{equation*}
\|u_\tau\|_2^2+\beta\|\nabla u_\tau\|_2^2-1=\frac{n^2(n-2)[2\alpha(n+1)-n]}{8(n+1)^2(n+3)}\tau^4+O(\tau^6).
\end{equation*}

The eigenvalues of the moment matrix of $u_\tau$ are
\begin{equation*}
m=\int_{\Sn}x^2u_\tau^p\,d\sigma,\qquad \frac{1-m}{n},\dots,\frac{1-m}{n}.
\end{equation*}
As $\tau\to\infty$, we have $1-m\sim\frac{4n}{n-2}e^{-2\tau}$. Then the two interaction integrals satisfy
\begin{equation*}
\int_{\Sn}v_\tau^{p-1}v_{-\tau}\,d\sigma=O(e^{-(n-2)\tau}),\qquad \int_{\Sn}v_\tau v_{-\tau}\,d\sigma=O(e^{-(n-2)\tau}).
\end{equation*}
On the other hand, a calculation in stereographic coordinates gives
\begin{equation*}
\|u_\tau\|_2^2\sim\frac{1}{\alpha}\begin{cases}
(4+\frac{8}{\pi})e^{-\tau}, & \text{if }n=3,\\
12\tau e^{-2\tau}, & \text{if }n=4,\\
\frac{4(n-1)}{n-4}e^{-2\tau}, & \text{if }n\geq5.
\end{cases}
\end{equation*}
In particular, the interaction integral $\int_{\Sn}v_\tau v_{-\tau}\,d\sigma$ contributes to the leading coefficient of $\|u_\tau\|_2^2$ only in dimension three.

The energy near the antipodal limit satisfies
\begin{equation*}
\|u_\tau\|_2^2+\beta\|\nabla u_\tau\|_2^2-1=(1-\alpha)\|u_\tau\|_2^2-\int_{\Sn}v_\tau^{p-1}v_{-\tau}\,d\sigma+o(e^{-(n-2)\tau}).
\end{equation*}
Consequently, as $\tau\to\infty$,
\begin{equation*}
\|u_\tau\|_2^2+\beta\|\nabla u_\tau\|_2^2-1\sim\begin{cases}
[(2^{\frac{2}{3}}-1)(4+\frac{8}{\pi})-\frac{16}{3\pi}]e^{-\tau}, & \text{if }n=3,\\
12(\sqrt{2}-1)\tau e^{-2\tau}, & \text{if }n=4,\\
\frac{4(n-1)}{n-4}(2^{\frac{2}{n}}-1)e^{-2\tau}, & \text{if }n\geq5.
\end{cases}
\end{equation*}
All leading coefficients are positive. However, the mechanism behind the positivity differs with the dimension. If $n\geq4$, the positive $L^2$ contribution dominates the negative interaction term. If $n=3$, these two contributions have the same order and their leading coefficients must be compared. This distinction motivates our separate treatment of the three-dimensional case. Instead of relying on variations along the principal direction alone, we derive quantitative concentration estimates and use conformal variations in transverse directions. Related differences between dimension three and higher dimensions arise in the energy asymptotics of the Brezis--Nirenberg problem, see Frank-K\"onig-Kova\v{r}\'ik \cite{FKK20,FKK21}.

To prove Theorem~\ref{thm:stable-critical-point-classification}, we use conformal variations modified to preserve both constraints to first order. For a unit eigenvector $Me=me$, write $x=e\cdot\xi$. As shown in Section~\ref{sec:common},
\begin{equation}\label{unified:variation}
\Phi_e=\nabla x\cdot\nabla u-\frac{n-2}{2}xu+\frac{1-m}{pm}xu\in\Tu.
\end{equation}
The term $\frac{1-m}{pm}xu$ compensates for the change in barycenter under the conformal variation. We will show that $Q(\Phi_e,\Phi_e)<0$ for some eigenvector $e$.

If $M=\frac{I}{n+1}$, the sum of the second variations in an orthonormal eigenbasis is negative by \eqref{eq:isotropic-trace}. For unequal eigenvalues, Proposition~\ref{new:weighted:proposition} gives instability by a weighted sum of the same variations in the range \eqref{eq:small-eigenvalue}. The weights are chosen so that the terms involving the second moments of $u^2$ can be summed explicitly. The third moments are controlled by \eqref{eq:joint-moments}.

For $n\geq4$, the remaining ranges of $\lambda_1(M)$ are covered by the estimates in a principal direction in Propositions~\ref{new:four-principal} and~\ref{new:high-principal}. In dimension four we use the Cauchy--Schwarz inequality in $L^2(u^4\,d\sigma)$, see Lemma~\ref{new:four-dual}. In dimensions five and six we use a cubic Hermite interpolation inequality for $|z|^p$ with $p\geq3$, see Lemma~\ref{common:affine-lemma}. If $n\geq7$, then $2<p<3$ and we instead apply the Sobolev inequality to the positive and negative parts of translated affine multiples of $u$, see Lemma~\ref{new:high-translates}. In the study of quantitative Struwe decomposition, Ciraolo, Figalli and Maggi \cite{CFM} established a linear estimate in the single-bubble case. A similar dimensional threshold at $n=6$ occurs in the multi-bubble case: the linear estimate of Figalli and Glaudo \cite{FG} for $3\leq n\leq5$ is replaced by a sharp logarithmic correction for $n=6$ and power-type estimates for $n\geq7$, as shown by Deng, Sun and Wei \cite{DSW}. More recently, Jin, Xiong and Zhou \cite{JXZ} established {\L}ojasiewicz--Simon inequalities for the Yamabe functional near bubbling configurations with or without a regular component.

In dimension three, Proposition~\ref{new:weighted:proposition} gives instability if $m=\lambda_1(M)\leq\frac{3}{10}$. The estimate \eqref{eq:affine-bound} does not give a negative second variation in the principal direction when $1-m=O(r)$ as $r\to0$, see Remark~\ref{new:three-principal-obstruction}. Instead, we use the second variations in the directions orthogonal to the principal eigenvector. Averaging the dual Sobolev inequality gives an estimate for
\begin{equation}\label{intro:D-origin}
1-\left[\frac{2}{\pi}+\left(1-\frac{2}{\pi}\right)\tr M^2\right]^{\frac{2}{3}},
\end{equation}
which vanishes for a probability measure satisfying the zero-barycenter condition if and only if the measure consists of two equally weighted atoms at antipodal points. In Section~\ref{sec:three}, spherical harmonic estimates first give $r<\frac{43}{64}$. Comparing the upper and lower bounds for \eqref{intro:D-origin} then yields, by Lemma~\ref{n3:inputs:localization},
\begin{equation*}
0<r<\frac{1}{64},\qquad \frac{5}{8}r<1-m<\frac{2}{3}r.
\end{equation*}
Under these conditions, Lemma~\ref{n3:small:transverse-instability} gives a negative weighted sum of the transverse second variations. Together, these estimates prove Theorem~\ref{thm:stable-critical-point-classification}.

\subsection*{Acknowledgments}
The authors thank Lu Chen and Haolin Liu for useful discussions on $n=3$ case.
Juncheng Wei was supported by National R\&D Program of China (Grant No.\ 2022YFA1005602) and Hong Kong General Research Fund (No.\ 14303125) ``On Fujita equation in the critical or supercritical regime''. Zetian Yan is substantially supported by research grants from the Research Grants Council of the Hong Kong Special Administrative Region, China [Project No.: CUHK 14307123, CUHK 14306624 and CUHK 14308325].

\subsection*{Disclosure on AI assistance}

The manuscript was written by the authors, and the original ideas and insights presented in it are their own. ChatGPT/Codex was used to assist with checking calculations and polishing selected passages. The authors independently verified the content, reviewed the final text, and take full responsibility for the paper.

\section{Variational identities and Sobolev estimates}\label{sec:common}

For a unit vector $e\in\Sn$, write $x=e\cdot\xi$ and define
\begin{equation}\label{eq:conformal-direction}
Z_e=\nabla x\cdot\nabla u-\frac{n-2}{2}xu.
\end{equation}

\begin{lemma}\label{lem:conformal-identities}
Suppose $u>0$ is a smooth solution of
\begin{equation*}
u-\beta\Delta u=(a+b\cdot\xi)u^{p-1}.
\end{equation*}
Then for any $e\in\Sn$, we have
\begin{equation}\label{eq:linearized-equation}
[1-\beta\Delta-(p-1)(a+b\cdot\xi)u^{p-2}]Z_e=2(1-\alpha)xu+[b\cdot e-x(b\cdot\xi)]u^{p-1}
\end{equation}
and
\begin{equation}\label{eq:conformal-pairing}
\int_{\Sn} xuZ_e\,d\sigma=\frac{1}{2}\int_{\Sn}(3x^2-1)u^2\,d\sigma.
\end{equation}
Moreover, for any $F,H\in C^{\infty}(\Sn)$, we have
\begin{equation}\label{eq:conformal-integration}
\int_{\Sn} Fu^{p-1}Z_e\,d\sigma=-\frac{1}{p}\int_{\Sn}\nabla x\cdot\nabla F\,u^p\,d\sigma
\end{equation}
and
\begin{equation}\label{eq:product-test}
\int_{\Sn}[FHu^2+\beta\nabla(Fu)\cdot\nabla(Hu)]\,d\sigma=\int_{\Sn}(a+b\cdot\xi)FHu^p\,d\sigma+\beta\int_{\Sn} u^2\nabla F\cdot\nabla H\,d\sigma.
\end{equation}
\end{lemma}

\begin{proof}
Recall that $\nabla^2x=-xg_c$, $\Delta x=-nx$ and
\begin{equation*}
\Delta(\nabla x\cdot\nabla u) =\nabla x\cdot\nabla\Delta u-2x\Delta u +(n-2)\nabla x\cdot\nabla u.
\end{equation*}
The commutator identity and the equation for $u$ give \eqref{eq:linearized-equation}.

Integration by parts gives
\begin{equation*}
\begin{aligned}
&\int_{\Sn} xuZ_e\,d\sigma=\frac{1}{2}\int_{\Sn} x\nabla x\cdot\nabla(u^2)\,d\sigma-\frac{n-2}{2}\int_{\Sn} x^2u^2\,d\sigma\\
=&-\frac{1}{2}\int_{\Sn}(1-(n+1)x^2)u^2\,d\sigma-\frac{n-2}{2}\int_{\Sn} x^2u^2\,d\sigma=\frac{1}{2}\int_{\Sn}(3x^2-1)u^2\,d\sigma,
\end{aligned}
\end{equation*}
which proves \eqref{eq:conformal-pairing}.

The identity $p(n-2)=2n$ gives
\begin{equation*}
\int_{\Sn} Fu^{p-1}Z_e\,d\sigma=\frac{1}{p}\int_{\Sn} F\nabla x\cdot\nabla(u^p)\,d\sigma-\frac{n-2}{2}\int_{\Sn} Fxu^p\,d\sigma=-\frac{1}{p}\int_{\Sn}\nabla x\cdot\nabla F\,u^p\,d\sigma.
\end{equation*}

Finally, testing the equation against $FHu$ gives
\begin{equation*}
\int_{\Sn}[FHu^2+\beta\nabla u\cdot\nabla(FHu)]\,d\sigma=\int_{\Sn}(a+b\cdot\xi)FHu^p\,d\sigma.
\end{equation*}
Using
\begin{equation*}
\nabla(Fu)\cdot\nabla(Hu)=\nabla u\cdot\nabla(FHu)+u^2\nabla F\cdot\nabla H,
\end{equation*}
we obtain \eqref{eq:product-test}.
\end{proof}

In particular, for any $F,H\in C^\infty(\Sn)$, it follows from \eqref{eq:product-test} that
\begin{equation*}
Q(Fu,Hu)=\beta\int_{\Sn} u^2\nabla F\cdot\nabla H\,d\sigma-(p-2)\int_{\Sn}(a+b\cdot\xi)FHu^p\,d\sigma.
\end{equation*}

For the rest of this section, let $u$ be a nonconstant solution of \eqref{eq:critical-equation}. We use $r$ and $M$ as defined in \eqref{intro:coordinates}. We order the eigenvalues as $\lambda_1(M)\geq\cdots\geq\lambda_{n+1}(M)$. Matrix inequalities are understood in the sense of quadratic forms.

The following lemma gives several moment identities for $u$ and its gradient.

\begin{lemma}\label{common:moment-identities}
We have
\begin{equation}\label{eq:first-moment}
(1-\alpha)\int_{\Sn}\xi u^2\,d\sigma=-\frac{1}{p}(I-M)b,
\end{equation}
\begin{equation}\label{eq:gradient-first-moment}
\beta\int_{\Sn}\xi|\nabla u|^2\,d\sigma=\left[M+\frac{n-2+2\alpha}{2n(1-\alpha)}(I-M)\right]b
\end{equation}
and
\begin{equation}\label{eq:gradient-second-moment}
\beta\int_{\Sn}\xi\xi^T|\nabla u|^2\,d\sigma=aM+\int_{\Sn}(b\cdot\xi)\xi\xi^Tu^p\,d\sigma+\beta arI-[1+(n+1)\beta]\int_{\Sn}\xi\xi^Tu^2\,d\sigma.
\end{equation}
In particular, we have
\begin{equation*}
0<\lambda_{n+1}(M)\leq\lambda_1(M)<1,\qquad \tr M=1,
\end{equation*}
and
\begin{equation}\label{eq:multiplier-bound}
\frac{|b|}{a}<1-r,\qquad \left[\lambda_1(M)+\frac{n-2+2\alpha}{2n(1-\alpha)}(1-\lambda_1(M))\right] \frac{|b|}{a}\leq1-r.
\end{equation}
\end{lemma}

\begin{proof}
By the conformal invariance of $\alpha\|u\|_2^2+\beta\|\nabla u\|_2^2$, we have for any $e\in\Sn$,
\begin{equation*}
\int_{\Sn}(\alpha uZ_e+\beta\nabla u\cdot\nabla Z_e)\,d\sigma=0.
\end{equation*}
Testing \eqref{eq:critical-equation} against $Z_e$ and using \eqref{eq:conformal-integration} and the identity $\int_{\Sn} uZ_e\,d\sigma=\int_{\Sn} xu^2\,d\sigma$, we obtain
\begin{equation*}
(1-\alpha)\int_{\Sn} xu^2\,d\sigma=(1-\alpha)\int_{\Sn} uZ_e\,d\sigma=\int_{\Sn}(a+b\cdot\xi)u^{p-1}Z_e\,d\sigma=-\frac{1}{p}\int_{\Sn}[b\cdot e-x(b\cdot\xi)]u^p\,d\sigma=-\frac{1}{p} e\cdot(I-M)b.
\end{equation*}
Since $e$ is arbitrary, this proves \eqref{eq:first-moment}.

Testing \eqref{eq:critical-equation} against $xu$ gives \eqref{eq:gradient-first-moment} by \eqref{eq:first-moment} and the identity $\int_{\Sn} u\nabla x\cdot\nabla u\,d\sigma=\frac{n}{2}\int_{\Sn} xu^2\,d\sigma$.

Similarly, testing against $\xi_i\xi_ju$ and using $\Delta(\xi_i\xi_j)=2\delta_{ij}-2(n+1)\xi_i\xi_j$ gives \eqref{eq:gradient-second-moment}.

Positivity of $u$ and its normalization give $0<\lambda_{n+1}(M)\leq\lambda_1(M)<1$ and $\tr M=1$. As $\alpha>\frac{n+2}{2n+2}$, the smallest eigenvalue of
\begin{equation*}
M+\frac{n-2+2\alpha}{2n(1-\alpha)}(I-M)
\end{equation*}
is
\begin{equation*}
\lambda_1(M)+\frac{n-2+2\alpha}{2n(1-\alpha)}(1-\lambda_1(M))>1.
\end{equation*}
Therefore, \eqref{eq:multiplier-bound} follows from taking norms in \eqref{eq:gradient-first-moment} and using
\begin{equation*}
\beta|\int_{\Sn}\xi|\nabla u|^2\,d\sigma|\leq\beta\int_{\Sn}|\nabla u|^2\,d\sigma=a(1-r).
\end{equation*}
\end{proof}

For every unit eigenvector $Me=me$, we write $t=\frac{1-m}{pm}$ and consider the variation
\begin{equation*}
\Phi_e=Z_e+txu.
\end{equation*}

\begin{lemma}\label{common:second-variation}
For every unit eigenvector $Me=me$, we have $\Phi_e\in\Tu$ and
\begin{equation}\label{eq:conformal-second-variation}
\begin{aligned}
\frac{Q(\Phi_e,\Phi_e)}{a} &=[\beta t^2-(1-\alpha)]r-(p-2)t^2m+\frac{(1-\alpha)(3+4t)-\beta t^2}{a}\int_{\Sn} x^2u^2\,d\sigma\\
&\quad-\left[\frac{2}{p}+2t+(p-2)t^2\right]\frac{1}{a}\int_{\Sn}(b\cdot\xi)x^2u^p\,d\sigma.
\end{aligned}
\end{equation}
\end{lemma}

\begin{proof}
Taking $F=1$ and $F=\xi_j$ in \eqref{eq:conformal-integration} gives
\begin{equation*}
\int_{\Sn} u^{p-1}\Phi_e\,d\sigma=0
\end{equation*}
and
\begin{equation*}
\int_{\Sn}\xi u^{p-1}\Phi_e\,d\sigma=-\frac{1}{p}(I-M)e+tMe=\left(-\frac{1-m}{p}+tm\right)e=0
\end{equation*}
respectively. Hence, we have $\Phi_e\in\Tu$.

Using \eqref{eq:linearized-equation}, \eqref{eq:conformal-integration}, \eqref{eq:conformal-pairing} and \eqref{eq:product-test}, we have
\begin{equation*}
\begin{aligned}
Q(Z_e,Z_e) &=(1-\alpha)\left(3\int_{\Sn} x^2u^2\,d\sigma-\int_{\Sn} u^2\,d\sigma\right) -\frac{2}{p}\int_{\Sn}(b\cdot\xi)x^2u^p\,d\sigma,\\
Q(Z_e,xu) &=2(1-\alpha)\int_{\Sn} x^2u^2\,d\sigma-\int_{\Sn}(b\cdot\xi)x^2u^p\,d\sigma,\\
Q(xu,xu) &=\beta\int_{\Sn}(1-x^2)u^2\,d\sigma -(p-2)\left[am+\int_{\Sn}(b\cdot\xi)x^2u^p\,d\sigma\right].
\end{aligned}
\end{equation*}
Equation \eqref{eq:conformal-second-variation} now follows from
\begin{equation*}
Q(\Phi_e,\Phi_e)=Q(Z_e,Z_e)+2tQ(Z_e,xu)+t^2Q(xu,xu).
\end{equation*}
\end{proof}

In our earlier work on the axially symmetric case \cite[Lemma~4.3]{WYY}, we proved that $|\int_{\Sn} x^3u^p\,d\sigma|\leq m(1-m)$ along the symmetry axis. The following lemma extends this estimate to $\int_{\Sn}(b\cdot\xi)x^2u^p\,d\sigma$ in \eqref{eq:conformal-second-variation}.

\begin{lemma}\label{lem:moment-projection}
For every unit vector $v$ and every real function $H\in L^2(u^p d\sigma)$,
\begin{equation}\label{eq:moment-projection}
\begin{aligned}
\pm\int_{\Sn}(v\cdot\xi)\xi\xi^Tu^p\,d\sigma&\leq M-M^2,\\
\left(\int_{\Sn}\xi Hu^p\,d\sigma\right)^TM^{-1}\int_{\Sn}\xi Hu^p\,d\sigma&\leq\int_{\Sn} H^2u^p\,d\sigma-\left(\int_{\Sn} Hu^p\,d\sigma\right)^2.
\end{aligned}
\end{equation}
Consequently, for every unit eigenvector $Me=me$, we have
\begin{equation}\label{eq:axial-third-moment}
\left|\int_{\Sn}(b\cdot\xi)x^2u^p\,d\sigma\right|\leq|b|m(1-m),\qquad\left|\int_{\Sn} x^3u^p\,d\sigma\right|\leq m(1-m).
\end{equation}
In particular, we have
\begin{equation}\label{eq:transverse-third-moment}
\left|\int_{\Sn}(\xi-xe)x^2u^p\,d\sigma\right|^2\leq m(1-m)^2.
\end{equation}

In an orthonormal eigenbasis $Me_i=\lambda_i e_i$, write $\xi_i=e_i\cdot\xi$. Then for all $c_1,\ldots,c_{n+1}\in \R$, we have
\begin{equation}\label{eq:joint-moments}
\sum_i\frac{\left(\int_{\Sn}\xi_i\sum_jc_j\xi_j^2u^p\,d\sigma\right)^2}{\lambda_i(1-\lambda_i)}\leq\sum_i\lambda_ic_i^2-\left(\sum_i\lambda_ic_i\right)^2.
\end{equation}
\end{lemma}

\begin{proof}
For arbitrary $y\in\R^{n+1}$, we write $z=My$. The zero-barycenter condition and Cauchy--Schwarz inequality give
\begin{equation*}
|z|^4\leq\int_{\Sn}(1-v\cdot\xi)(y\cdot\xi)^2u^p\,d\sigma \int_{\Sn}\frac{((\xi-v)\cdot z)^2}{1-v\cdot\xi}u^p\,d\sigma\leq|z|^2\int_{\Sn}(1-v\cdot\xi)(y\cdot\xi)^2u^p\,d\sigma.
\end{equation*}
The second inequality uses
\begin{equation*}
\frac{((\xi-v)\cdot z)^2}{1-v\cdot\xi} \leq(1+v\cdot\xi)|z|^2-2(v\cdot z)(\xi\cdot z),
\end{equation*}
and the zero-barycenter condition. Both choices of sign for $v$ give the matrix inequality in \eqref{eq:moment-projection}.

The other estimate in \eqref{eq:moment-projection} follows from
\begin{equation*}
\int_{\Sn}\left(H(\xi)-\int_{\Sn} H(\zeta)u^p(\zeta)\,d\sigma(\zeta)-\xi^TM^{-1}\int_{\Sn}\zeta H(\zeta)u^p(\zeta)\,d\sigma(\zeta)\right)^2u^p(\xi)\,d\sigma(\xi)\geq0.
\end{equation*}

If $b\neq0$, use $v=\frac{b}{|b|}$ for the first bound in \eqref{eq:axial-third-moment}. If $b=0$, that bound is immediate. Taking $v=e$ gives the second bound in either case.

With $H=x^2$, the second inequality in \eqref{eq:moment-projection} gives \eqref{eq:transverse-third-moment} because each eigenvalue on $e^\perp$ is at most $1-m$ and $x^4\leq x^2$.

For $H=\sum_i c_i\xi_i^2$, the same inequality gives
\begin{equation*}
\sum_i\frac{(\int_{\Sn}\xi_iHu^p\,d\sigma)^2}{\lambda_i} \leq\int_{\Sn} H^2u^p\,d\sigma-\left(\int_{\Sn} Hu^p\,d\sigma\right)^2.
\end{equation*}

Since $\nabla H=2\sum_i c_i\xi_i e_i-2H\xi$ and $\int_{\Sn}\xi u^p\,d\sigma=0$, we have
\begin{equation*}
\int_{\Sn}\nabla H u^p\,d\sigma=-2\int_{\Sn}\xi Hu^p\,d\sigma.
\end{equation*}

The Cauchy--Schwarz inequality applied to $\nabla(v\cdot\xi)\cdot\nabla H$ yields
\begin{equation*}
\sum_i\frac{(\int_{\Sn}\xi_iHu^p\,d\sigma)^2}{1-\lambda_i}\leq\frac{1}{4}\int_{\Sn}|\nabla H|^2u^p\,d\sigma.
\end{equation*}
Taking the sum and using
\begin{equation*}
\frac{1}{\lambda_i}+\frac{1}{1-\lambda_i}=\frac{1}{\lambda_i(1-\lambda_i)}
\end{equation*}
and
\begin{equation*}
H^2+\frac{1}{4}|\nabla H|^2=\sum_i c_i^2\xi_i^2,
\end{equation*}
we obtain \eqref{eq:joint-moments}.
\end{proof}

For a unit eigenvector $Me=me$, define
\begin{equation}\label{eq:nodal-remainder}
\mathcal R_N=\frac{\alpha\|xu\|_2^2+\beta\|\nabla(xu)\|_2^2}{a}-m=\frac{1}{a}\int_{\Sn}(b\cdot\xi)x^2u^p\,d\sigma+\beta r-\frac{1-\alpha+\beta}{a}\int_{\Sn} x^2u^2\,d\sigma.
\end{equation}

We use \eqref{eq:nodal-remainder} to eliminate the term $\int_{\Sn} x^2u^2\,d\sigma$ from \eqref{eq:conformal-second-variation}. Define
\begin{equation}\label{eq:gamma}
\Gamma=\frac{nm[(1-\alpha)(3+4t)-\beta t^2]}{1-\alpha+\beta}
\end{equation}
and
\begin{equation}\label{eq:cubic-coefficient}
C_m=(1-\alpha+\beta)\left[\Gamma-nm\left(\frac{2}{p}+2t+(p-2)t^2\right)\right].
\end{equation}
Substituting \eqref{eq:nodal-remainder} into \eqref{eq:conformal-second-variation} gives
\begin{equation}\label{eq:principal-variation}
\frac{nm}{a}Q(\Phi_e,\Phi_e)=(1-\alpha)[2nm(1+2t)-\Gamma]r-\frac{2}{p}(1-m)^2-\Gamma\mathcal R_N+\frac{C_m}{a(1-\alpha+\beta)}\int_{\Sn}(b\cdot\xi)x^2u^p\,d\sigma.
\end{equation}
Thus, when $\Gamma>0$, the lower bounds for $\mathcal R_N$ give upper bounds for $Q(\Phi_e,\Phi_e)$, with the remaining cubic term $\int_{\Sn}(b\cdot\xi)x^2u^p\,d\sigma$ controlled by \eqref{eq:axial-third-moment}.

\begin{lemma}\label{lem:nodal-sobolev}
For every real $H\in L^p(u^p d\sigma)$, we have
\begin{equation}\label{eq:nodal-norm}
\alpha\bigl(\|H_+u\|_p^2+\|H_-u\|_p^2\bigr) \geq\int_{\Sn} H^2u^p\,d\sigma-(1-\alpha)\left(\int_{\Sn} Hu^p\,d\sigma\right)^2.
\end{equation}
For a unit eigenvector $Me=me$ and every $\tau\in \R$, we have
\begin{equation}\label{eq:translated-energy}
\frac{\alpha\|(\tau+x)u\|_2^2+\beta\|\nabla((\tau+x)u)\|_2^2}{a}=[1-(1-\alpha)r]\tau^2+2m(1+t)\tau\frac{b\cdot e}{a}+m+\mathcal R_N\geq\alpha\tau^2+m.
\end{equation}
Consequently,
\begin{equation}\label{eq:axial-discriminant}
m^2(1+t)^2\left(\frac{b\cdot e}{a}\right)^2\leq(1-\alpha)(1-r)\mathcal R_N.
\end{equation}

More generally, for every unit vector $v$, we have
\begin{equation}\label{eq:directional-discriminant}
\frac{[v\cdot(I+(p-1)M)b]^2}{p^2a^2}\leq(1-\alpha)(1-r)\left[\beta r+\frac{1}{a}\int_{\Sn}(b\cdot\xi)(v\cdot\xi)^2u^p\,d\sigma-\frac{1-\alpha+\beta}{a}\int_{\Sn}(v\cdot\xi)^2u^2\,d\sigma\right].
\end{equation}
\end{lemma}

\begin{proof}
Let $\Omega_+=\{H>0\}$ and $\Omega_-=\{H\leq0\}$, with masses $\rho_+=\rho$ and $\rho_-=1-\rho$ under $u^p d\sigma$. For $\rho_\pm>0$, define
\begin{equation*}
c_\pm=\frac{1}{\rho_\pm}\int_{\Omega_\pm}H_\pm u^p\,d\sigma.
\end{equation*}

If $c_\pm=0$, then $H_\pm=0$ almost everywhere on $\Omega_\pm$ and \eqref{eq:nodal-variance} is immediate for this sign. Otherwise, H\"older's inequality gives
\begin{equation*}
\int_{\Omega_\pm}H_\pm^2u^p\,d\sigma\leq \left(\int_{\Omega_\pm}H_\pm u^p\,d\sigma\right)^{\frac{p-2}{p-1}}\left(\int_{\Omega_\pm}H_\pm^pu^p\,d\sigma\right)^{\frac{1}{p-1}},
\end{equation*}
and hence
\begin{equation*}
\|H_\pm u\|_p^2\geq(\rho_\pm c_\pm)^{-4/n}\left(\int_{\Omega_\pm}H_\pm^2u^p\,d\sigma\right)^{1+2/n}=\rho_\pm^{1-2/n}c_\pm^2\left(1+\frac{\int_{\Omega_\pm}(H_\pm-c_\pm)^2u^p\,d\sigma}{\rho_\pm c_\pm^2}\right)^{1+2/n}.
\end{equation*}

Applying the elementary inequality
\begin{equation*}
(1+s)^{1+2/n}\geq1+\left(1+\frac{2}{n}\right)s, \qquad s\geq0,
\end{equation*}
gives
\begin{equation}\label{eq:nodal-variance}
\|H_\pm u\|_p^2\geq\rho_\pm^{1-2/n}c_\pm^2 +(1+2/n)\rho_\pm^{-2/n}\int_{\Omega_\pm}(H_\pm-c_\pm)^2u^p\,d\sigma.
\end{equation}

Since $\alpha(1+2/n)\rho_\pm^{-2/n}\geq1$ and
\begin{equation*}
\int_{\Sn} Hu^p\,d\sigma=\rho c_+-(1-\rho)c_-,
\end{equation*}
it follows from \eqref{eq:nodal-variance} that \eqref{eq:nodal-norm} holds provided
\begin{equation*}
\sum_{\pm}\rho_\pm(\rho_\pm^{-2/n}-1)c_\pm^2 \geq(2^{2/n}-1)\rho(1-\rho)(c_++c_-)^2.
\end{equation*}

For $0<\rho<1$, Cauchy--Schwarz inequality gives
\begin{equation*}
(c_++c_-)^2\leq\left[\sum_{\pm}\rho_\pm(\rho_\pm^{-2/n}-1)c_\pm^2\right]\left[\sum_{\pm}\frac{1}{\rho_\pm(\rho_\pm^{-2/n}-1)}\right].
\end{equation*}
Thus it suffices to prove
\begin{equation*}
\frac{(1-\rho)\rho^{2/n}}{1-\rho^{2/n}}+\frac{\rho(1-\rho)^{2/n}}{1-(1-\rho)^{2/n}}\leq\frac{1}{2^{2/n}-1}.
\end{equation*}

To establish this bound, we note that the function $s\mapsto (1-s)s^{2/n}/(1-s^{2/n})$ is concave on $(0,1)$. Evaluating at $\rho$ and $1-\rho$ gives
\begin{equation*}
\frac{(1-\rho)\rho^{2/n}}{1-\rho^{2/n}}+\frac{\rho(1-\rho)^{2/n}}{1-(1-\rho)^{2/n}}\leq2\frac{(1-\frac{1}{2})(\frac{1}{2})^{2/n}}{1-(\frac{1}{2})^{2/n}}=\frac{1}{2^{2/n}-1}.
\end{equation*}
The endpoint cases $\rho=0,1$ follow by continuity. This proves \eqref{eq:nodal-norm}.

The identity in \eqref{eq:translated-energy} follows from \eqref{eq:product-test} and \eqref{eq:first-moment}. We now apply \eqref{eq:nodal-norm} to $H=\tau+x$. By the classical Sobolev inequality, $a\leq1$ and the zero-barycenter condition, we have
\begin{equation*}
\frac{\alpha\|(\tau+x)u\|_2^2+\beta\|\nabla((\tau+x)u)\|_2^2}{a}\geq\alpha\bigl(\|(\tau+x)_+u\|_p^2+\|(\tau+x)_-u\|_p^2\bigr)\geq\alpha\tau^2+m.
\end{equation*}
This gives the inequality in \eqref{eq:translated-energy}.

Consequently, for every $\tau\in\R$, we have
\begin{equation*}
(1-\alpha)(1-r)\tau^2+2m(1+t)\frac{b\cdot e}{a}\tau+\mathcal R_N\geq0.
\end{equation*}
Taking the discriminant of this quadratic polynomial yields \eqref{eq:axial-discriminant}. The same argument with $x=v\cdot\xi$ gives the linear coefficient $2v\cdot(I+(p-1)M)b/(pa)$ and the corresponding discriminant condition yields \eqref{eq:directional-discriminant}.
\end{proof}

The next estimate applies the classical Sobolev inequality to the span of $u$ and $xu$, retaining the term involving $((b\cdot e)/a)^2$. The restriction $p\geq3$, equivalently $3\leq n\leq6$, ensures that the fourth distributional derivative of $|z|^p$ is nonnegative
\begin{equation*}
\frac{d^4}{dz^4}|z|^p=
\begin{cases}
p(p-1)(p-2)(p-3)|z|^{p-4},&\text{if }p>3,\\
12\delta_0,&\text{if }p=3.
\end{cases}
\end{equation*}
This allows us to bound the affine $L^p$ norm from below by matching the first three moments and applying cubic Hermite interpolation.

\begin{lemma}\label{common:affine-lemma}
If $p\geq3$, then
\begin{equation}\label{eq:affine-bound}
[1-2(1-\alpha)r]\mathcal R_N\geq(1-\alpha)rm+2m^2(1+t)^2\left(\frac{b\cdot e}{a}\right)^2.
\end{equation}
\end{lemma}

\begin{proof}
We first construct a probability measure supported at two points whose first three moments agree with those of $x/\sqrt m$ with respect to $u^p d\sigma$. Let $z_+>0>z_-$ be the roots of
\begin{equation*}
z^2-\frac{\int_{\Sn} x^3u^p\,d\sigma}{m^{3/2}}z-1=0,
\end{equation*}
and set
\begin{equation*}
\rho_+=\frac{-z_-}{z_+-z_-}, \qquad \rho_-=\frac{z_+}{z_+-z_-}.
\end{equation*}
These weights are positive and satisfy
\begin{equation*}
\sum_\pm\rho_\pm=1, \quad \sum_\pm\rho_\pm z_\pm=0,\quad \sum_\pm\rho_\pm z_\pm^2=1,\quad \sum_\pm\rho_\pm z_\pm^3=\frac{\int_{\Sn} x^3u^p\,d\sigma}{m^{3/2}}.
\end{equation*}
Consequently, for every polynomial $P$ of degree at most three, we have
\begin{equation*}
\int_{\Sn} P\left(\frac{x}{\sqrt m}\right)u^p\,d\sigma =\rho_+P(z_+)+\rho_-P(z_-).
\end{equation*}

For every affine function $F$, the cubic Hermite interpolation gives
\begin{equation}\label{eq:hermite-bound}
\int_{\Sn}\Bigg|F\left(\frac{x}{\sqrt m}\right)\Bigg|^pu^p\,d\sigma\geq\rho_+|F(z_+)|^p+\rho_-|F(z_-)|^p.
\end{equation}
Indeed, let $g_\epsilon$ be a mollification of $|F|^p$. Since the fourth distributional derivative of $|F|^p$ is nonnegative, we have $g_\epsilon^{(4)}\geq0$. By cubic Hermite interpolation, we can find a polynomial $P_\epsilon$ of degree at most three satisfying
\begin{equation*}
P_\epsilon(z_\pm)=g_\epsilon(z_\pm),\qquad P_\epsilon'(z_\pm)=g_\epsilon'(z_\pm)
\end{equation*}
with
\begin{equation*}
g_\epsilon(z)-P_\epsilon(z) =\frac{g_\epsilon^{(4)}(\eta)}{4!} (z-z_+)^2(z-z_-)^2\geq0,
\end{equation*}
where $\eta$ lies in the smallest interval containing $z,z_+$ and $z_-$. Therefore,
\begin{equation*}
\int_{\Sn} g_\epsilon\left(\frac{x}{\sqrt m}\right)u^p\,d\sigma\geq\int_{\Sn} P_\epsilon\left(\frac{x}{\sqrt m}\right)u^p\,d\sigma=\rho_+P_\epsilon(z_+)+\rho_-P_\epsilon(z_-)=\rho_+g_\epsilon(z_+)+\rho_-g_\epsilon(z_-).
\end{equation*}
Since $g_\epsilon\to |F|^p$ locally uniformly, passing to the limit yields \eqref{eq:hermite-bound}.

By \eqref{eq:translated-energy}, for any $\tau,\eta\in\R$, we have
\begin{equation*}
\begin{aligned}
\frac{1}{a}\left[ \alpha\left\|\left(\tau+\frac{\eta x}{\sqrt m}\right)u\right\|_2^2 +\beta\left\|\nabla\left(\left(\tau+\frac{\eta x}{\sqrt m}\right)u\right)\right\|_2^2 \right] =&[1-(1-\alpha)r]\tau^2 +2\sqrt m(1+t)\frac{b\cdot e}{a}\tau\eta +\left(1+\frac{\mathcal R_N}{m}\right)\eta^2\\
\geq&\alpha\left( \rho_+|\tau+\eta z_+|^p +\rho_-|\tau+\eta z_-|^p \right)^{2/p}\\
\geq&\alpha\rho_\pm^{2/p}(\tau+\eta z_\pm)^2,
\end{aligned}
\end{equation*}
where we use the ordinary Sobolev inequality, $a\leq1$ and \eqref{eq:hermite-bound} in the first inequality. Comparison of the two quadratic forms for all $\tau,\eta\in\R$ gives
\begin{equation*}
\begin{pmatrix}
1-(1-\alpha)r &\sqrt m(1+t)\frac{b\cdot e}{a}\\[2pt]
\sqrt m(1+t)\frac{b\cdot e}{a} &1+\frac{\mathcal R_N}{m}
\end{pmatrix}
\geq\alpha\rho_\pm^{2/p}
\begin{pmatrix}
1&z_\pm\\
z_\pm&z_\pm^2
\end{pmatrix}.
\end{equation*}

The left-hand matrix is positive definite, so the matrix inequality is equivalent to
\begin{equation*}
\frac{\rho_\pm\left[1+\frac{\mathcal R_N}{m} -2\sqrt m(1+t)\frac{b\cdot e}{a}z_\pm+[1-(1-\alpha)r]z_\pm^2 \right]}{[1-(1-\alpha)r]\left(1+\frac{\mathcal R_N}{m}\right) -m(1+t)^2\left(\frac{b\cdot e}{a}\right)^2} \leq\alpha^{-1}\rho_\pm^{2/n}.
\end{equation*}
Summing over the two signs and using the moment identities yields
\begin{equation*}
\frac{2-(1-\alpha)r+\frac{\mathcal R_N}{m}}{[1-(1-\alpha)r]\left(1+\frac{\mathcal R_N}{m}\right)-m(1+t)^2\left(\frac{b\cdot e}{a}\right)^2}\leq\alpha^{-1}(\rho_+^{2/n}+\rho_-^{2/n})\leq2,
\end{equation*}
where the last inequality follows from concavity of $s^{2/n}$ and $\rho_++\rho_-=1$. Then \eqref{eq:affine-bound} follows.
\end{proof}

\section{A weighted sum of the conformal variations}\label{new:weighted:section}

Let $e_1,\ldots,e_{n+1}$ be an orthonormal eigenbasis of $M$, with $Me_i=\lambda_i e_i$ and $\lambda_1\geq\dots\geq \lambda_{n+1}$. Write $x_i=e_i\cdot\xi$. For every unit vector $v$, we have
\begin{equation*}
v^TMv=\int_{\Sn}(v\cdot\xi)^2u^p\,d\sigma,\qquad \int_{\Sn}|\xi-(v\cdot\xi)v|^2u^p\,d\sigma =1-v^TMv\geq1-\lambda_1(M).
\end{equation*}
Thus an upper bound for $\lambda_1(M)$ gives a uniform lower bound for the mean squared distance from every axis $\mathbb Rv$. Since $\operatorname{tr}M=1$, we have $\lambda_1(M)\geq\frac{1}{n+1}$, with equality precisely when $M=\frac{I}{n+1}$. In this case, the second moments are equal in all directions and $\Phi_{e_i}=\nabla x_i\cdot\nabla u$. Summing \eqref{eq:conformal-second-variation} over $i$, we obtain
\begin{equation}\label{eq:isotropic-trace}
\frac{1}{a}\sum_{i=1}^{n+1}Q(\Phi_{e_i},\Phi_{e_i})=-(n-2)(1-r)<0.
\end{equation}
For unequal eigenvalues, we choose the weights so that the terms involving the second moments of $u^2$ can again be summed explicitly. We use \eqref{eq:joint-moments} to estimate the resulting third moment. The following proposition extends \eqref{eq:isotropic-trace} to the stated ranges of $\lambda_1(M)$ without an additional restriction on $r$.

\begin{proposition}\label{new:weighted:proposition}
Let $n\geq3$ and let $u$ be a nonconstant solution of \eqref{eq:critical-equation}. If
\begin{equation}\label{eq:small-eigenvalue}
\lambda_1(M)\leq
\begin{cases}
\frac{3}{10}, & \text{if }n=3,\\[2mm]
1-\frac{1}{\sqrt{2}}, & \text{if }n=4,\\[2mm]
\frac{1+\frac{3}{n}}{6}, & \text{if }n\geq5,
\end{cases}
\end{equation}
then $Q(\Phi_{e_i},\Phi_{e_i})<0$ for some $i$.
\end{proposition}

\begin{proof}
We write $t_i=\frac{1-\lambda_i}{p\lambda_i}$. To eliminate the term $-(p-2)t_i^2\lambda_i$ in \eqref{eq:conformal-second-variation}, we use \eqref{eq:gradient-second-moment} to obtain
\begin{equation*}
\begin{aligned}
&\frac{Q(\Phi_{e_i},\Phi_{e_i})}{a}+\frac{\beta(p-2)t_i^2}{a}\int_{\Sn} x_i^2|\nabla u|^2\,d\sigma\\
=&[(p-1)\beta t_i^2-(1-\alpha)]r-\frac{4+(5n+2)\beta}{(n-2)a}\left[t_i^2-\frac{(n-2)(1-\alpha)}{4+(5n+2)\beta}(3+4t_i)\right]\int_{\Sn} x_i^2u^2\,d\sigma -\frac{2}{pa\lambda_i}\int_{\Sn}(b\cdot\xi)x_i^2u^p\,d\sigma.
\end{aligned}
\end{equation*}
Here we use $1+pt_i=\frac{1}{\lambda_i}$. We will verify below that
\begin{equation*}
(p-2)\left[t_i^2-\frac{(n-2)(1-\alpha)}{4+(5n+2)\beta}(3+4t_i)\right]>0
\end{equation*}
under \eqref{eq:small-eigenvalue}. Weighting the identity by the reciprocal of this positive expression makes the coefficient of $\int_{\Sn} x_i^2u^2\,d\sigma$ equal to $-\frac{4+(5n+2)\beta}{4a}$ for every $i$. To sum the remaining third moments, we define
\begin{equation}\label{eq:weights}
F(\xi)=\frac{2}{p(p-2)}\sum_i\frac{x_i^2}{\lambda_i\left[t_i^2-\frac{(n-2)(1-\alpha)}{4+(5n+2)\beta}(3+4t_i)\right]}.
\end{equation}
Summing the resulting identities over $i$, we obtain
\begin{equation}\label{eq:weighted-trace}
\begin{aligned}
\sum_i\frac{Q(\Phi_{e_i},\Phi_{e_i})}{a(p-2)\left[t_i^2-\frac{(n-2)(1-\alpha)}{4+(5n+2)\beta}(3+4t_i)\right]} &=-\frac{\beta}{a}\int_{\Sn}\left(\sum_i\frac{t_i^2x_i^2}{t_i^2-\frac{(n-2)(1-\alpha)}{4+(5n+2)\beta}(3+4t_i)}\right)|\nabla u|^2\,d\sigma\\
&\quad-\frac{(1-\alpha)r}{n+1} \sum_i\frac{\left[t_i-\frac{n-2}{2}\right]^2}{t_i^2-\frac{(n-2)(1-\alpha)}{4+(5n+2)\beta}(3+4t_i)} -\frac{1}{a}\int_{\Sn}(b\cdot\xi)Fu^p\,d\sigma.
\end{aligned}
\end{equation}
Here we use $\sum_i\lambda_i=\sum_i x_i^2=1$, $\int_{\Sn}\xi u^p\,d\sigma=0$, $n\beta=(p-2)\alpha$ and
\begin{equation*}
(p-2)t_i^2-4t_i+n-2=(p-2)\left(t_i-\frac{n-2}{2}\right)^2.
\end{equation*}

In the eigenbasis of $M$, \eqref{eq:gradient-first-moment} gives
\begin{equation*}
\frac{1}{a}\int_{\Sn}(b\cdot\xi)Fu^p\,d\sigma=\frac{\beta}{a}\int_{\Sn}|\nabla u|^2\sum_i\frac{x_i\int_{\Sn}(e_i\cdot\zeta)F(\zeta)u^p(\zeta)\,d\sigma(\zeta)}{\lambda_i+ \left[\frac{1}{2(1-\alpha)}-\frac{1}{n}\right](1-\lambda_i)}\,d\sigma(\xi).
\end{equation*}
Therefore, to show $Q(\Phi_{e_i},\Phi_{e_i})<0$ for some $i$, it suffices to prove for every $\xi\in\Sn$,
\begin{equation}\label{eq:trace-comparison}
\left|\sum_i\frac{x_i\int_{\Sn} x_iFu^p\,d\sigma}{\lambda_i+ \left[\frac{1}{2(1-\alpha)}-\frac{1}{n}\right](1-\lambda_i)}\right|< \sum_i\frac{t_i^2x_i^2}{t_i^2-\frac{(n-2)(1-\alpha)}{4+(5n+2)\beta}(3+4t_i)}.
\end{equation}

Since $F(e_i)>0$ and $\sum_i x_i^2=1$, we have $\|F\|_\infty=\max_iF(e_i)$. The joint moment estimate \eqref{eq:joint-moments} gives
\begin{equation}\label{eq:weighted-variance}
\sum_i \frac{\left(\int_{\Sn} x_iFu^p\,d\sigma\right)^2}{\lambda_i(1-\lambda_i)}\leq \sum_i\lambda_iF(e_i)^2-\left(\sum_i\lambda_iF(e_i)\right)^2\leq\frac{\|F\|_\infty^2}{4}.
\end{equation}
The variance in \eqref{eq:weighted-variance} vanishes when all eigenvalues agree. By \eqref{eq:weighted-variance} and the Cauchy--Schwarz inequality, \eqref{eq:trace-comparison} follows provided that, for every $i$,
\begin{equation}\label{eq:weight-bound}
\frac{\|F\|_\infty^2}{4}\frac{\lambda_i(1-\lambda_i)}{\left[\lambda_i+\left(\frac{1}{2(1-\alpha)}-\frac{1}{n}\right)(1-\lambda_i)\right]^2} <\frac{2t_i^2}{t_i^2-\frac{(n-2)(1-\alpha)}{4+(5n+2)\beta}(3+4t_i)}-1.
\end{equation}
Indeed, \eqref{eq:weighted-variance} and \eqref{eq:weight-bound} give
\begin{equation*}
\begin{aligned}
&\left|\sum_i\frac{x_i\int_{\Sn} x_iFu^p\,d\sigma}{\lambda_i+\left[\frac{1}{2(1-\alpha)}-\frac{1}{n}\right](1-\lambda_i)}\right|^2 \leq\frac{\|F\|_\infty^2}{4}\sum_i\frac{\lambda_i(1-\lambda_i)x_i^2}{\left[\lambda_i+\left(\frac{1}{2(1-\alpha)}-\frac{1}{n}\right)(1-\lambda_i)\right]^2}\\
<&2\sum_i\frac{t_i^2x_i^2}{t_i^2-\frac{(n-2)(1-\alpha)}{4+(5n+2)\beta}(3+4t_i)}-1 \leq\left(\sum_i\frac{t_i^2x_i^2}{t_i^2-\frac{(n-2)(1-\alpha)}{4+(5n+2)\beta}(3+4t_i)}\right)^2,
\end{aligned}
\end{equation*}
which proves \eqref{eq:trace-comparison}.

It remains to verify the positivity of the denominators in \eqref{eq:weights} and inequality \eqref{eq:weight-bound} under \eqref{eq:small-eigenvalue}. For each $i$, we write $\tau=\frac{\lambda_i}{1-\lambda_i}=\frac{1}{pt_i}$. In terms of $\tau$ these quantities are
\begin{equation}\label{eq:weights-in-tau}
\begin{aligned}
\frac{t_i^2}{t_i^2-\frac{(n-2)(1-\alpha)}{4+(5n+2)\beta}(3+4t_i)}&=\frac{1}{1-\frac{8n(1-\alpha)}{4+(5n+2)\beta}\tau\left(1+\frac{3p}{4}\tau\right)},\\
F(e_i)&=\frac{n\tau(1+\tau)}{1-\frac{8n(1-\alpha)}{4+(5n+2)\beta}\tau\left(1+\frac{3p}{4}\tau\right)},\\
\frac{\lambda_i(1-\lambda_i)}{\left[\lambda_i+ \left(\frac{1}{2(1-\alpha)}-\frac{1}{n}\right)(1-\lambda_i) \right]^2}&=\frac{\tau}{\left[\frac{1}{2(1-\alpha)}-\frac{1}{n}+\tau\right]^2}.
\end{aligned}
\end{equation}

If $n=3$, then $p=6$, $\beta=\frac{4\alpha}{3}$ and $0<\tau\leq\frac{3}{7}$. The coefficients in \eqref{eq:weights-in-tau} satisfy
\begin{equation*}
\frac{1}{3}<\frac{8n(1-\alpha)}{4+(5n+2)\beta} =\frac{18(1-\alpha)}{3+17\alpha}<\frac{1}{2}, \qquad \frac{1}{2(1-\alpha)}-\frac{1}{3}>1.
\end{equation*}
In particular,
\begin{equation*}
1-\frac{18(1-\alpha)}{3+17\alpha}\tau\left(1+\frac{9}{2}\tau\right)>\frac{73}{196}.
\end{equation*}
Thus the denominators in \eqref{eq:weights} are positive and \eqref{eq:weights-in-tau} gives $\|F\|_\infty<5$.

We also have $1-\frac{\tau}{3}-\frac{3\tau^2}{2}>0$ and
\begin{equation*}
\left(1+\frac{\tau}{3}+\frac{3}{2}\tau^2\right)(1+\tau)^2-\frac{25}{4}\tau\left(1-\frac{\tau}{3}-\frac{3}{2}\tau^2\right) =1-\frac{47}{12}\tau+\frac{21}{4}\tau^2+\frac{305}{24}\tau^3+\frac{3}{2}\tau^4 \geq(1-2\tau)^2+\tau^2>0.
\end{equation*}
Consequently, we have
\begin{equation*}
\frac{\|F\|_\infty^2}{4}\frac{\tau}{\left[\frac{1}{2(1-\alpha)}-\frac{1}{3}+\tau\right]^2} <\frac{25\tau}{4(1+\tau)^2}<\frac{2}{1-\frac{\tau}{3}-\frac{3\tau^2}{2}}-1 <\frac{2}{1-\frac{18(1-\alpha)}{3+17\alpha}\tau\left(1+\frac{9}{2}\tau\right)}-1.
\end{equation*}
Together with \eqref{eq:weights-in-tau}, this proves \eqref{eq:weight-bound} for $n=3$.

If $n=4$, then $p=4$ and
\begin{equation*}
0<\tau\leq\sqrt{2}-1<\frac{5}{12},\qquad \frac{3}{4}<\frac{8n(1-\alpha)}{4+(5n+2)\beta}<\frac{4}{5},
\end{equation*}
while
\begin{equation*}
\frac{1}{2(1-\alpha)}-\frac{1}{n}=\frac{3+2\sqrt2}{4}>\frac{7}{5}.
\end{equation*}
In particular, we have
\begin{equation*}
1-\frac{64-32\sqrt2}{8+11\sqrt2}\tau(1+3\tau)>\frac{1}{4}.
\end{equation*}
Thus the denominators in \eqref{eq:weights} are positive. The expression for $F(e_i)$ in \eqref{eq:weights-in-tau} increases with $\tau$ and is less than $9$ at $\tau=\sqrt{2}-1$. Hence $\|F\|_\infty<9$.

We have $1-\frac{3}{4}\tau-\frac{9}{4}\tau^2>0$ and
\begin{equation*}
\left(1+\frac{3}{4}\tau+\frac{9}{4}\tau^2\right)\left(\frac{7}{5}+\tau\right)^2-\frac{81}{4}\tau\left(1-\frac{3}{4}\tau-\frac{9}{4}\tau^2\right) \geq\frac{49}{25}-16\tau+22\tau^2+50\tau^3>0.
\end{equation*}
Consequently,
\begin{equation}\label{eq:four-dimensional-weights}
\frac{\|F\|_\infty^2}{4}\frac{\tau}{\left[\frac{3+2\sqrt2}{4}+\tau\right]^2}< \frac{81\tau}{4\left(\frac{7}{5}+\tau\right)^2}< \frac{2}{1-\frac{3}{4}\tau-\frac{9}{4}\tau^2}-1<\frac{2}{1-\frac{64-32\sqrt2}{8+11\sqrt2}\tau(1+3\tau)}-1.
\end{equation}
Together with \eqref{eq:weights-in-tau}, this proves \eqref{eq:weight-bound} for $n=4$.

If $n\geq5$, then
\begin{equation*}
0<\tau\leq\frac{n+3}{5n-3}\leq\frac{4}{11}.
\end{equation*}
Direct computation gives
\begin{equation}\label{eq:dimension-bounds}
\begin{aligned}
&\frac{8n(1-\alpha)}{4+(5n+2)\beta}>1, \qquad \frac{1}{2(1-\alpha)}-\frac{1}{n}>\frac{n}{2\sqrt2},\\
&\frac{8n(1-\alpha)}{4+(5n+2)\beta}\frac{n+3}{5n-3}\left(1+\frac{3p(n+3)}{4(5n-3)}\right)<\frac{3}{4}.
\end{aligned}
\end{equation}

By \eqref{eq:weights-in-tau} and \eqref{eq:dimension-bounds}, the denominators in \eqref{eq:weights} are positive and
\begin{equation}\label{eq:maximum-weight}
\|F\|_\infty<\frac{240n}{121}<2n.
\end{equation}
On this interval $1-\tau-\frac{3}{2}\tau^2>0$ and
\begin{equation*}
2-(1+8\tau)\left(1-\tau-\frac{3}{2}\tau^2\right)>0.
\end{equation*}

Equations \eqref{eq:dimension-bounds} and \eqref{eq:maximum-weight} now give
\begin{equation*}
\frac{\|F\|_\infty^2}{4}\frac{\tau}{\left[\frac{1}{2(1-\alpha)}-\frac{1}{n}+\tau\right]^2} <\frac{n^2\tau}{\left[\frac{1}{2(1-\alpha)}-\frac{1}{n}\right]^2} <8\tau<\frac{2}{1-\tau-\frac{3}{2}\tau^2}-1 <\frac{2}{1-\frac{8n(1-\alpha)}{4+(5n+2)\beta}\tau\left(1+\frac{3p}{4}\tau\right)}-1.
\end{equation*}
Together with \eqref{eq:weights-in-tau}, this proves \eqref{eq:weight-bound} for $n\geq5$.

Thus \eqref{eq:trace-comparison} holds in all three cases. Since $\frac{\beta}{a}\int_{\Sn}|\nabla u|^2\,d\sigma=1-r>0$, the right-hand side of \eqref{eq:weighted-trace} is strictly negative. All weights on its left-hand side are positive, so $Q(\Phi_{e_i},\Phi_{e_i})<0$ for some $i$.
\end{proof}

\section{Instability in dimensions four to six}\label{sec:new-four-to-six}\label{sec:new-four}

Throughout this section, $4\leq n\leq6$ and $u$ is a nonconstant solution of \eqref{eq:critical-equation}. write $m=\lambda_1(M)$ and fix a unit eigenvector $e$ satisfying $Me=me$. Recall that $x=e\cdot\xi$ and $t=\frac{1-m}{pm}$.

We estimate $\mathcal R_N$ defined in \eqref{eq:nodal-remainder} together with $\int_{\Sn}(b\cdot\xi)x^2u^p\,d\sigma$ in \eqref{eq:principal-variation}. If $n=4$, we use the Cauchy--Schwarz inequality in $L^2(u^4\,d\sigma)$. For $n=5,6$, we combine the affine estimate \eqref{eq:affine-bound} with bounds on $b\cdot e$ and $b-(b\cdot e)e$. The following proposition and Proposition~\ref{new:weighted:proposition} prove Theorem~\ref{thm:stable-critical-point-classification} in these dimensions.

\begin{proposition}\label{new:four-principal}\label{new:five:proposition}
If $n=4$ and $m\geq1-\alpha$, or if $n\in\{5,6\}$ and $m\geq\frac{1}{4}$, then $Q(\Phi_e,\Phi_e)<0$.
\end{proposition}

\subsection{Dimension four}
Here $p=4$, $\alpha=\frac{1}{\sqrt{2}}$ and $\beta=\frac{\alpha}{2}$. Equation \eqref{eq:principal-variation} requires a lower bound for
\begin{equation*}
\mathcal R_N-\frac{c}{a}\int_{\mathbb{S}^4}(b\cdot\xi)x^2u^4\,d\sigma,\qquad c=\frac{C_m}{\Gamma(1-\alpha+\beta)}
\end{equation*}
when $\Gamma=\frac{4m[(1-\alpha)(3+4t)-\beta t^2]}{1-\alpha+\beta}>0$. For every affine function $F$, we have
\begin{equation*}
\|Fu\|_4^2=\|F^2\|_{L^2(u^4\,d\sigma)},\qquad \|b\cdot\xi\|_{L^2(u^4\,d\sigma)}^2=b^TMb.
\end{equation*}

We first estimate $\frac{b^TMb}{a^2}$.
\begin{lemma}\label{new:four-multiplier}
We have
\begin{equation}\label{eq:multiplier-four}
\frac{\sqrt{b^TMb}}{a}\leq1-\alpha,\qquad \frac{b^TMb}{a^2}\leq mt^2+\sqrt{2}(1-m)(1-\alpha)r.
\end{equation}
\end{lemma}

\begin{proof}
Let $V$ be an $M$-invariant linear subspace with $\operatorname{proj}_Vb\ne0$ and set $v=\frac{\operatorname{proj}_Vb}{|\operatorname{proj}_Vb|}$. Then $b\cdot v=|\operatorname{proj}_Vb|$ and $v^TMb=(v^TMv)(b\cdot v)$. By \eqref{eq:first-moment} and \eqref{eq:gradient-first-moment}, we have
\begin{equation*}
\frac{1}{a}\int_{\mathbb{S}^4}(v\cdot\xi)u^2\,d\sigma=-\frac{1-v^TMv}{4(1-\alpha)}\frac{b\cdot v}{a}, \qquad \frac{\beta}{a}\int_{\mathbb{S}^4}(v\cdot\xi)|\nabla u|^2\,d\sigma=\left[v^TMv+\frac{1+2\beta}{4(1-\alpha)}(1-v^TMv)\right]\frac{b\cdot v}{a}.
\end{equation*}
Using $\int_{\mathbb{S}^4} u\nabla(v\cdot\xi)\cdot\nabla u\,d\sigma=2\int_{\mathbb{S}^4}(v\cdot\xi)u^2\,d\sigma$ and $|\nabla(v\cdot\xi)|^2=1-(v\cdot\xi)^2$, the Cauchy--Schwarz inequality gives
\begin{equation}\label{eq:weighted-gradient-test}
\begin{aligned}
\frac{\beta(1-v^TMv)^2}{4(1-\alpha)^2}\left(\frac{b\cdot v}{a}\right)^2 &\leq\left[r-\frac{1-v^TMv}{4(1-\alpha)}\frac{b\cdot v}{a}\right] \cdot\left[1-r-\left\{v^TMv+\frac{1+2\beta}{4(1-\alpha)}(1-v^TMv)\right\}\frac{b\cdot v}{a}\right]\\
&\leq r-\frac{1-v^TMv}{4(1-\alpha)}\frac{b\cdot v}{a},
\end{aligned}
\end{equation}
Consequently, we have
\begin{equation*}
r\geq\frac{1-v^TMv}{4(1-\alpha)}\frac{b\cdot v}{a}+\frac{\beta(1-v^TMv)^2}{4(1-\alpha)^2}\left(\frac{b\cdot v}{a}\right)^2.
\end{equation*}

Apply \eqref{eq:sobolev} to $u\sqrt{1\pm v\cdot\xi}$. These functions belong to $H^1(\mathbb{S}^4)$, so \eqref{eq:product-test} applies by approximation. Together with \eqref{eq:first-moment}, it gives
\begin{equation*}
1\pm\frac{1+v^TMv}{2}\frac{b\cdot v}{a} -\frac{1-\alpha-\frac{\beta}{4}}{a}\int_{\mathbb{S}^4}(1\mp v\cdot\xi)u^2\,d\sigma \geq\alpha\sqrt{1+v^TMv}.
\end{equation*}
Since $1-\alpha>\frac{\beta}{4}$, it follows that
\begin{equation}\label{eq:square-root-test}
\frac{b\cdot v}{a}\leq\frac{2[1-\alpha\sqrt{1+v^TMv}]}{1+v^TMv}=\frac{1-v^TMv}{(1+v^TMv)(1+\sqrt{\frac{1+v^TMv}{2}})}.
\end{equation}
Taking $V=\mathbb R^5$ in \eqref{eq:square-root-test} and using $2\sqrt{v^TMv}\leq1+v^TMv$, we obtain
\begin{equation*}
\frac{\sqrt{b^TMb}}{a}\leq 1-\alpha.
\end{equation*}

For $V=\mathbb Re$, we use
\begin{equation*}
(1+m)(1+\sqrt{\frac{1+m}{2}})\geq\frac{(1+m)(3+m)}{2}\geq4m
\end{equation*}
and \eqref{eq:square-root-test} to obtain $\frac{|b\cdot e|}{a}\leq t$.

Taking $V=e^\perp$, we have $b\cdot v=|b-(b\cdot e)e|$. By \eqref{eq:square-root-test},
\begin{equation*}
1-v^TMv\geq(1+v^TMv) \left(1+\sqrt{\frac{1+v^TMv}{2}}\right)\frac{b\cdot v}{a}.
\end{equation*}
Combining this with \eqref{eq:weighted-gradient-test}, we obtain
\begin{equation*}
\frac{|b-(b\cdot e)e|^2}{a^2(1-\alpha)r} \leq\frac{4}{(1+v^TMv)(1+\sqrt{\frac{1+v^TMv}{2}}) +\frac{\beta}{1-\alpha}(1-v^TMv)^2}.
\end{equation*}

Since
\begin{equation*}
\sqrt{\frac{1+v^TMv}{2}}\geq\alpha+(1-\alpha)v^TMv,
\end{equation*}
we get
\begin{equation*}
\begin{aligned}
&(1+v^TMv)\left(1+\sqrt{\frac{1+v^TMv}{2}}\right) +\frac{\beta}{1-\alpha}(1-v^TMv)^2\\
\geq&(1+v^TMv)[1+\alpha+(1-\alpha)v^TMv] +\frac{\beta}{1-\alpha}(1-v^TMv)^2\\
=&\frac{3}{2}+\sqrt{2}+(1-\sqrt{2})v^TMv+\frac{3}{2}(v^TMv)^2>2\sqrt{2}.
\end{aligned}
\end{equation*}
Consequently,
\begin{equation*}
\frac{|b-(b\cdot e)e|^2}{a^2(1-\alpha)r}<\sqrt{2}.
\end{equation*}
If $\operatorname{proj}_V b=0$, its bound is immediate. Since $Me=me$ and the eigenvalues of $M$ on $e^\perp$ are at most $1-m$, we have
\begin{equation*}
b^TMb\leq m(b\cdot e)^2+(1-m)|b-(b\cdot e)e|^2.
\end{equation*}
The bounds for $|b\cdot e|$ and $|b-(b\cdot e)e|$ now give
\begin{equation*}
\frac{b^TMb}{a^2} \leq mt^2+\sqrt{2}(1-m)(1-\alpha)r,
\end{equation*}
which proves the second inequality in \eqref{eq:multiplier-four}.
\end{proof}

The following lemma gives a lower bound for $\mathcal R_N-\frac{c}{a}\int_{\mathbb{S}^4}(b\cdot\xi)x^2u^4\,d\sigma$ needed in \eqref{eq:principal-variation}.

\begin{lemma}\label{new:four-dual}
Let $c\in\mathbb R$ satisfy $\frac{c^2b^TMb}{a^2}<\alpha^2$. For every $z>0$, we have
\begin{equation}\label{eq:four-dimensional-dual}
\frac{\mathcal R_N}{m}-\frac{c}{am}\int_{\mathbb{S}^4}(b\cdot\xi)x^2u^4\,d\sigma-z(1-\alpha)r \geq\sqrt{\alpha^2-\frac{c^2b^TMb}{a^2}}\sqrt{(1+z)^2+4z}-(1+z).
\end{equation}
\end{lemma}

\begin{proof}
Fix $0<s<\sqrt{\alpha^2-\frac{c^2b^TMb}{a^2}}$. We choose $\lambda_\pm$ such that
\begin{equation*}
\left\|s+\frac{c}{a}(b\cdot\xi)+\lambda_\pm x \right\|_{L^2(u^4d\sigma)}=\alpha.
\end{equation*}
Since $\int_{\mathbb{S}^4}\xi u^4\,d\sigma=0$ and $Me=me$, we have
\begin{equation*}
\int_{\mathbb{S}^4}\left(s+\frac{c}{a}(b\cdot\xi)+\lambda x\right)^2u^4\,d\sigma =s^2+\frac{c^2b^TMb}{a^2} +\frac{2cm(b\cdot e)}{a}\lambda+m\lambda^2.
\end{equation*}

Take $\lambda_-<0<\lambda_+$ to be the roots of
\begin{equation*}
m\lambda^2+\frac{2cm(b\cdot e)}{a}\lambda +s^2+\frac{c^2b^TMb}{a^2}-\alpha^2=0.
\end{equation*}
It follows that $\|s+\frac{c}{a}(b\cdot\xi)+\lambda_\pm x\|_{L^2(u^4d\sigma)}=\alpha$.

Set
\begin{equation*}
\rho_+=\frac{-\lambda_-}{\lambda_+-\lambda_-},\qquad \rho_-=\frac{\lambda_+}{\lambda_+-\lambda_-}.
\end{equation*}
Then we have $\rho_++\rho_-=1$, $\sum_\pm\rho_\pm\lambda_\pm=0$ and
\begin{equation*}
\sum_\pm\rho_\pm\lambda_\pm^2=-\lambda_-\lambda_+ =\frac{\alpha^2-s^2-\frac{c^2b^TMb}{a^2}}{m}.
\end{equation*}

Next, we choose
\begin{equation*}
\tau_\pm=\frac{m\sqrt{z}\,\lambda_\pm}{\sqrt{\alpha^2-s^2-\frac{c^2b^TMb}{a^2}}}.
\end{equation*}
so that $\sum_\pm\rho_\pm\tau_\pm=0$ and $\sum_\pm\rho_\pm\tau_\pm^2=mz$.

By the classical Sobolev inequality \eqref{eq:sobolev}, $a\leq1$ and the Cauchy--Schwarz inequality, we have
\begin{equation*}
\begin{aligned}
&\frac{1}{a}\int_{\mathbb{S}^4}\left[\alpha(x+\tau_\pm)^2u^2 +\beta|\nabla((x+\tau_\pm)u)|^2\right]\,d\sigma \geq\alpha\left(\int_{\mathbb{S}^4}(x+\tau_\pm)^4u^4\,d\sigma\right)^{1/2}\\
\geq&\int_{\mathbb{S}^4}\left[s+\frac{c}{a}(b\cdot\xi)+\lambda_\pm x\right] (x+\tau_\pm)^2u^4\,d\sigma.
\end{aligned}
\end{equation*}

The identities for the weights $\rho_\pm$ give
\begin{equation*}
m+\mathcal R_N+[1-(1-\alpha)r]mz\geq sm(1+z)+\frac{c}{a}\int_{\mathbb{S}^4}(b\cdot\xi)x^2u^4\,d\sigma+2m\sqrt{z}\sqrt{\alpha^2-s^2-\frac{c^2b^TMb}{a^2}}.
\end{equation*}
Since
\begin{equation*}
\max_{0\leq s\leq\sqrt{\alpha^2-\frac{c^2b^TMb}{a^2}}}\left[sm(1+z)+2m\sqrt{z}\sqrt{\alpha^2-s^2-\frac{c^2b^TMb}{a^2}}\right] =m\sqrt{(1+z)^2+4z}\sqrt{\alpha^2-\frac{c^2b^TMb}{a^2}},
\end{equation*}
we obtain \eqref{eq:four-dimensional-dual}.
\end{proof}

\begin{proof}[Proof of Proposition~\ref{new:four-principal} for $n=4$]
The assumption $m\geq1-\alpha$ is equivalent to $0<t\leq\frac{1+\sqrt{2}}{4}$. By \eqref{eq:gamma}, we have
\begin{equation*}
\frac{\Gamma}{4m}=\frac{6+8t-(1+\sqrt{2})t^2}{3+\sqrt{2}}\geq\frac{6}{3+\sqrt{2}}>\frac{4}{3}, \qquad \frac{(1+2t)^2}{2}<\frac{8}{3}.
\end{equation*}

In Lemma~\ref{new:four-dual}, we take
\begin{equation*}
c=\frac{C_m}{\Gamma(1-\alpha+\beta)} =1-\frac{2m(1+2t)^2}{\Gamma}, \qquad z=1+2t+3t^2.
\end{equation*}
The preceding bounds imply $|c|<1$ and \eqref{eq:multiplier-four} shows that
\begin{equation*}
\frac{c^2b^TMb}{a^2}<(1-\alpha)^2<\alpha^2.
\end{equation*}

Substituting \eqref{eq:four-dimensional-dual} into \eqref{eq:principal-variation} gives
\begin{equation*}
\frac{4}{a\Gamma}Q(\Phi_e,\Phi_e)\leq1+z-\frac{8mt^2}{\Gamma}+\left[\frac{8(1+2t)}{\Gamma}-\frac{1}{m}-z\right](1-\alpha)r-\sqrt{(1+z)^2+4z}\sqrt{\frac{1}{2}-\frac{c^2b^TMb}{a^2}}.
\end{equation*}
By \eqref{eq:multiplier-four} and $1-m=4mt$, it suffices to prove
\begin{equation}\label{eq:quartic-comparison}
[(1+z)^2+4z]\left[\frac{1}{2}-c^2m\{t^2+4\sqrt{2}t(1-\alpha)r\}\right]-\left[1+z-\frac{8mt^2}{\Gamma} +\left\{\frac{8(1+2t)}{\Gamma}-\frac{1}{m}-z\right\}(1-\alpha)r\right]^2>0.
\end{equation}

For $z=1+2t+3t^2$, \eqref{eq:quartic-comparison} is equivalent to $P(t,r)>0$, where
\begin{equation}\label{eq:quartic-polynomial}
\begin{aligned}
&P(t,r)= \frac{1+4t}{2}\left(3+4t-\frac{1+\sqrt{2}}{2}t^2\right)^2(8+16t+28t^2+12t^3+9t^4)\\
&-(1+4t)\Biggl\{\left(3+4t-\frac{1+\sqrt{2}}{2}t^2\right) \left[2+2t+3t^2-(1-\alpha)r(2+6t+3t^2)\right]\\
&\qquad+(3+\sqrt{2})\left[(1-\alpha)r(1+4t)(1+2t)-t^2\right]\Biggr\}^2\\
&-(8+16t+28t^2+12t^3+9t^4)\left[\frac{9-\sqrt{2}}{4}+(1-\sqrt{2})t-\frac{7+3\sqrt{2}}{2}t^2\right]^2[t^2+4(\sqrt{2}-1)tr].
\end{aligned}
\end{equation}
Since $P(t,r)$ is concave in $r$, it suffices to show $P(t,0),P(t,1)>0$. By direct calculation, we have for $0<t<1$,
\begin{equation*}
3\frac{P(t,0)}{t^2}-t\frac{d}{dt}\left(\frac{P(t,0)}{t^2}\right)>0,\qquad 5P(t,1)-t\frac{d P(t,1)}{dt}>0.
\end{equation*}
Hence, $\frac{P(t,0)}{t^5}$ and $\frac{P(t,1)}{t^5}$ are decreasing on $(0,1)$. At $\frac{1+\sqrt{2}}{4}$, we have
\begin{equation*}
\frac{P(\frac{1+\sqrt{2}}{4},0)}{(\frac{1+\sqrt{2}}{4})^2}>\frac{4}{5}, \qquad P(\frac{1+\sqrt{2}}{4},1)>\frac{2}{5}.
\end{equation*}
Thus $P(t,0),P(t,1)>0$ for $0<t\leq\frac{1+\sqrt{2}}{4}$. Concavity yields \eqref{eq:quartic-comparison} for $0<r<1$ and proves $Q(\Phi_e,\Phi_e)<0$.
\end{proof}

\subsection{Dimensions five and six}

\begin{proof}[Proof of Proposition~\ref{new:five:proposition} for $n=5,6$]
We set $b_\perp=b-(b\cdot e)e$. In these cases, $3\leq p<4$, $0<t\leq\frac{3}{p}$ and $\beta<1-\alpha$. Equation \eqref{eq:principal-variation} gives
\begin{equation}\label{eq:axial-variation}
\frac{Q(\Phi_e,\Phi_e)}{a}=-(p-2)mt^2+\left(2+4t-\frac{\Gamma}{nm}\right)(1-\alpha)r-\frac{\Gamma}{nm}\mathcal R_N +\frac{C_m}{anm(1-\alpha+\beta)}\int_{\Sn}(b\cdot\xi)x^2u^p\,d\sigma.
\end{equation}

Suppose on the contrary that $Q(\Phi_e,\Phi_e)\geq0$. We will first show that $C_m>0$ and prove
\begin{equation*}
\int_{\Sn}(b\cdot\xi)x^2u^p\,d\sigma>m(1-m)|b|,
\end{equation*}
which contradicts \eqref{eq:axial-third-moment}.

For any $s>0$, \eqref{eq:affine-bound} and the arithmetic--geometric mean inequality give
\begin{equation}\label{eq:affine-amgm}
\frac{(p-2)mt^2+s\mathcal R_N}{(1-\alpha)r} \geq2(p-2)mt^2+sm +\frac{1-m}{p}\sqrt{8(p-2)s}\left[1+m(1+t)^2\left(\frac{b\cdot e}{a(1-\alpha)r}\right)^2\right]^{\frac{1}{2}}.
\end{equation}
Indeed, \eqref{eq:affine-bound} gives
\begin{equation*}
\frac{\mathcal R_N}{(1-\alpha)r}\geq m+\frac{2m(1-\alpha)r}{1-2(1-\alpha)r}\left[1+m(1+t)^2 \left(\frac{b\cdot e}{a(1-\alpha)r}\right)^2\right].
\end{equation*}
Since
\begin{equation*}
\frac{1}{(1-\alpha)r} =2+\frac{1-2(1-\alpha)r}{(1-\alpha)r},
\end{equation*}
the arithmetic--geometric mean inequality gives
\begin{equation*}
\begin{aligned}
&\frac{(p-2)mt^2+s\mathcal R_N}{(1-\alpha)r} \geq2(p-2)mt^2+sm+(p-2)mt^2\frac{1-2(1-\alpha)r}{(1-\alpha)r}+\frac{2sm(1-\alpha)r}{1-2(1-\alpha)r} \left[1+m(1+t)^2 \left(\frac{b\cdot e}{a(1-\alpha)r}\right)^2\right]\\
\geq&2(p-2)mt^2+sm +mt\sqrt{8(p-2)s} \left[1+m(1+t)^2 \left(\frac{b\cdot e}{a(1-\alpha)r}\right)^2\right]^{\frac{1}{2}}.
\end{aligned}
\end{equation*}
Using $pmt=1-m$, we obtain \eqref{eq:affine-amgm}.

To estimate $b_\perp$, we project \eqref{eq:first-moment} onto $e^\perp$ and apply the Cauchy--Schwarz inequality to get
\begin{equation}\label{eq:transverse-projection}
\left(\frac{(1-\lambda_2(M))|b_\perp|}{ap(1-\alpha)r}\right)^2\leq1-\frac{1}{ar}\int_{\Sn} x^2u^2\,d\sigma=\frac{1-\alpha}{1-\alpha+\beta}\left[1+\frac{\mathcal R_N-a^{-1}\int_{\Sn}(b\cdot\xi)x^2u^p\,d\sigma}{(1-\alpha)r}\right],
\end{equation}
where the equality follows from \eqref{eq:nodal-remainder}.

With $m=\frac{1}{1+pt}$, the definitions \eqref{eq:gamma} and \eqref{eq:cubic-coefficient} give $\Gamma>0$ and the following inequalities.
\begin{equation}\label{eq:transverse-coefficient}
\frac{4}{n}+2t-(p-2)(1+3m)t^2 -\frac{1-m}{p}\sqrt{8(p-2)\left[\frac{2}{pm}+(p-2)t^2\right]} \leq\frac{5C_m}{4n(1-\alpha+\beta)},
\end{equation}
\begin{equation}\label{eq:sign-coefficient}
\frac{4}{n}+2t-(p-2)(1+3m)t^2 -\frac{1-m}{p}\sqrt{8(p-2)\left[\frac{2}{pm}+(p-2)t^2\right]} <\frac{C_m}{nm(1-\alpha)},
\end{equation}
and
\begin{equation}\label{eq:axial-coefficients}
\begin{aligned}
&\frac{1}{p}\sqrt{\frac{8(p-2)\Gamma}{nm}}>\frac{C_m}{nm(1-\alpha+\beta)},\qquad \frac{\sqrt{8(p-2)}}{p\left(\sqrt{\frac{\Gamma}{nm}}+\sqrt{\frac{2}{pm}+(p-2)t^2}\right)}\geq\frac{\sqrt{2}}{6},\\
&p\sqrt{\frac{1-\alpha}{1-\alpha+\beta}}\leq\frac{5}{2}.
\end{aligned}
\end{equation}

By \eqref{eq:cubic-coefficient} and $pmt=1-m$, we have
\begin{equation*}
\frac{C_m}{nm(1-\alpha+\beta)}=\frac{\Gamma}{nm}-\frac{2}{pm}-(p-2)t^2.
\end{equation*}

From \eqref{eq:nodal-remainder},
\begin{equation*}
\frac{1}{a}\int_{\Sn}(b\cdot\xi)x^2u^p\,d\sigma=\mathcal R_N-\beta r+\frac{1-\alpha+\beta}{a}\int_{\Sn} x^2u^2\,d\sigma
\end{equation*}
so \eqref{eq:axial-variation} becomes
\begin{equation*}
\frac{Q(\Phi_e,\Phi_e)}{a(1-\alpha)r}=-\frac{(p-2)mt^2+\left[\frac{2}{pm}+(p-2)t^2\right]\mathcal R_N}{(1-\alpha)r}+2-\frac{2}{p}+2t-(p-2)t^2-\frac{C_m}{nm(1-\alpha)}\left(1-\frac{1}{ar}\int_{\Sn} x^2u^2\,d\sigma\right).
\end{equation*}

Applying \eqref{eq:affine-amgm} with $s=\frac{2}{pm}+(p-2)t^2$ and using \eqref{eq:sign-coefficient}, we have
\begin{equation*}
\begin{aligned}
&0\leq\frac{Q(\Phi_e,\Phi_e)}{a(1-\alpha)r}\leq\frac{4}{n}+2t-(p-2)(1+3m)t^2 -\frac{1-m}{p}\sqrt{8(p-2)\left[\frac{2}{pm}+(p-2)t^2\right]}\\
&\quad-\frac{C_m}{nm(1-\alpha)} \left(1-\frac{1}{ar}\int_{\Sn} x^2u^2\,d\sigma\right)\\
<&\frac{C_m}{anm(1-\alpha)r}\int_{\Sn} x^2u^2\,d\sigma.
\end{aligned}
\end{equation*}
Hence $C_m>0$.

The positivity of $C_m$ and \eqref{eq:transverse-projection} give
\begin{equation*}
\begin{aligned}
&\frac{C_m}{nm(1-\alpha)}\left(\frac{(1-\lambda_2(M))|b_\perp|}{ap(1-\alpha)r}\right)^2\leq\frac{C_m}{nm(1-\alpha)}\left(1-\frac{1}{ar}\int_{\Sn} x^2u^2\,d\sigma\right)\\
\leq&2-\frac{2}{p}+2t-(p-2)t^2-\frac{(p-2)mt^2+\left[\frac{2}{pm}+(p-2)t^2\right]\mathcal R_N}{(1-\alpha)r}\\
\leq&\frac{4}{n}+2t-(p-2)(1+3m)t^2-\frac{1-m}{p}\sqrt{8(p-2)\left[\frac{2}{pm}+(p-2)t^2\right]}\left[1+m(1+t)^2\left(\frac{b\cdot e}{a(1-\alpha)r}\right)^2\right]^{\frac{1}{2}}\\
\leq&\frac{5C_m}{4n(1-\alpha+\beta)},
\end{aligned}
\end{equation*}
where we use \eqref{eq:affine-amgm} with $s=\frac{2}{pm}+(p-2)t^2$ and \eqref{eq:transverse-coefficient} in the last two inequalities. By \eqref{eq:axial-coefficients}, it follows that
\begin{equation*}
\frac{(1-\lambda_2(M))|b_\perp|}{a(1-\alpha)r}\leq p\sqrt{\frac{1-\alpha}{1-\alpha+\beta}}\sqrt{\frac{5m}{4}}\leq\frac{5}{2}\sqrt{\frac{5m}{4}}.
\end{equation*}
Note that $\lambda_2(M)\leq\min(m,1-m)$, which gives $\frac{m(1-m)}{1-\lambda_2(M)}\leq\min(m,1-m)$. Then we have
\begin{equation}\label{eq:transverse-multiplier}
\frac{m(1-m)|b_\perp|}{a(1-\alpha)r}\leq\frac{5}{2}\min(m,1-m)\sqrt{\frac{5m}{4}}<1-\frac{m}{4}+\frac{\sqrt{2}}{6}(1-m).
\end{equation}

We now use \eqref{eq:affine-amgm} with $s=\frac{\Gamma}{nm}$ in \eqref{eq:axial-variation}. Under the contradiction hypothesis this gives
\begin{equation*}
\begin{aligned}
&\frac{C_m}{nm(1-\alpha+\beta)}\frac{\int_{\Sn}(b\cdot\xi)x^2u^p\,d\sigma}{a(1-\alpha)r}\geq\frac{(p-2)mt^2+\frac{\Gamma}{nm}\mathcal R_N}{(1-\alpha)r}-2-4t+\frac{\Gamma}{nm}\\
\geq&2(p-2)mt^2+(1+m)\frac{\Gamma}{nm}-2-4t+\frac{1-m}{p}\sqrt{\frac{8(p-2)\Gamma}{nm}}\left[1+m(1+t)^2\left(\frac{b\cdot e}{a(1-\alpha)r}\right)^2\right]^{\frac{1}{2}}.
\end{aligned}
\end{equation*}
Equation \eqref{eq:cubic-coefficient} also gives
\begin{equation*}
2(p-2)mt^2+(1+m)\frac{\Gamma}{nm}-2-4t=\frac{(1+m)C_m}{nm(1-\alpha+\beta)}-\frac{4}{n}-2t+(p-2)(1+3m)t^2.
\end{equation*}
The identity for $2(p-2)mt^2+(1+m)\frac{\Gamma}{nm}-2-4t$ and \eqref{eq:transverse-coefficient} therefore give
\begin{equation*}
\begin{aligned}
&\frac{C_m}{nm(1-\alpha+\beta)}\frac{\int_{\Sn}(b\cdot\xi)x^2u^p\,d\sigma}{a(1-\alpha)r}\geq\frac{(1-\frac{m}{4})C_m}{nm(1-\alpha+\beta)}\\
&\quad+\frac{1-m}{p}\sqrt{\frac{8(p-2)\Gamma}{nm}}\left\{\left[1+m(1+t)^2\left(\frac{b\cdot e}{a(1-\alpha)r}\right)^2\right]^{\frac{1}{2}}-1\right\}\\
&+\frac{1-m}{p}\sqrt{8(p-2)}\left(\sqrt{\frac{\Gamma}{nm}}-\sqrt{\frac{2}{pm}+(p-2)t^2}\right).
\end{aligned}
\end{equation*}

By the second inequality in \eqref{eq:axial-coefficients}, we have
\begin{equation*}
\begin{aligned}
&\frac{\sqrt{8(p-2)}}{p}\left(\sqrt{\frac{\Gamma}{nm}}-\sqrt{\frac{2}{pm}+(p-2)t^2}\right)=\frac{\sqrt{8(p-2)}}{p\left(\sqrt{\frac{\Gamma}{nm}}+\sqrt{\frac{2}{pm}+(p-2)t^2}\right)}\frac{C_m}{nm(1-\alpha+\beta)}\\
\geq&\frac{\sqrt{2}}{6}\frac{C_m}{nm(1-\alpha+\beta)}.
\end{aligned}
\end{equation*}
As $C_m>0$, the first inequality in \eqref{eq:axial-coefficients} yields
\begin{equation}\label{eq:necessary-third-moment}
\frac{\int_{\Sn}(b\cdot\xi)x^2u^p\,d\sigma}{a(1-\alpha)r}\geq\frac{3m}{4}+\frac{\sqrt{2}}{6}(1-m)+(1-m)\left[1+m(1+t)^2\left(\frac{b\cdot e}{a(1-\alpha)r}\right)^2\right]^{\frac{1}{2}}>0.
\end{equation}

Using \eqref{eq:transverse-multiplier} and $m(1+t)^2\geq m^2$, we find
\begin{equation*}
\begin{aligned}
\left[\frac{\int_{\Sn}(b\cdot\xi)x^2u^p\,d\sigma}{a(1-\alpha)r}\right]^2 \geq&\left[1-\frac{m}{4}+\frac{\sqrt{2}}{6}(1-m)\right]^2+m(1+t)^2(1-m)^2\left(\frac{b\cdot e}{a(1-\alpha)r}\right)^2\\
>&\frac{m^2(1-m)^2|b_\perp|^2}{a^2(1-\alpha)^2r^2}+\frac{m^2(1-m)^2(b\cdot e)^2}{a^2(1-\alpha)^2r^2}=\frac{m^2(1-m)^2|b|^2}{a^2(1-\alpha)^2r^2}.
\end{aligned}
\end{equation*}
Consequently, we get
\begin{equation*}
\int_{\Sn}(b\cdot\xi)x^2u^p\,d\sigma>m(1-m)|b|
\end{equation*}
contrary to \eqref{eq:axial-third-moment}.
\end{proof}

\section{Instability in dimensions $n\geq7$}\label{sec:new-high}

Throughout this section, let $n\geq7$ and $u$ be a nonconstant positive solution of \eqref{eq:critical-equation}. write $m=\lambda_1(M)$ and fix a unit eigenvector $e$ satisfying $Me=me$. Recall that $x=e\cdot\xi$ and $t=\frac{1-m}{pm}$. If $z\neq0$, then $2<p<3$ gives
\begin{equation*}
\frac{d^4}{dz^4}|z|^p=p(p-1)(p-2)(p-3)|z|^{p-4}<0.
\end{equation*}
Thus, the Hermite interpolation argument for \eqref{eq:hermite-bound} does not provide the lower bound for $\mathcal R_N$ in \eqref{eq:affine-bound}. Instead, we apply the classical Sobolev inequality separately to the positive and negative parts of $(\frac{x}{\sqrt m}\pm\tau)u$ for a suitable $\tau>0$. The following proposition together with Proposition~\ref{new:weighted:proposition} gives instability in dimensions $n\geq7$, since
\begin{equation*}
\frac{1+\frac{3}{n}}{6}-\frac{p}{12}=\frac{n-6}{6n(n-2)}>0.
\end{equation*}

\begin{proposition}\label{new:high-principal}
If $m\geq\frac{p}{12}$, then $Q(\Phi_e,\Phi_e)<0$.
\end{proposition}

For the estimates below, set
\begin{equation*}
c_1=\frac{\beta}{1-\alpha},\qquad c_2=\frac{1}{n(1-\alpha^2)-1}.
\end{equation*}
We begin with a lower bound for $\mathcal R_N$.

\begin{lemma}\label{new:high-translates}
We have
\begin{equation}\label{eq:translated-bounds}
\mathcal R_N\geq c_2(1-\alpha)rm,\qquad (1-c_1+c_2m)r\leq1-m.
\end{equation}
\end{lemma}

\begin{proof}
Let $\tau=\sqrt{c_2}$. From $n(1-\alpha^2)>2$, we have $0<\tau<1$. We first claim that
\begin{equation}\label{eq:translated-norms}
\frac{\alpha}{2}\sum_{\varepsilon=\pm1}\left[\left\|\left(\frac{x}{\sqrt m}+\varepsilon\tau\right)_+u\right\|_p^2+\left\|\left(\frac{x}{\sqrt m}+\varepsilon\tau\right)_-u\right\|_p^2\right]\geq1+\tau^2.
\end{equation}
Indeed, since $\int_{\Sn} xu^p\,d\sigma=0$ and $\int_{\Sn} x^2u^p\,d\sigma=m$, we have
\begin{equation*}
\left\|\left(\frac{|1\pm\frac{x}{\sqrt m}|}{\sqrt2}\right)^{\frac{4}{n}}\right\|_{L^{\frac{n}{2}}(u^p\,d\sigma)}=1.
\end{equation*}
Thus H\"older's inequality gives, for $\varepsilon=\pm1$,
\begin{equation*}
\left\|\left(\frac{x}{\sqrt m}+\varepsilon\tau\right)_\pm u\right\|_p^2\geq\int_{\Sn}\left(\frac{x}{\sqrt m}+\varepsilon\tau\right)_\pm^2\left(\frac{|1\pm\frac{x}{\sqrt m}|}{\sqrt2}\right)^{\frac{4}{n}}u^p\,d\sigma.
\end{equation*}
The weighted sum over both signs on the left-hand side of \eqref{eq:translated-norms} is bounded below by $\int_{\Sn} F(\frac{|x|}{\sqrt m})u^p\,d\sigma$, where
\begin{equation*}
F(z)=\begin{cases}
(z^2+\tau^2)\left(\frac{1+z}{2}\right)^{\frac{4}{n}},& \text{if }z\geq\tau,\\[1mm]
2^{-\frac{4}{n}-1}\left[(z+\tau)^2(1+z)^{\frac{4}{n}}+(\tau-z)^2(1-z)^{\frac{4}{n}}\right],&\text{if }0\leq z\leq\tau.
\end{cases}
\end{equation*}
Since $\int_{\Sn}\frac{x^2}{m}u^p\,d\sigma=1$ and $\frac{1+\tau^2}{n}=(1-\alpha^2)\tau^2$, inequality \eqref{eq:translated-norms} follows once we show
\begin{equation}\label{eq:tangent-bound}
F(z)\geq\alpha^2\tau^2+\left(1+\frac{1+\tau^2}{n}\right)z^2\qquad\text{if }z\geq0.
\end{equation}
For $z>\tau$, we have
\begin{equation*}
\left.\frac{d^2}{ds^2}F(\sqrt s)\right|_{s=z^2} = \frac{(1+z)^{\frac{4}{n}-2}}{n2^{\frac{4}{n}}z^3}\left[4z^2(1+z)+(z^2+\tau^2)\left\{\left(\frac{4}{n}-2\right)z-1\right\}\right] \geq \frac{2(1+z)^{\frac{4}{n}-2}}{n2^{\frac{4}{n}}z}\left(1+\frac{4z}{n}\right)>0.
\end{equation*}
Moreover,
\begin{equation*}
F(1)=1+\tau^2,\qquad \left.\frac{d}{ds}F(\sqrt s)\right|_{s=1}=1+\frac{1+\tau^2}{n}.
\end{equation*}
The convexity of $s\mapsto F(\sqrt s)$ on $[\tau^2,\infty)$ gives
\begin{equation*}
F(z)\geq1+\tau^2+\left(1+\frac{1+\tau^2}{n}\right)(z^2-1)=\alpha^2\tau^2+\left(1+\frac{1+\tau^2}{n}\right)z^2.
\end{equation*}
If $0\leq z\leq\tau$, the coefficient of $z^{2j}$ in the power series for $2^{\frac{4}{n}}F(z)$ is, for $j\geq2$,
\begin{equation*}
\binom{\frac{4}{n}}{2j-2}+2\tau\binom{\frac{4}{n}}{2j-1}+\tau^2\binom{\frac{4}{n}}{2j} = \binom{\frac{4}{n}}{2j-2}\left[\left(1-\frac{\tau(2j-2-\frac{4}{n})}{2j-1}\right)^2+\frac{\tau^2(2j-2-\frac{4}{n})(1+\frac{4}{n})}{(2j-1)^2(2j)}\right]<0.
\end{equation*}
Thus $s\mapsto F(\sqrt s)$ is concave on $[0,\tau^2]$. Inequality \eqref{eq:tangent-bound} holds at $z=\tau$ by the convexity argument and is an equality at $z=0$. Hence it also holds for $0\leq z\leq\tau$. Integrating \eqref{eq:tangent-bound} with $z=\frac{|x|}{\sqrt m}$ yields \eqref{eq:translated-norms}.

By the classical Sobolev inequality applied to the positive and negative parts of $(\frac{x}{\sqrt m}+\varepsilon\tau)u$, \eqref{eq:translated-energy} and \eqref{eq:translated-norms}, we have
\begin{equation*}
1+\frac{\mathcal R_N}{m}+\tau^2[1-(1-\alpha)r] \geq \frac{\alpha}{2a}\sum_{\varepsilon=\pm1}\left[\left\|\left(\frac{x}{\sqrt m}+\varepsilon\tau\right)_+u\right\|_p^2+\left\|\left(\frac{x}{\sqrt m}+\varepsilon\tau\right)_-u\right\|_p^2\right] \geq 1+\tau^2,
\end{equation*}
which gives the first inequality in \eqref{eq:translated-bounds}.

For the second inequality, by \eqref{eq:nodal-remainder}, we have
\begin{equation*}
m+\mathcal R_N+\frac{\alpha\|u\sqrt{1-x^2}\|_2^2+\beta\|\nabla(u\sqrt{1-x^2})\|_2^2}{a}=1+[\beta-(1-\alpha)]r.
\end{equation*}

On the other hand, the classical Sobolev inequality and H\"older's inequality give
\begin{equation*}
\frac{\alpha\|u\sqrt{1-x^2}\|_2^2+\beta\|\nabla(u\sqrt{1-x^2})\|_2^2}{a}\geq\frac{\alpha}{a}\left(\int_{\Sn}(1-x^2)^{\frac{p}{2}}u^p\,d\sigma\right)^{\frac{2}{p}}\geq\alpha(1-m).
\end{equation*}

Combining these inequalities with the first inequality in \eqref{eq:translated-bounds}, we obtain
\begin{equation*}
1+[\beta-(1-\alpha)]r\geq m+\mathcal R_N+\alpha(1-m)\geq m+c_2(1-\alpha)rm+\alpha(1-m),
\end{equation*}
which yields the second inequality in \eqref{eq:translated-bounds}.
\end{proof}

The following lemma gives an upper bound for $|b|$.

\begin{lemma}\label{new:high-multiplier}
We have
\begin{equation}\label{eq:high-multiplier}
\frac{|b|^2}{a^2(1-\alpha)^2r(1-r)}\leq\max\left\{\frac{2p^2c_1r}{2r+c_1(1-r)}, pc_1, c_1+\frac{p^2}{4+2c_1}\right\}.
\end{equation}
\end{lemma}

\begin{proof}
If $b=0$, \eqref{eq:high-multiplier} is immediate. Otherwise write $w=\frac{b}{|b|}$. Equation \eqref{eq:directional-discriminant} with $v=w$ gives
\begin{equation*}
\frac{[1+(p-1)w^TMw]^2|b|^2}{p^2a^2(1-\alpha)(1-r)}\leq\beta r+\frac{|b|}{a}\int_{\Sn}(w\cdot\xi)^3u^p\,d\sigma-\frac{1-\alpha+\beta}{a}\int_{\Sn}(w\cdot\xi)^2u^2\,d\sigma.
\end{equation*}

By \eqref{eq:first-moment} and the Cauchy--Schwarz inequality, we get
\begin{equation*}
\frac{1}{a}\int_{\Sn}(w\cdot\xi)^2u^2\,d\sigma\geq\frac{(1-w^TMw)^2|b|^2}{p^2a^2(1-\alpha)^2r}.
\end{equation*}

Integrating $(1-w\cdot\xi)(w\cdot\xi+w^TMw)^2u^p\geq0$ yields
\begin{equation*}
\int_{\Sn}(w\cdot\xi)^3u^p\,d\sigma\leq(w^TMw)(1-w^TMw).
\end{equation*}

Combining these estimates, we obtain
\begin{equation*}
\begin{aligned}
&\frac{|b|^2}{a^2(1-\alpha)^2r(1-r)}\left\{[1+(p-1)w^TMw]^2+(1+c_1)(1-w^TMw)^2\frac{1-r}{r}\right\}\\
\leq&p^2c_1+\frac{p^2|b|}{a(1-\alpha)r}(w^TMw)(1-w^TMw).
\end{aligned}
\end{equation*}

By Young's inequality, we have
\begin{equation*}
\frac{p^2|b|}{a(1-\alpha)r}(w^TMw)(1-w^TMw)\leq\left(1+\frac{c_1}{2}\right)\frac{|b|^2(1-w^TMw)^2}{a^2(1-\alpha)^2r^2}+\frac{p^4}{4+2c_1}(w^TMw)^2.
\end{equation*}

Since $1+(p-1)w^TMw=1-w^TMw+p(w^TMw)$, we have
\begin{equation*}
\begin{aligned}
&\frac{|b|^2}{a^2(1-\alpha)^2r(1-r)}\left[\left(1+\frac{c_1(1-r)}{2r}\right)(1-w^TMw)^2+2p(w^TMw)(1-w^TMw)+p^2(w^TMw)^2\right]\\
\leq&p^2c_1(1-w^TMw)^2+2p^2c_1(w^TMw)(1-w^TMw)+\left(p^2c_1+\frac{p^4}{4+2c_1}\right)(w^TMw)^2.
\end{aligned}
\end{equation*}

Since $0<w^TMw<1$, comparing the coefficients of $(1-w^TMw)^2$, $(w^TMw)(1-w^TMw)$ and $(w^TMw)^2$ gives \eqref{eq:high-multiplier}.
\end{proof}

\begin{proof}[Proof of Proposition~\ref{new:high-principal}]
Since $m\geq\frac{p}{12}$ and $n(\alpha^{-1}-1)>1$, we have
\begin{equation*}
c_1t\leq\frac{(p-2)(12-p)}{p^2n(\alpha^{-1}-1)}<1,\qquad \frac{\Gamma}{nm}=\frac{3+4t-c_1t^2}{1+c_1}>0.
\end{equation*}
Therefore, by \eqref{eq:principal-variation}, \eqref{eq:axial-third-moment} and \eqref{eq:translated-bounds}, we have
\begin{equation}\label{eq:high-variation}
\frac{Q(\Phi_e,\Phi_e)}{a}\leq-(p-2)mt^2(1-r)-(1-\alpha)r\left[\frac{(p-2)mt^2}{1-\alpha}+\frac{\Gamma}{nm}(1+c_2m)-2-4t\right]+\frac{(1-m)|C_m|}{n(1-\alpha+\beta)}\frac{|b|}{a}.
\end{equation}

To prove $Q(\Phi_e,\Phi_e)<0$, it suffices to show
\begin{equation}\label{eq:principal-criterion}
\frac{p^2m}{4(p-2)(1-\alpha)r(1-r)} \left[\frac{C_m}{n(1-\alpha+\beta)}\frac{|b|}{a}\right]^2 <\frac{(p-2)mt^2}{1-\alpha} +\frac{\Gamma}{nm}(1+c_2m)-2-4t.
\end{equation}
If \eqref{eq:principal-criterion} holds, $1-m=pmt$ and the arithmetic--geometric mean inequality give
\begin{equation*}
\begin{aligned}
\frac{(1-m)|C_m|}{n(1-\alpha+\beta)}\frac{|b|}{a}<&2\sqrt{(p-2)mt^2(1-r)(1-\alpha)r \left[\frac{(p-2)mt^2}{1-\alpha} +\frac{\Gamma}{nm}(1+c_2m)-2-4t\right]}\\
\leq&(p-2)mt^2(1-r) +(1-\alpha)r\left[\frac{(p-2)mt^2}{1-\alpha} +\frac{\Gamma}{nm}(1+c_2m)-2-4t\right].
\end{aligned}
\end{equation*}
Equation \eqref{eq:high-variation} then gives $Q(\Phi_e,\Phi_e)<0$.

We first prove the following lower bound for the right-hand side of \eqref{eq:principal-criterion}.
\begin{equation}\label{eq:remainder-coefficient}
\frac{(p-2)mt^2}{1-\alpha}+\frac{\Gamma}{nm}(1+c_2m)-2-4t\geq\frac{1+3c_2-2c_1}{1+c_1}\,m\frac{1-c_1+(1+2c_1)m}{2+c_1}.
\end{equation}

Write $z=\frac{1}{n}\in(0,\frac{1}{7}]$ with $p=\frac{2}{1-2z}$ and $\alpha=4^{-z}$. All derivatives below are with respect to $z$. By the definitions of $c_1$ and $c_2$, we have
\begin{equation*}
c_1=\frac{4z^2}{(1-2z)(4^z-1)},\qquad c_2=\frac{z}{1-4^{-2z}-z}.
\end{equation*}
The derivatives satisfy
\begin{equation*}
0<\frac{d}{dz}\log[n(\alpha^{-1}-1)]<\log4,\qquad c_1'\geq\frac{4}{\log4},\qquad 0<c_2'\leq(\log4)c_2(1+c_2).
\end{equation*}
If $n=7$, then $\alpha^7=\frac{1}{4}$ and $\beta=\frac{4\alpha}{35}$ give
\begin{equation}\label{eq:dimension-seven}
\frac{41}{50}<\alpha<\frac{32}{39},\qquad \frac{13}{25}<c_1<\frac{11}{21},\qquad \frac{17}{22}<c_2<\frac{7}{9}.
\end{equation}
In particular, $c_1<\frac{11}{21}$ and $c_2<\frac{7}{9}$ for all $n\geq7$. The derivative bounds for $c_1,c_2$ and $p=\frac{2}{1-2z}$ give $3c_2'<2c_1'$ and $p'\leq3c_1'$. Consequently, $c_2-c_1$ and $1+3c_2-2c_1$ decrease with $z$ and remain positive by \eqref{eq:dimension-seven}. In particular, the right-hand side of \eqref{eq:remainder-coefficient} is positive.

Using $pmt=1-m$ and $m<1$, we see that \eqref{eq:remainder-coefficient} is equivalent to
\begin{equation}\label{eq:remainder-comparison}
1-2c_1+\frac{4}{p}(c_2-c_1)+(1+3c_2-2c_1)\frac{1+2c_1}{2+c_1}m+\frac{c_1t}{p}\left[\frac{n(1+c_1)}{\alpha}-1-pt-4p-c_2\right]\geq 0.
\end{equation}

Since $m\geq\frac{p}{12}$, we have
\begin{equation*}
1-2c_1+\frac{4}{p}(c_2-c_1)+(1+3c_2-2c_1)\frac{1+2c_1}{2+c_1}m \geq 1-2c_1+\frac{4}{p}(c_2-c_1)+(1+3c_2-2c_1)\frac{1+2c_1}{2+c_1}\frac{p}{12}.
\end{equation*}
This lower bound decreases with $z$. Indeed, $c_2-c_1$ and $1+3c_2-2c_1$ are positive and decreasing while $p$ increases. Using $1+3c_2-2c_1\leq\frac{10}{3}$, $\frac{p}{12}\leq\frac{7}{30}$ and $p'\leq3c_1'$, we have
\begin{equation*}
\frac{d}{dz}\left[(1+3c_2-2c_1)\frac{1+2c_1}{2+c_1}\frac{p}{12}\right] \leq\frac{17}{12}c_1'<2c_1'.
\end{equation*}

Since $0<t\leq\frac{5-12z}{p}$, by concavity in $t$, we have
\begin{equation}\label{eq:endpoint-in-t}
\frac{c_1t}{p}\left[\frac{n(1+c_1)}{\alpha}-1-pt-4p-c_2\right] \geq-\frac{c_1(5-12z)}{p^2}\left[6-12z+4p+c_2-\frac{n(1+c_1)}{\alpha}\right]_+,
\end{equation}
where the right-hand side is nonincreasing in $z$.

Both lower bounds are nonincreasing in $z$. At $z=\frac{1}{7}$, \eqref{eq:dimension-seven} gives
\begin{equation*}
1-2c_1+\frac{4}{p}(c_2-c_1)+(1+3c_2-2c_1)\frac{1+2c_1}{2+c_1}\frac{p}{12}>\frac{8}{11}
\end{equation*}
and
\begin{equation*}
-\frac{c_1(5-12z)}{p^2}\left[6-12z+4p+c_2-\frac{n(1+c_1)}{\alpha}\right]_+>-\frac{8}{11}.
\end{equation*}
The two bounds at $z=\frac{1}{7}$ give \eqref{eq:remainder-comparison} and hence \eqref{eq:remainder-coefficient}.

We next estimate $|C_m|$. By \eqref{eq:gamma} and \eqref{eq:cubic-coefficient},
\begin{equation*}
\frac{C_m}{n(1-\alpha)}=\frac{2}{p}(1-c_1)+\left(3-\frac{4}{p}\right)m-\left[1+c_1+\frac{1}{n(\alpha^{-1}-1)}\right](p-2)mt^2.
\end{equation*}
Using $m\geq\frac{p}{12}$, $2<p\leq\frac{14}{5}$, $0<c_1<\frac{11}{21}$ and $n(\alpha^{-1}-1)>\frac{4}{3}$, we get
\begin{equation*}
\frac{2}{p}(1-c_1)+\left(3-\frac{4}{p}\right)m>\frac{1}{2},\qquad 0\leq\left[1+c_1+\frac{1}{n(\alpha^{-1}-1)}\right](p-2)mt^2<\frac{5}{6}.
\end{equation*}
It follows that
\begin{equation*}
\frac{|C_m|}{n(1-\alpha)}\leq\frac{2}{p}(1-c_1)+\left(3-\frac{4}{p}\right)m.
\end{equation*}

The bound $n(\alpha^{-1}-1)>\frac{4}{3}$ gives
\begin{equation*}
\frac{\frac{2}{p}(1-c_1)+m(3-\frac{4}{p})}{3-\frac{2}{p}(1+c_1)}=m+(1-m)\frac{2(1-c_1)}{4+[3n(\alpha^{-1}-1)-2]c_1}\leq\frac{1-c_1+(1+2c_1)m}{2+c_1}.
\end{equation*}
Consequently,
\begin{equation}\label{eq:cubic-bound}
\frac{|C_m|}{n(1-\alpha+\beta)}\leq\left(\frac{3}{1+c_1}-\frac{2}{p}\right)\frac{1-c_1+(1+2c_1)m}{2+c_1}.
\end{equation}

To combine \eqref{eq:remainder-coefficient} and \eqref{eq:cubic-bound}, we use
\begin{equation}\label{eq:dimension-factor}
\frac{p^2(1-\alpha)}{4(p-2)}\frac{(3p-2-2c_1)^2}{(1+c_1)(1+3c_2-2c_1)}<\frac{11}{3}.
\end{equation}
Indeed,
\begin{equation*}
\frac{p^2(1-\alpha)}{4(p-2)}=\frac{\alpha}{4(1-2z)}n(\alpha^{-1}-1),\qquad 3p-2-2c_1=4+[3n(\alpha^{-1}-1)-2]c_1.
\end{equation*}
The factor $\frac{p^2(1-\alpha)}{4(p-2)}$ increases with $z$. Since $n(\alpha^{-1}-1)>\frac{4}{3}$, the quotient $\frac{(3p-2-2c_1)^2}{1+c_1}$ also increases with $z$. On the other hand, $1+3c_2-2c_1$ is positive and decreasing in $z$. Thus the left-hand side of \eqref{eq:dimension-factor} increases with $z$. At $n=7$, the values $p=\frac{14}{5}$, $\alpha=2^{-\frac{2}{7}}$ and $\beta=\frac{4\alpha}{35}$ give \eqref{eq:dimension-factor}.

In view of \eqref{eq:remainder-coefficient}, \eqref{eq:cubic-bound} and \eqref{eq:dimension-factor}, to prove \eqref{eq:principal-criterion} it remains to show
\begin{equation}\label{eq:multiplier-comparison}
\frac{11}{3p^2}\frac{1-c_1+(1+2c_1)m}{2+c_1}\frac{|b|^2}{a^2(1-\alpha)^2r(1-r)}<1.
\end{equation}
By \eqref{eq:high-multiplier}, it suffices to prove \eqref{eq:multiplier-comparison} with $\frac{|b|^2}{a^2(1-\alpha)^2r(1-r)}$ replaced by each of the three terms in the maximum in \eqref{eq:high-multiplier}. Since $0<c_1<\frac{11}{21}$ and $p=2+n(\alpha^{-1}-1)c_1>2+\frac{4c_1}{3}$, one can check that
\begin{equation*}
pc_1\leq c_1+\frac{p^2}{4+2c_1},\qquad \frac{c_1}{p^2}+\frac{1}{4+2c_1}<\frac{3}{11}.
\end{equation*}
This gives the desired bound for the two terms independent of $r$, since $0<\frac{1-c_1+(1+2c_1)m}{2+c_1}<1$.

For the first term in the maximum in \eqref{eq:high-multiplier}, we need to prove
\begin{equation*}
\frac{11}{3}\frac{1-c_1+(1+2c_1)m}{2+c_1}<\frac{1}{c_1}+\frac{1-r}{2r}.
\end{equation*}
By the second inequality in \eqref{eq:translated-bounds}, we have
\begin{equation*}
\frac{1-r}{2r}\geq\frac{(1+c_2)m-c_1}{2(1-m)}.
\end{equation*}
It is therefore sufficient to show that
\begin{equation*}
\frac{11}{3}\frac{1-c_1+(1+2c_1)m}{2+c_1}<\frac{1}{c_1}+\frac{(1+c_2)m-c_1}{2(1-m)}.
\end{equation*}
For fixed $m\in(0,1)$, the difference between the right-hand side and the left-hand side decreases with $z$ by $3c_2'<2c_1'$ and $0<c_1<\frac{3}{5}$. Thus it suffices to check this inequality at $n=7$. Since $c_1>\frac{1}{2}$ in this case, we have $\frac{1-c_1+(1+2c_1)m}{2+c_1}\leq\frac{1+4m}{5}$. At this endpoint \eqref{eq:dimension-seven} gives
\begin{equation*}
\begin{aligned}
&c_1(1-m)\left[\frac{1}{c_1}+\frac{(1+c_2)m-c_1}{2(1-m)}-\frac{11}{3}\frac{1-c_1+(1+2c_1)m}{2+c_1}\right]\\
>&1-\frac{1}{2}\left(\frac{11}{21}\right)^2-\left[1-\frac{1}{2}\frac{13}{25}\left(1+\frac{17}{22}\right)\right]m-\frac{11}{15}\frac{11}{21}(1+4m)(1-m)>0.
\end{aligned}
\end{equation*}
This completes the proof of \eqref{eq:multiplier-comparison}.

Finally, combining \eqref{eq:cubic-bound} with \eqref{eq:dimension-factor} and \eqref{eq:multiplier-comparison}, we obtain
\begin{equation*}
\begin{aligned}
&\frac{p^2m}{4(p-2)(1-\alpha)r(1-r)}\left[\frac{C_m}{n(1-\alpha+\beta)}\frac{|b|}{a}\right]^2\\
\leq&\frac{p^2(1-\alpha)m}{4(p-2)}\left(\frac{3}{1+c_1}-\frac{2}{p}\right)^2\left[\frac{1-c_1+(1+2c_1)m}{2+c_1}\right]^2\frac{|b|^2}{a^2(1-\alpha)^2r(1-r)}\\
<&\frac{1+3c_2-2c_1}{1+c_1}\,m\frac{1-c_1+(1+2c_1)m}{2+c_1}\leq\frac{(p-2)mt^2}{1-\alpha}+\frac{\Gamma}{nm}(1+c_2m)-2-4t,
\end{aligned}
\end{equation*}
where the last inequality follows from \eqref{eq:remainder-coefficient}. This proves \eqref{eq:principal-criterion} and completes the proof.
\end{proof}

\section{The three-dimensional case}\label{sec:three}

Let $n=3$, so $p=6$, $\alpha=2^{-\frac{2}{3}}$ and $\beta=\frac{4\alpha}{3}$. Proposition~\ref{new:weighted:proposition} gives a negative second variation when $m=\lambda_1(M)\leq\frac{3}{10}$. Throughout this section, let $u$ be a nonconstant positive solution of \eqref{eq:critical-equation} with $m>\frac{3}{10}$. Fix a unit eigenvector $Me=me$ and write $x=e\cdot\xi$ and $t=\frac{1-m}{6m}$.

\begin{remark}\label{new:three-principal-obstruction}
In dimension three, \eqref{eq:affine-bound} does not yield a negative second variation in the principal direction when $1-m=O(r)$ as $r\to0$. Indeed, even when $b=0$, substituting \eqref{eq:affine-bound} into \eqref{eq:principal-variation} gives
\begin{equation*}
\frac{Q(\Phi_e,\Phi_e)}{a} \leq-\frac{(1-m)^2}{9m}+(1-\alpha)r\left[2+4t-\frac{\Gamma}{3m}\left(1+\frac{m}{1-2(1-\alpha)r}\right)\right]=\frac{4(1-\alpha)(5\alpha-3)}{3+\alpha}r+O(r^2),
\end{equation*}
with $5\alpha-3>0$. Thus this upper bound has a positive leading coefficient.
\end{remark}

We first give a sufficient condition for instability in the directions orthogonal to $e$. We then show that every solution in the remaining range satisfies this condition.

\subsection{The transverse second variation}\label{n3:small:transverse-section}
Choose an orthonormal eigenbasis $Me_i=\lambda_i e_i$ with $e_1=e$ and $\lambda_1=m$ and write $\xi_i=e_i\cdot\xi$. Recall from Section~\ref{sec:common} that
\begin{equation}\label{eq:transverse-directions}
6\lambda_i\Phi_{e_i}=6\lambda_i\nabla\xi_i\cdot\nabla u+(1-4\lambda_i)\xi_i u\in\Tu.
\end{equation}
The coefficients in this expression remain bounded as $\lambda_i\to0$.

\begin{lemma}\label{n3:small:transverse-instability}
Suppose $0<r\leq\frac{1}{64}$ and
\begin{equation}\label{eq:transverse-width}
\frac{5}{8}r<1-m<\frac{2}{3}r.
\end{equation}
Then we have
\begin{equation}\label{eq:transverse-instability}
\sum_{i=2}^4\lambda_i^2Q(\Phi_{e_i},\Phi_{e_i})<0.
\end{equation}
\end{lemma}

\begin{proof}
With $x=\xi_1$, the identities $\tr M=1$ and $\int_{\mathbb{S}^3}\xi u^6\,d\sigma=0$ give
\begin{equation}\label{eq:transverse-sums}
\sum_{i=2}^4\lambda_i=1-m,\qquad \sum_{i=1}^4\int_{\mathbb{S}^3}\xi_i^2u^2\,d\sigma=ar,\qquad \sum_{i=1}^4\int_{\mathbb{S}^3}(b\cdot\xi)\xi_i^2u^6\,d\sigma=0.
\end{equation}
In dimension three the constants satisfy
\begin{equation}\label{eq:sphere-three-constants}
\frac{5}{8}<\alpha<\frac{2}{3},\qquad \frac{5}{6}<\beta<\frac{6}{7},\qquad \beta-(1-\alpha)<\frac{1}{2},\qquad \frac{2}{3}<\frac{\beta}{1-\alpha+\beta}<\frac{5}{7}.
\end{equation}
In particular, \eqref{eq:transverse-width} gives $m>\frac{15}{16}$.

For $i=2,3,4$, \eqref{eq:conformal-second-variation} with $t_i=\frac{1-\lambda_i}{6\lambda_i}$ gives
\begin{equation}\label{eq:transverse-variation}
\begin{aligned}
\frac{36\lambda_i^2}{a}Q(\Phi_{e_i},\Phi_{e_i}) ={}&[\beta(1-\lambda_i)^2-36(1-\alpha)\lambda_i^2]r-4\lambda_i(1-\lambda_i)^2\\
&+[12(1-\alpha)\lambda_i(2+7\lambda_i)-\beta(1-\lambda_i)^2]\frac{1}{a}\int_{\mathbb{S}^3}\xi_i^2u^2\,d\sigma-\frac{4(1+\lambda_i+\lambda_i^2)}{a}\int_{\mathbb{S}^3}(b\cdot\xi)\xi_i^2u^6\,d\sigma.
\end{aligned}
\end{equation}
Since $\frac{|b|}{a}<1-r<1$, we have
\begin{equation*}
\left|\frac{1}{a}\int_{\mathbb{S}^3}(b\cdot\xi)\xi_i^2u^6\,d\sigma\right|\leq\lambda_i,\qquad \sum_{i=2}^4\lambda_i^k\leq(1-m)^k\quad \text{ for any }k\geq1.
\end{equation*}
Summing \eqref{eq:transverse-variation} over $i=2,3,4$ and using \eqref{eq:transverse-sums}, $\beta-36(1-\alpha)<0$ and $84(1-\alpha)-\beta<32$, we obtain
\begin{equation}\label{eq:transverse-trace}
\begin{aligned}
\frac{36}{a}\sum_{i=2}^4\lambda_i^2Q(\Phi_{e_i},\Phi_{e_i})\leq&2\beta r+\frac{\beta}{a}\int_{\mathbb{S}^3} x^2u^2\,d\sigma-4(1-m)+\frac{4}{a}\int_{\mathbb{S}^3}(b\cdot\xi)x^2u^6\,d\sigma+24(1-\alpha)r(1-m)\\
&+[12+4(1-m)+32r](1-m)^2\\
<&2\beta r+\frac{\beta}{a}\int_{\mathbb{S}^3} x^2u^2\,d\sigma-4(1-m)+\frac{4}{a}\int_{\mathbb{S}^3}(b\cdot\xi)x^2u^6\,d\sigma+10r(1-m)+13(1-m)^2.
\end{aligned}
\end{equation}

By \eqref{eq:nodal-remainder}, we have
\begin{equation*}
\frac{\beta}{a}\int_{\mathbb{S}^3} x^2u^2\,d\sigma=\frac{\beta}{1-\alpha+\beta}\left[\frac{1}{a}\int_{\mathbb{S}^3}(b\cdot\xi)x^2u^6\,d\sigma+\beta r-\mathcal R_N\right].
\end{equation*}
Substituting this into \eqref{eq:transverse-trace} gives
\begin{equation*}
\begin{aligned}
\frac{36}{a}\sum_{i=2}^4\lambda_i^2Q(\Phi_{e_i},\Phi_{e_i})<&\left(2\beta+\frac{\beta^2}{1-\alpha+\beta}\right)r-4(1-m)-\frac{\beta}{1-\alpha+\beta}\mathcal R_N\\
&+\frac{4(1-\alpha)+5\beta}{a(1-\alpha+\beta)}\int_{\mathbb{S}^3}(b\cdot\xi)x^2u^6\,d\sigma+10r(1-m)+13(1-m)^2.
\end{aligned}
\end{equation*}
Since $1-m<\frac{2r}{3}<2(1-\alpha)r$, we have $0<1-2(1-\alpha)r<m$, then \eqref{eq:affine-bound} gives
\begin{equation}\label{eq:axial-remainder}
\mathcal R_N\geq(1-\alpha)r+2m^2(1+t)^2\left(\frac{b\cdot e_1}{a}\right)^2.
\end{equation}

Set $b_\perp=b-(b\cdot e_1)e_1$. Since $\lambda_i\leq1-m$ for $i\geq2$, projecting \eqref{eq:first-moment} onto $e_1^\perp$ gives
\begin{equation*}
\frac{|b_\perp|}{a}\leq\frac{6(1-\alpha)r}{m}<\frac{5}{2}r.
\end{equation*}
Equations \eqref{eq:axial-third-moment} and \eqref{eq:transverse-third-moment} give
\begin{equation}\label{eq:transverse-moments}
\left|\int_{\mathbb{S}^3} x^3u^6\,d\sigma\right|\leq m(1-m),\qquad \left|\int_{\mathbb{S}^3}(\xi-xe_1)x^2u^6\,d\sigma\right|^2\leq m(1-m)^2.
\end{equation}
Combining these bounds with $b=(b\cdot e_1)e_1+b_\perp$ and $m<1$, we obtain
\begin{equation}\label{eq:mixed-third-moment}
\frac{1}{a}\left|\int_{\mathbb{S}^3}(b\cdot\xi)x^2u^6\,d\sigma\right| \leq m(1-m)\frac{|b\cdot e_1|}{a}+\frac{5}{2}r(1-m).
\end{equation}
Equations \eqref{eq:sphere-three-constants} and \eqref{eq:mixed-third-moment} imply
\begin{equation}\label{eq:axial-square}
\begin{aligned}
&\frac{4(1-\alpha)+5\beta}{a(1-\alpha+\beta)}\int_{\mathbb{S}^3}(b\cdot\xi)x^2u^6\,d\sigma-\frac{2\beta m^2(1+t)^2}{1-\alpha+\beta}\left(\frac{b\cdot e_1}{a}\right)^2\\
\leq&\frac{[4(1-\alpha)+5\beta]^2(1-m)^2}{8\beta(1-\alpha+\beta)(1+t)^2}+\frac{5[4(1-\alpha)+5\beta]}{2(1-\alpha+\beta)}r(1-m)\\
<&5(1-m)^2+\frac{25}{2}r(1-m).
\end{aligned}
\end{equation}

Combining \eqref{eq:transverse-trace} with \eqref{eq:nodal-remainder}, \eqref{eq:transverse-width}, \eqref{eq:sphere-three-constants}, \eqref{eq:axial-remainder} and \eqref{eq:axial-square}, we obtain
\begin{equation*}
\begin{aligned}
\frac{36}{a}\sum_{i=2}^4\lambda_i^2Q(\Phi_{e_i},\Phi_{e_i}) <&\left[2\beta+\frac{\beta[\beta-(1-\alpha)]}{1-\alpha+\beta}\right]r-4(1-m)+\frac{45}{2}r(1-m)+18(1-m)^2\\
<&-\frac{3}{7}r+24r^2\leq-\frac{3}{56}r<0
\end{aligned}
\end{equation*}
as desired.
\end{proof}

\subsection{Concentration estimates}\label{n3:local-section}
We now prove that every solution under consideration satisfies the assumptions of Lemma~\ref{n3:small:transverse-instability}. The estimates in this subsection do not use stability.

Since $\beta\|\nabla u\|_2^2=a(1-r)$, we put
\begin{equation}\label{eq:normalized-deviation}
H=\frac{u-\int_{\mathbb{S}^3} u\,d\sigma}{\sqrt{a(1-r)}}.
\end{equation}
Then we have $\int_{\mathbb{S}^3} H\,d\sigma=0$ and $\beta\|\nabla H\|_2^2=1$.

Let $H_j$ denote the orthogonal projection of $H$ onto the spherical harmonics of degree $j$. We will estimate $\|H\|_2^2$ to exclude large $r$. To estimate $1-m$, we consider the quantity
\begin{equation}\label{eq:angular-defect}
1-\left[\frac{2}{\pi}+\left(1-\frac{2}{\pi}\right)\tr M^2\right]^{\frac{2}{3}}.
\end{equation}
Since
\begin{equation*}
1-\tr M^2=\iint_{\mathbb{S}^3\times\mathbb{S}^3}[1-(\xi\cdot\zeta)^2]u^6(\xi)u^6(\zeta)\,d\sigma(\xi)\,d\sigma(\zeta),
\end{equation*}
\eqref{eq:angular-defect} vanishes for a probability measure satisfying the zero-barycenter condition precisely when the measure is $\frac{1}{2}(\delta_e+\delta_{-e})$ for some $e\in\mathbb{S}^3$. Moreover, $m^2\leq\tr M^2\leq m$ gives
\begin{equation*}
\frac{2}{3}\left(1-\frac{2}{\pi}\right)(1-m)\leq1-\left[\frac{2}{\pi}+\left(1-\frac{2}{\pi}\right)\tr M^2\right]^{\frac{2}{3}}\leq2\left(1-\frac{2}{\pi}\right)(1-m).
\end{equation*}

We first apply the classical Sobolev inequality to $u\sqrt{\frac{1\pm x}{2}}$, which belongs to $H^1(\mathbb{S}^3)$. Since
\begin{equation*}
\sum_\pm\frac{1\pm x}{2}=1,\qquad \sum_\pm\left|\nabla\sqrt{\frac{1\pm x}{2}}\right|^2=\frac{1}{4}\quad\text{a.e.},
\end{equation*}
we have
\begin{equation*}
\begin{aligned}
\alpha\sum_\pm\left\|u\sqrt{\frac{1\pm x}{2}}\right\|_6^2\leq&\sum_\pm\left[\alpha\left\|u\sqrt{\frac{1\pm x}{2}}\right\|_2^2+\beta\left\|\nabla\left(u\sqrt{\frac{1\pm x}{2}}\right)\right\|_2^2\right]\\
=&\left(\alpha+\frac{\beta}{4}\right)\|u\|_2^2+\beta\|\nabla u\|_2^2=a[1-(1-\beta)r].
\end{aligned}
\end{equation*}
On the other hand, the zero-barycenter condition gives
\begin{equation*}
\left\|u\sqrt{\frac{1\pm x}{2}}\right\|_6^2=\frac{1}{2}\left(1+3m\pm\int_{\mathbb{S}^3} x^3u^6\,d\sigma\right)^{\frac{1}{3}}.
\end{equation*}

Since the function $s\mapsto(1+3m+s)^{\frac{1}{3}}+(1+3m-s)^{\frac{1}{3}}$ decreases with $|s|$ for $|s|\leq m(1-m)$, using $|\int_{\mathbb{S}^3} x^3u^6\,d\sigma|\leq m(1-m)$ from \eqref{eq:axial-third-moment}, we obtain
\begin{equation}\label{eq:hemisphere-bound}
\frac{\alpha}{2}\left[(1+m)^{\frac{2}{3}}+(1+4m-m^2)^{\frac{1}{3}}\right]\leq a[1-(1-\beta)r].
\end{equation}

The following formula will be used to compute the kernel in the averaged Sobolev inequality.
\begin{lemma}\label{n3:green-kernel}
Let $P_2=1-\frac{4}{3}\Delta$ on $\mathbb{S}^3$. For every $f\in C^\infty(\mathbb{S}^3)$, we have
\begin{equation*}
P_2^{-1}f(\xi)=\frac{3\pi}{8}\int_{\mathbb{S}^3}\frac{f(\zeta)}{|\xi-\zeta|}\,d\sigma(\zeta).
\end{equation*}
\end{lemma}

\begin{proof}
For fixed $\zeta\in\mathbb{S}^3$, write $\theta=\arccos(\xi\cdot\zeta)$ and
\begin{equation*}
K_\zeta(\xi)=\frac{3\pi}{8|\xi-\zeta|}=\frac{3\pi}{16\sin\frac{\theta}{2}}.
\end{equation*}
The radial Laplacian gives
\begin{equation*}
\Delta K_\zeta=\left(\frac{d^2}{d\theta^2}+2\cot\theta\frac{d}{d\theta}\right)\frac{3\pi}{16\sin\frac{\theta}{2}}=\frac{3}{4}K_\zeta, \qquad 0<\theta<\pi.
\end{equation*}
Since $K_\zeta$ is smooth away from $\zeta$, we have $P_2K_\zeta=0$ on $\mathbb{S}^3\setminus\{\zeta\}$. Moreover,
\begin{equation*}
\lim_{\theta\to0}\theta K_\zeta=\frac{3\pi}{8},\qquad \lim_{\theta\to0}\theta^2\frac{\partial K_\zeta}{\partial\theta}=-\frac{3\pi}{8}.
\end{equation*}
For every $\varphi\in C^\infty(\mathbb{S}^3)$, Green's formula gives
\begin{equation*}
\int_{\mathbb{S}^3}K_\zeta P_2\varphi\,d\sigma=\lim_{\epsilon\to0}\frac{2}{3\pi^2}\int_{\partial B_\epsilon(\zeta)}\left(K_\zeta\frac{\partial\varphi}{\partial\theta}-\varphi\frac{\partial K_\zeta}{\partial\theta}\right)\,dS=\varphi(\zeta),
\end{equation*}
where $dS$ denotes the induced surface measure on $\partial B_\epsilon(\zeta)$ and we have used $|\partial B_\epsilon(\zeta)|=4\pi\sin^2\epsilon$. Thus $P_2K_\zeta=\delta_\zeta$ in the distributional sense, which proves the lemma.
\end{proof}

The following lemma gives upper bounds for \eqref{eq:angular-defect} in terms of $\int_{\mathbb{S}^3} u\,d\sigma$, $\int_{\mathbb{S}^3} u^5\,d\sigma$ and $\|H\|_2^2$.

\begin{lemma}\label{n3:inputs:green}
Every solution under consideration satisfies
\begin{equation}\label{eq:defect-moments}
\begin{aligned}
&1-\left[\frac{2}{\pi}+\left(1-\frac{2}{\pi}\right)\tr M^2\right]^{\frac{2}{3}}\leq\left[1-\frac{1}{\alpha}\left(\frac{1}{5}+\frac{64}{75\pi}\right)\right]a\left(\int_{\mathbb{S}^3} u^5\,d\sigma\right)^2\\
&\quad-\left[1+\frac{2}{5(1-r)^2}\right]\left(\sqrt a\int_{\mathbb{S}^3} u^5\,d\sigma-\frac{\int_{\mathbb{S}^3} u\,d\sigma}{\sqrt a}\right)^2.
\end{aligned}
\end{equation}
If $0<r\leq\frac{43}{64}$, we also have
\begin{equation}\label{eq:defect-variance}
\begin{aligned}
&1-\left[\frac{2}{\pi}+\left(1-\frac{2}{\pi}\right)\tr M^2\right]^{\frac{2}{3}}\leq\left[1-\frac{1}{\alpha}\left(\frac{1}{5}+\frac{64}{75\pi}\right)\right]a\left(\int_{\mathbb{S}^3} u^5\,d\sigma\right)^2\\
&\quad-\frac{22-9r}{16}\left(\sqrt a\int_{\mathbb{S}^3} u^5\,d\sigma-\frac{\int_{\mathbb{S}^3} u\,d\sigma}{\sqrt a}\right)^2-\frac{4r(1-r)}{3}\|H\|_2^2.
\end{aligned}
\end{equation}
For all $0<r<1$, we have
\begin{equation}\label{eq:defect-bound}
1-\left[\frac{2}{\pi}+\left(1-\frac{2}{\pi}\right)\tr M^2\right]^{\frac{2}{3}}<\frac{5}{16}r.
\end{equation}
In particular, $1-m<\frac{4r}{3}$ and $a>1-\frac{r}{2}$.
\end{lemma}

\begin{proof}
By \eqref{eq:critical-equation} and \eqref{eq:multiplier-bound}, we have
\begin{equation}\label{eq:integrated-equation}
\frac{|b|}{a}\leq1-r,\qquad 1+\frac{b\cdot\xi}{a}\geq r>0,\qquad \int_{\mathbb{S}^3} u\,d\sigma=a\int_{\mathbb{S}^3}\left(1+\frac{b\cdot\xi}{a}\right)u^5\,d\sigma.
\end{equation}

For every $h\in L^{\frac{6}{5}}(\mathbb{S}^3)$, the dual form of the classical Sobolev inequality \eqref{eq:sobolev} gives
\begin{equation*}
\langle h,P_2^{-1}h\rangle\leq\|h\|_{\frac{6}{5}}^2.
\end{equation*}
For a unit vector $v\in\mathbb{S}^3$, write $y=v\cdot\xi$. By H\"older's inequality and $\int_{\mathbb{S}^3} y_\pm u^6\,d\sigma=\frac{1}{2}\int_{\mathbb{S}^3} |y|u^6\,d\sigma$, we have
\begin{equation*}
\begin{aligned}
&\sum_\pm\langle y_\pm u^5,P_2^{-1}(y_\pm u^5)\rangle \leq\sum_\pm\|y_\pm u^5\|_{\frac{6}{5}}^2 \leq\sum_\pm\left(\int_{\mathbb{S}^3} y_\pm u^6\,d\sigma\right)^{\frac{4}{3}}\left(\int_{\mathbb{S}^3} y_\pm^2u^6\,d\sigma\right)^{\frac{1}{3}}\\
\leq&\alpha\left(\int_{\mathbb{S}^3} |y|u^6\,d\sigma\right)^{\frac{4}{3}}\left(\int_{\mathbb{S}^3} y^2u^6\,d\sigma\right)^{\frac{1}{3}}.
\end{aligned}
\end{equation*}

Using $\int_{\mathbb{S}^3}\int_{\mathbb{S}^3} y^2u^6\,d\sigma(\xi)\,d\sigma(v)=\frac{1}{4}$ and H\"older's inequality in $v$, we obtain the first inequality in
\begin{equation}\label{eq:averaged-sobolev}
\frac{16}{3\beta}\int_{\mathbb{S}^3}\sum_\pm\langle y_\pm u^5,P_2^{-1}(y_\pm u^5)\rangle\,d\sigma(v) \leq\left[4\int_{\mathbb{S}^3}\left(\int_{\mathbb{S}^3} |y|u^6\,d\sigma(\xi)\right)^2\,d\sigma(v)\right]^{\frac{2}{3}}\leq\left[\frac{2}{\pi}+\left(1-\frac{2}{\pi}\right)\tr M^2\right]^{\frac{2}{3}}.
\end{equation}

Let $\theta$ be the angle between $\xi$ and $\zeta$. Averaging first over the polar angle of the projection of $v$ onto their plane and using that the mean squared length of this projection is $\frac{1}{2}$, we have
\begin{equation*}
4\int_{\mathbb{S}^3}|v\cdot\xi||v\cdot\zeta|\,d\sigma(v)=\frac{1}{\pi}\int_0^{2\pi}|\cos\varphi\cos(\varphi-\theta)|\,d\varphi=\frac{2}{\pi}\left[\sin\theta+\left(\frac{\pi}{2}-\theta\right)\cos\theta\right].
\end{equation*}
write $s=(\xi\cdot\zeta)^2$. By convexity, we have
\begin{equation*}
4\int_{\mathbb{S}^3}|v\cdot\xi||v\cdot\zeta|\,d\sigma(v)=\frac{2}{\pi}\left(\sqrt{1-s}+\sqrt{s}\arcsin\sqrt{s}\right)\leq\frac{2}{\pi}+\left(1-\frac{2}{\pi}\right)s.
\end{equation*}
Integrating against $u^6(\xi)u^6(\zeta)\,d\sigma(\xi)\,d\sigma(\zeta)$ yields the last inequality in \eqref{eq:averaged-sobolev}.

Using $A_+B_++A_-B_-=\frac{1}{2}(|A||B|+AB)$ for real $A,B$ and
\begin{equation*}
\int_{\mathbb{S}^3}(v\cdot\xi)(v\cdot\zeta)\,d\sigma(v)=\frac{\cos\theta}{4},
\end{equation*}
we have
\begin{equation*}
\sum_\pm\int_{\mathbb{S}^3}(v\cdot\xi)_\pm(v\cdot\zeta)_\pm\,d\sigma(v)=\frac{\sin\theta+(\pi-\theta)\cos\theta}{4\pi}.
\end{equation*}

By Lemma~\ref{n3:green-kernel}, the integral operator corresponding to the quadratic form on the left-hand side of \eqref{eq:averaged-sobolev} has kernel
\begin{equation*}
\frac{\sin\theta+(\pi-\theta)\cos\theta}{2\beta|\xi-\zeta|}=\frac{\sin\theta+(\pi-\theta)\cos\theta}{4\beta\sin\frac{\theta}{2}}.
\end{equation*}

Let $\kappa_j$ denote its eigenvalue on spherical harmonics of degree $j$. The normalized zonal harmonic is $\frac{\sin((j+1)\theta)}{(j+1)\sin\theta}$ and the radial measure is $\frac{2}{\pi}\sin^2\theta\,d\theta$. Thus
\begin{equation*}
\kappa_j=\frac{1}{\beta\pi(j+1)}\int_0^\pi[\sin\theta+(\pi-\theta)\cos\theta]\cos\left(\frac{\theta}{2}\right)\sin((j+1)\theta)\,d\theta.
\end{equation*}

Using $(\frac{d^2}{d\theta^2}+1)[\sin\theta+(\pi-\theta)\cos\theta]=2\sin\theta$, integration by parts gives
\begin{equation*}
\int_0^\pi[\sin\theta+(\pi-\theta)\cos\theta]\sin(s\theta)\,d\theta=\frac{\pi s}{s^2-1}+\frac{2\sin(\pi s)}{(s^2-1)^2},\qquad s=j+1\pm\frac{1}{2}.
\end{equation*}
It follows that
\begin{equation}\label{eq:spherical-eigenvalues}
\kappa_j=\frac{1}{2\beta(j+1)}\sum_{\epsilon=\pm1}\left[ \frac{j+1+\frac{\epsilon}{2}}{(j+1+\frac{\epsilon}{2})^2-1} +\frac{2\sin(\pi(j+1+\frac{\epsilon}{2}))}{\pi[(j+1+\frac{\epsilon}{2})^2-1]^2}\right].
\end{equation}
In particular,
\begin{equation*}
\kappa_0=\frac{1}{\alpha}\left(\frac{1}{5}+\frac{64}{75\pi}\right),\qquad \kappa_1=\frac{1}{4\beta}\left[\frac{6}{5}+\frac{10}{21}-\frac{32}{\pi}\left(\frac{1}{25}-\frac{1}{21^2}\right)\right].
\end{equation*}

For $j=0,1$, we have
\begin{equation}\label{eq:low-degrees}
\frac{6}{25}<1-\kappa_0,\qquad (1-\kappa_0)\left(\frac{7}{5}+\frac{5}{16}\right)<\frac{7}{16},\qquad \kappa_1-\frac{1}{1+3\beta}>\frac{1}{10}.
\end{equation}

For $j\geq2$, we have
\begin{equation}\label{eq:higher-degrees}
1+\beta j(j+2)-\kappa_j^{-1}>1.
\end{equation}
Indeed, with $k=j+1$, \eqref{eq:spherical-eigenvalues} gives
\begin{equation*}
\beta \kappa_{k-1}-\frac{1}{k^2-1}\geq\frac{4k^2+11}{(k^2-1)(4k^2-9)(4k^2-1)}-\frac{2^8(4k^2-3)}{3(4k^2-9)^2(4k^2-1)^2}>0.
\end{equation*}

Let $f_j$ and $u_j$ be the orthogonal projections of $u^5$ and $u$ onto the spherical harmonics of degree $j$. Using $\int_{\mathbb{S}^3} u^6\,d\sigma=1$ and $\int_{\mathbb{S}^3} u(1-\beta\Delta)u\,d\sigma=a$, we obtain from \eqref{eq:averaged-sobolev}
\begin{equation}\label{eq:harmonic-decomposition}
\begin{aligned}
1-\left[\frac{2}{\pi}+\left(1-\frac{2}{\pi}\right)\tr M^2\right]^{\frac{2}{3}}&\leq(1-\kappa_0)a\left(\int_{\mathbb{S}^3} u^5\,d\sigma\right)^2-\left(\sqrt a\int_{\mathbb{S}^3} u^5\,d\sigma-\frac{\int_{\mathbb{S}^3} u\,d\sigma}{\sqrt a}\right)^2\\
&\quad-\sum_{j\geq1}\left[a\kappa_j\|f_j\|_2^2-2\langle f_j,u_j\rangle+\frac{1+\beta j(j+2)}{a}\|u_j\|_2^2\right].
\end{aligned}
\end{equation}

For each $j\geq1$, we have
\begin{equation}\label{eq:harmonic-square}
a\kappa_j\|f_j\|_2^2-2\langle f_j,u_j\rangle+\frac{1+\beta j(j+2)}{a}\|u_j\|_2^2=\kappa_j\left\|\sqrt a f_j-\frac{u_j}{\sqrt a \kappa_j}\right\|_2^2+\frac{1+\beta j(j+2)-\kappa_j^{-1}}{a}\|u_j\|_2^2.
\end{equation}

Since $2\xi_1,\ldots,2\xi_4$ form an orthonormal basis of the degree-one harmonics, \eqref{eq:integrated-equation} gives
\begin{equation}\label{eq:first-harmonic-moment}
\left|\sqrt a\int_{\mathbb{S}^3} u^5\,d\sigma-\frac{\int_{\mathbb{S}^3} u\,d\sigma}{\sqrt a}\right|=\sqrt a\left|\frac{b}{a}\cdot\int_{\mathbb{S}^3}\xi u^5\,d\sigma\right|\leq\frac{\sqrt a(1-r)}{2}\|f_1\|_2.
\end{equation}

Using \eqref{eq:low-degrees}, we obtain
\begin{equation*}
a\kappa_1\|f_1\|_2^2-2\langle f_1,u_1\rangle+\frac{1+3\beta}{a}\|u_1\|_2^2=\frac{1+3\beta}{a}\left\|u_1-\frac{a}{1+3\beta}f_1\right\|_2^2+a\left(\kappa_1-\frac{1}{1+3\beta}\right)\|f_1\|_2^2\geq\frac{a}{10}\|f_1\|_2^2.
\end{equation*}
It follows from \eqref{eq:first-harmonic-moment} that
\begin{equation*}
\frac{a}{10}\|f_1\|_2^2\geq\frac{2}{5(1-r)^2}\left(\sqrt a\int_{\mathbb{S}^3} u^5\,d\sigma-\frac{\int_{\mathbb{S}^3} u\,d\sigma}{\sqrt a}\right)^2.
\end{equation*}

For $j\geq2$, each summand in \eqref{eq:harmonic-decomposition} is nonnegative by \eqref{eq:higher-degrees} and \eqref{eq:harmonic-square}. The lower bound for the degree-one summand in \eqref{eq:harmonic-decomposition} now gives \eqref{eq:defect-moments}.

To prove \eqref{eq:defect-variance}, let $r\leq\frac{43}{64}$. We have
\begin{equation*}
\begin{aligned}
&a\kappa_1\|f_1\|_2^2-2\langle f_1,u_1\rangle+\frac{1+3\beta}{a}\|u_1\|_2^2=\frac{4r}{3a}\|u_1\|_2^2+a\left[\kappa_1-\frac{1}{1+3\beta-\frac{4r}{3}}\right]\|f_1\|_2^2\\
&\quad+\frac{1+3\beta-\frac{4r}{3}}{a}\left\|u_1-\frac{a}{1+3\beta-\frac{4r}{3}}f_1\right\|_2^2.
\end{aligned}
\end{equation*}
It follows from \eqref{eq:low-degrees} that
\begin{equation*}
\kappa_1-\frac{1}{1+3\beta-\frac{4r}{3}}>\frac{1}{10} -\frac{\frac{4r}{3}}{(1+3\beta)(1+3\beta-\frac{4r}{3})}>0.
\end{equation*}

For $j\geq2$, \eqref{eq:higher-degrees} and \eqref{eq:harmonic-square} give a lower bound of $\frac{4r}{3a}\|u_j\|_2^2$ for each summand in \eqref{eq:harmonic-decomposition}.

If $r\leq\frac{2}{3}$, concavity and the endpoint values at $0$ and $\frac{2}{3}$ give
\begin{equation*}
\kappa_1-\frac{1}{1+3\beta-\frac{4r}{3}}>\frac{1}{10}\left(1-\frac{3r}{2}\right).
\end{equation*}
Together with \eqref{eq:first-harmonic-moment}, this yields
\begin{equation*}
\begin{aligned}
&a\left[\kappa_1-\frac{1}{1+3\beta-\frac{4r}{3}}\right]\|f_1\|_2^2\geq\frac{2}{5(1-r)^2}\left(1-\frac{3r}{2}\right)\left(\sqrt a\int_{\mathbb{S}^3} u^5\,d\sigma-\frac{\int_{\mathbb{S}^3} u\,d\sigma}{\sqrt a}\right)^2\\
\geq&\frac{6-9r}{16}\left(\sqrt a\int_{\mathbb{S}^3} u^5\,d\sigma-\frac{\int_{\mathbb{S}^3} u\,d\sigma}{\sqrt a}\right)^2.
\end{aligned}
\end{equation*}
If $r>\frac{2}{3}$, we also have
\begin{equation*}
a\left[\kappa_1-\frac{1}{1+3\beta-\frac{4r}{3}}\right]\|f_1\|_2^2\geq0\geq\frac{6-9r}{16}\left(\sqrt a\int_{\mathbb{S}^3} u^5\,d\sigma-\frac{\int_{\mathbb{S}^3} u\,d\sigma}{\sqrt a}\right)^2.
\end{equation*}

Combining these estimates and using $\frac{1}{a}\sum_{j\geq1}\|u_j\|_2^2=(1-r)\|H\|_2^2$, we obtain
\begin{equation*}
\sum_{j\geq1}\left[a\kappa_j\|f_j\|_2^2-2\langle f_j,u_j\rangle+\frac{1+\beta j(j+2)}{a}\|u_j\|_2^2\right]\geq\frac{4r(1-r)}{3}\|H\|_2^2 +\frac{6-9r}{16}\left(\sqrt a\int_{\mathbb{S}^3} u^5\,d\sigma-\frac{\int_{\mathbb{S}^3} u\,d\sigma}{\sqrt a}\right)^2.
\end{equation*}
It follows from \eqref{eq:harmonic-decomposition} that \eqref{eq:defect-variance} holds.

For the unrestricted estimate \eqref{eq:defect-bound}, equations \eqref{eq:defect-moments} and \eqref{eq:low-degrees} give
\begin{equation*}
1-\left[\frac{2}{\pi}+\left(1-\frac{2}{\pi}\right)\tr M^2\right]^{\frac{2}{3}}\leq\frac{(1-\kappa_0)(\frac{7}{5})}{\frac{7}{5}-(1-\kappa_0)}\frac{(\int_{\mathbb{S}^3} u\,d\sigma)^2}{a}<\frac{5}{16}r.
\end{equation*}
Since $\tr M^2\leq m$, we have
\begin{equation*}
\frac{6}{25}(1-m)\leq1-\left[\frac{2}{\pi}+\left(1-\frac{2}{\pi}\right)\tr M^2\right]^{\frac{2}{3}}<\frac{5}{16}r.
\end{equation*}
Thus $1-m<\frac{4r}{3}$.

By concavity in $m$, the left-hand side of \eqref{eq:hemisphere-bound} satisfies
\begin{equation*}
\frac{\alpha}{2}\big[(1+m)^{\frac{2}{3}}+(1+4m-m^2)^{\frac{1}{3}}\big] \geq\alpha(1-m)+m=1-(1-\alpha)(1-m).
\end{equation*}
Combining this with \eqref{eq:hemisphere-bound} and $1-m<\frac{4r}{3}$, we obtain
\begin{equation}\label{eq:energy-lower-three}
a>1-\frac{r}{2},\qquad1\leq a^{-\frac{3}{4}}\leq(1-\frac{r}{2})^{-1}\leq1+r\leq1+2r.
\end{equation}

\end{proof}

We next estimate $\int_{\mathbb{S}^3} u\,d\sigma$ and $\int_{\mathbb{S}^3} u^5\,d\sigma$. In particular, the following lemma gives a lower bound for $\|H\|_2^2$.

\begin{lemma}\label{n3:inputs:normalized-H}
The normalized function $H$ in \eqref{eq:normalized-deviation} satisfies
\begin{equation}\label{eq:variance-bounds}
\frac{r^2}{1+3r}<\|H\|_2^2=\frac{r-\frac{(\int_{\mathbb{S}^3} u\,d\sigma)^2}{a}}{1-r}\leq\frac{1}{3\beta},
\end{equation}
and
\begin{equation}\label{eq:harmonic-variance}
\frac{1}{8\beta}+\frac{5}{8}\|H_1\|_2^2-\|H\|_2^2=\frac{1}{8}\sum_{j\geq3}[j(j+2)-8]\|H_j\|_2^2\geq0.
\end{equation}
The first and fifth moments also satisfy
\begin{equation}\label{eq:mean-bounds}
r(2+3r)\sqrt a\int_{\mathbb{S}^3} u^5\,d\sigma+(1+5r+9r^2)\frac{\int_{\mathbb{S}^3} u\,d\sigma}{\sqrt a} \leq(1+4r)(1+r)\sqrt{r(1+3r)},\quad \frac{\int_{\mathbb{S}^3} u\,d\sigma}{\sqrt a}<(1+r)\sqrt{\frac{r}{1+3r}}.
\end{equation}
\end{lemma}

\begin{proof}
To estimate $\int_{\mathbb{S}^3} u^5\,d\sigma$, write $\tau=3+\frac{1}{r}$ and define
\begin{equation}\label{eq:contact-polynomial}
F(s)=\frac{[(3\tau-2)s^2+2\sqrt\tau(\tau+1)s+\tau^2+4](s^2-\sqrt\tau s+1)^2}{2\sqrt\tau(2\tau-3)}\geq0,\qquad s\geq0.
\end{equation}
We have
\begin{equation}\label{eq:contact-coefficients}
F(s)=\frac{\tau^2+4+(\tau^3-2\tau^2+3\tau+6)s^2+(3\tau-2)s^6}{2\sqrt\tau(2\tau-3)}-\frac{\tau^2-\tau+3}{2\tau-3}s-s^5.
\end{equation}
Set $f=a^{-\frac{3}{4}}F(a^{\frac{1}{4}}u)\geq0$. Since $r=\frac{1}{\tau-3}$, we have
\begin{equation*}
\frac{\tau^3+\tau^2-8\tau+12}{(\tau-3)(\tau^2+4)}-(1+2r) =\frac{2(\tau-4)(\tau-2)}{(\tau-3)(\tau^2+4)}\geq0.
\end{equation*}
Together with \eqref{eq:energy-lower-three}, this implies
\begin{equation}\label{eq:contact-constant}
\frac{(\tau^2+4)a^{-\frac{3}{4}}+\frac{\tau^3+\tau^2-8\tau+12}{\tau-3}a^{\frac{3}{4}}}{2\sqrt\tau(2\tau-3)}-\frac{\tau(\tau+1)}{2\tau-3}(1+r)\sqrt{\frac{r}{1+3r}}=\frac{a^{-\frac{3}{4}}-1}{2\sqrt\tau(2\tau-3)}\left[\tau^2+4-\frac{\tau^3+\tau^2-8\tau+12}{\tau-3}a^{\frac{3}{4}}\right]\leq0.
\end{equation}

By \eqref{eq:contact-coefficients}, we have
\begin{equation*}
0\leq\int_{\mathbb{S}^3} f\,d\sigma=\frac{(\tau^2+4)a^{-\frac{3}{4}}+[(\tau^3-2\tau^2+3\tau+6)r+3\tau-2]a^{\frac{3}{4}}}{2\sqrt\tau(2\tau-3)} -\frac{\tau^2-\tau+3}{2\tau-3}\frac{\int_{\mathbb{S}^3} u\,d\sigma}{\sqrt a}-\sqrt a\int_{\mathbb{S}^3} u^5\,d\sigma.
\end{equation*}
It follows from \eqref{eq:contact-constant} that
\begin{equation*}
\sqrt a\int_{\mathbb{S}^3} u^5\,d\sigma+\frac{\tau^2-\tau+3}{2\tau-3}\frac{\int_{\mathbb{S}^3} u\,d\sigma}{\sqrt a} \leq\frac{\tau(\tau+1)}{2\tau-3}(1+r)\sqrt{\frac{r}{1+3r}}.
\end{equation*}
This proves the first inequality in \eqref{eq:mean-bounds}.

Testing \eqref{eq:critical-equation} against $\xi$ gives
\begin{equation*}
(1+3\beta)\frac{\int_{\mathbb{S}^3}\xi u\,d\sigma}{\sqrt a}=\sqrt a\int_{\mathbb{S}^3}\xi u^5\,d\sigma+\left(\sqrt a\int_{\mathbb{S}^3}\xi\xi^Tu^5\,d\sigma\right)\frac{b}{a}.
\end{equation*}
Combining this with \eqref{eq:first-moment} and \eqref{eq:contact-coefficients}, we obtain
\begin{equation}\label{eq:contact-first-moment}
\left[\frac{\tau(\tau+1)}{2\tau-3}+3\beta\right]\frac{\int_{\mathbb{S}^3}\xi u\,d\sigma}{\sqrt a}+\int_{\mathbb{S}^3}\xi f\,d\sigma =\left[\sqrt a\int_{\mathbb{S}^3}\xi\xi^Tu^5\,d\sigma-\frac{(\tau^3-2\tau^2+3\tau+6)a^{\frac{3}{4}}}{12(1-\alpha)\sqrt\tau(2\tau-3)}(I-M)\right]\frac{b}{a}.
\end{equation}

Testing \eqref{eq:critical-equation} against $b\cdot\xi$ and using \eqref{eq:contact-first-moment}, we have
\begin{equation*}
\frac{b}{a}\cdot\int_{\mathbb{S}^3}\xi f\,d\sigma =-\frac{\tau^2-\tau+3}{2\tau-3}\frac{b}{a}\cdot\frac{\int_{\mathbb{S}^3}\xi u\,d\sigma}{\sqrt a}-\sqrt a\int_{\mathbb{S}^3}\frac{b\cdot\xi}{a}u^5\,d\sigma -\frac{(\tau^3-2\tau^2+3\tau+6)a^{\frac{3}{4}}}{12(1-\alpha)\sqrt\tau(2\tau-3)}\left(\frac{b}{a}\right)^T(I-M)\frac{b}{a}.
\end{equation*}
Using \eqref{eq:integrated-equation} and \eqref{eq:critical-equation}, we obtain
\begin{equation}\label{eq:weighted-contact}
\begin{aligned}
&\int_{\mathbb{S}^3}\left[\frac{\tau^2-\tau+3}{2\tau-3}+(1+3\beta)\left(1+\frac{b\cdot\xi}{a}\right)\right]f\,d\sigma =\left[\frac{\tau(\tau+1)}{2\tau-3}+3\beta\right]\left[\int_{\mathbb{S}^3} f\,d\sigma+\sqrt a\int_{\mathbb{S}^3} u^5\,d\sigma-\frac{\int_{\mathbb{S}^3} u\,d\sigma}{\sqrt a}\right]\\
&\quad-\frac{(1+3\beta)(\tau^3-2\tau^2+3\tau+6)a^{\frac{3}{4}}}{12(1-\alpha)\sqrt\tau(2\tau-3)}\left(\frac{b}{a}\right)^T(I-M)\frac{b}{a} -\frac{\tau^2-\tau+3}{2\tau-3}\sqrt a\int_{\mathbb{S}^3}\left(\frac{b\cdot\xi}{a}\right)^2u^5\,d\sigma.
\end{aligned}
\end{equation}

It follows from \eqref{eq:contact-coefficients} and \eqref{eq:contact-constant} that
\begin{equation*}
\int_{\mathbb{S}^3} f\,d\sigma+\sqrt a\int_{\mathbb{S}^3} u^5\,d\sigma-\frac{\int_{\mathbb{S}^3} u\,d\sigma}{\sqrt a}\leq\frac{\tau(\tau+1)}{2\tau-3}\left[(1+r)\sqrt{\frac{r}{1+3r}}-\frac{\int_{\mathbb{S}^3} u\,d\sigma}{\sqrt a}\right].
\end{equation*}
Combining this with \eqref{eq:integrated-equation} and \eqref{eq:weighted-contact}, we have
\begin{equation}\label{eq:contact-bound}
\left|\int_{\mathbb{S}^3}\xi f\,d\sigma\right|\leq\int_{\mathbb{S}^3} f\,d\sigma\leq\frac{\frac{\tau(\tau+1)}{2\tau-3}\left[\frac{\tau(\tau+1)}{2\tau-3}+3\beta\right]}{\frac{\tau^2-\tau+3}{2\tau-3}+(1+3\beta)r} \left[(1+r)\sqrt{\frac{r}{1+3r}}-\frac{\int_{\mathbb{S}^3} u\,d\sigma}{\sqrt a}\right].
\end{equation}

By \eqref{eq:contact-bound}, if
\begin{equation*}
\frac{\int_{\mathbb{S}^3} u\,d\sigma}{\sqrt a}=(1+r)\sqrt{\frac{r}{1+3r}}
\end{equation*}
then we have $f\equiv0$. It follows from \eqref{eq:contact-polynomial} that
\begin{equation*}
a^{\frac{1}{2}}u^2-\sqrt\tau\,a^{\frac{1}{4}}u+1=0\text{ on }\mathbb{S}^3.
\end{equation*}
Then $u$ must be constant. This contradiction proves the strict inequality in \eqref{eq:mean-bounds}.

Consequently, we have
\begin{equation*}
\|H\|_2^2 =\frac{r-\frac{(\int_{\mathbb{S}^3} u\,d\sigma)^2}{a}}{1-r} >\frac{r-\frac{r(1+r)^2}{1+3r}}{1-r} =\frac{r^2}{1+3r}.
\end{equation*}

Finally, since $\beta\sum_{j\geq1}j(j+2)\|H_j\|_2^2=1$, orthogonality give
\begin{equation*}
\|H\|_2^2=\sum_{j\geq1}\|H_j\|_2^2 \leq\frac{1}{3}\sum_{j\geq1}j(j+2)\|H_j\|_2^2=\frac{1}{3\beta},\quad \frac{1}{8\beta}+\frac{5}{8}\|H_1\|_2^2-\|H\|_2^2=\frac{1}{8}\sum_{j\geq3}[j(j+2)-8]\|H_j\|_2^2\geq0,
\end{equation*}
which prove the upper bound in \eqref{eq:variance-bounds} and \eqref{eq:harmonic-variance}.
\end{proof}

We next use \eqref{eq:variance-bounds} and \eqref{eq:harmonic-variance} to exclude $8\beta r^2\geq1+3r$.

\begin{lemma}\label{n3:inputs:upper-cutoff}
Every nonconstant solution under consideration satisfies
\begin{equation}\label{eq:upper-cutoff}
r<\frac{3+\sqrt{9+32\beta}}{16\beta},\qquad \frac{2}{3}<\frac{3+\sqrt{9+32\beta}}{16\beta}<\frac{43}{64}.
\end{equation}
\end{lemma}

\begin{proof}
Using $\beta^3=\frac{16}{27}$, one can check that the positive root of $8\beta r^2-3r-1=0$ lies in $(\frac{2}{3},\frac{43}{64})$.

Suppose that $8\beta r^2\geq1+3r$. Then $r>\frac{2}{3}$ and \eqref{eq:variance-bounds} gives
\begin{equation*}
\|H\|_2^2>\frac{r^2}{1+3r}\geq\frac{1}{8\beta}.
\end{equation*}
Thus $H_1\ne0$ by \eqref{eq:harmonic-variance} and
\begin{equation}\label{eq:first-harmonic-gap}
\left[\frac{1}{8\beta}+\frac{5}{8}\|H_1\|_2^2-\|H\|_2^2\right]+\left[\|H\|_2^2-\frac{r^2}{1+3r}\right] \leq\frac{5}{8}\|H_1\|_2^2.
\end{equation}
We will prove that
\begin{equation}\label{eq:first-harmonic-bound}
\|H_1\|_2\leq\frac{2}{3}\left[\frac{1}{8\beta}+\frac{5}{8}\|H_1\|_2^2-\|H\|_2^2\right]^{\frac{1}{2}}+\frac{3}{2}\left[\|H\|_2^2-\frac{r^2}{1+3r}\right].
\end{equation}

Recall that $\tau=3+\frac{1}{r}$ and $f=a^{-\frac{3}{4}}F(a^{\frac{1}{4}}u)$, where $F$ is defined in \eqref{eq:contact-polynomial}. Since $4<\tau<\frac{9}{2}$, we have
\begin{equation*}
3\leq\frac{\tau^2-\tau+3}{2\tau-3}\leq\frac{25}{8},\qquad \frac{\tau^3-2\tau^2+3\tau+6}{2\sqrt\tau(2\tau-3)}<3,\qquad \frac{\tau(\tau+1)}{2\tau-3}+3\beta>\frac{13}{2}.
\end{equation*}
Using these bounds and \eqref{eq:contact-bound}, we obtain
\begin{equation}\label{eq:contact-estimates}
\left\|\sqrt a\int_{\mathbb{S}^3}\xi\xi^Tu^5\,d\sigma-\frac{(\tau^3-2\tau^2+3\tau+6)a^{\frac{3}{4}}}{12(1-\alpha)\sqrt\tau(2\tau-3)}(I-M)\right\|_{\op}\leq\frac{3}{2},\quad \left|\int_{\mathbb{S}^3}\xi f\,d\sigma\right|\leq\frac{21}{4}\left[(1+r)\sqrt{\frac{r}{1+3r}}-\frac{\int_{\mathbb{S}^3} u\,d\sigma}{\sqrt a}\right].
\end{equation}
Indeed, $0\leq\sqrt a\int_{\mathbb{S}^3}\xi\xi^Tu^5\,d\sigma\leq I$ and $0\leq I-M\leq I$. The coefficient of $I-M$ in the difference is at most $\frac{3}{2}$ in absolute value. For the second inequality, we use \eqref{eq:contact-bound} and
\begin{equation*}
\frac{\frac{\tau(\tau+1)}{2\tau-3}}{1-\frac{1+3\beta}{\frac{\tau(\tau+1)}{2\tau-3}+3\beta}(1-r)}<\frac{21}{4}.
\end{equation*}

Using \eqref{eq:first-moment}, \eqref{eq:gradient-first-moment} and \eqref{eq:normalized-deviation}, we have
\begin{equation*}
(1-r)\beta\int_{\mathbb{S}^3}\xi(|\nabla H|^2-4H^2)\,d\sigma=\left[I+\frac{5(2\beta-1)}{6(1-\alpha)}(I-M)\right]\frac{b}{a}+8\beta\frac{\int_{\mathbb{S}^3} u\,d\sigma}{\sqrt a}\frac{\int_{\mathbb{S}^3}\xi u\,d\sigma}{\sqrt a}.
\end{equation*}
It follows from \eqref{eq:contact-first-moment} that
\begin{equation*}
\begin{aligned}
&(1-r)\beta\int_{\mathbb{S}^3}\xi(|\nabla H|^2-4H^2)\,d\sigma+\frac{8\beta(2\tau-3)}{\tau(\tau+1)+3\beta(2\tau-3)}\frac{\int_{\mathbb{S}^3} u\,d\sigma}{\sqrt a}\int_{\mathbb{S}^3}\xi f\,d\sigma\\
=&\Biggl\{I+\left[\frac{6\beta-5}{6(1-\alpha)}+\frac{2\beta}{3(1-\alpha)}\left(1-\frac{(\tau^3-2\tau^2+3\tau+6)a^{\frac{3}{4}}}{\sqrt\tau[\tau(\tau+1)+3\beta(2\tau-3)]}\frac{\int_{\mathbb{S}^3} u\,d\sigma}{\sqrt a}\right)\right](I-M)\\
&+\frac{8\beta(2\tau-3)\int_{\mathbb{S}^3} u\,d\sigma}{\tau(\tau+1)+3\beta(2\tau-3)}\int_{\mathbb{S}^3}\xi\xi^Tu^5\,d\sigma\Biggr\}\frac{b}{a}.
\end{aligned}
\end{equation*}
We have $\frac{\int_{\mathbb{S}^3} u\,d\sigma}{\sqrt a}\leq1$ and
\begin{equation*}
\frac{(\tau^3-2\tau^2+3\tau+6)a^{\frac{3}{4}}}{\sqrt\tau[\tau(\tau+1)+3\beta(2\tau-3)]}<\frac{12}{13}.
\end{equation*}
Thus the coefficient of $I-M$ is positive and the coefficient matrix of $\frac{b}{a}$ is bounded below by $I$. Together with \eqref{eq:contact-first-moment}, \eqref{eq:contact-estimates} and $\|H_1\|_2=\frac{2|\int_{\mathbb{S}^3}\xi u\,d\sigma|}{\sqrt{a(1-r)}}$, this gives
\begin{equation}\label{eq:multiplier-first-harmonic}
\begin{aligned}
\frac{|b|}{a}&\leq(1-r)\beta\left|\int_{\mathbb{S}^3}\xi(|\nabla H|^2-4H^2)\,d\sigma\right|+\frac{8\beta(2\tau-3)}{\tau(\tau+1)+3\beta(2\tau-3)}\frac{\int_{\mathbb{S}^3} u\,d\sigma}{\sqrt a}\left|\int_{\mathbb{S}^3}\xi f\,d\sigma\right|,\\
\|H_1\|_2&\leq\frac{2(2\tau-3)}{[\tau(\tau+1)+3\beta(2\tau-3)]\sqrt{1-r}} \left[\frac{3}{2}\frac{|b|}{a}+\left|\int_{\mathbb{S}^3}\xi f\,d\sigma\right|\right].
\end{aligned}
\end{equation}

To estimate $\int_{\mathbb{S}^3}\xi(|\nabla H|^2-4H^2)\,d\sigma$ in \eqref{eq:multiplier-first-harmonic}, we prove that
\begin{equation}\label{eq:adjacent-harmonics}
\left|\int_{\mathbb{S}^3}\xi(|\nabla H|^2-4H^2)\,d\sigma\right|\leq\left[\beta^{-1}-2\|H\|_2^2-\|H_1\|_2^2\right]^{\frac{1}{2}}\left[8\left(\frac{1}{8\beta}+\frac{5}{8}\|H_1\|_2^2-\|H\|_2^2\right)\right]^{\frac{1}{2}}.
\end{equation}

For a unit vector $v$, the product $(v\cdot\xi)H_j$ contains only spherical harmonics of degrees $j-1$ and $j+1$. In polar coordinates with axis $v$, the degree-$j$ spherical harmonics of transverse degree $k$ are proportional to $(1-(v\cdot\xi)^2)^{\frac{k}{2}}C_{j-k}^{k+1}(v\cdot\xi)$ times degree-$k$ spherical harmonics on $\mathbb{S}^2$, where $C_l^{\nu}$ denotes the Gegenbauer polynomial. Using the recurrence
\begin{equation*}
2(l+\nu)sC_l^\nu(s)=(l+1)C_{l+1}^\nu(s)+(l+2\nu-1)C_{l-1}^\nu(s)
\end{equation*}
and
\begin{equation*}
\frac{\int_{-1}^1(C_{l+1}^{\nu})^2(1-s^2)^{\nu-\frac{1}{2}}\,ds}{\int_{-1}^1(C_l^{\nu})^2(1-s^2)^{\nu-\frac{1}{2}}\,ds}=\frac{(l+2\nu)(l+\nu)}{(l+1)(l+\nu+1)}
\end{equation*}
with $l=j-k$ and $\nu=k+1$, we obtain the coefficient from degree $j$ to degree $j+1$ in the orthonormal basis
\begin{equation*}
\frac{1}{2}\left[1-\frac{k(k+1)}{(j+1)(j+2)}\right]^{\frac{1}{2}}\leq\frac{1}{2},\qquad 0\leq k\leq j.
\end{equation*}
Hence
\begin{equation*}
\left|\int_{\mathbb{S}^3}(v\cdot\xi)H_jH_{j+1}\,d\sigma\right|\leq\frac{1}{2}\|H_j\|_2\|H_{j+1}\|_2.
\end{equation*}

Integration by parts gives
\begin{equation*}
\int_{\mathbb{S}^3}(v\cdot\xi)\nabla H_j\cdot\nabla H_{j+1}\,d\sigma =j(j+3)\int_{\mathbb{S}^3}(v\cdot\xi)H_jH_{j+1}\,d\sigma.
\end{equation*}
Thus
\begin{equation*}
\int_{\mathbb{S}^3}(v\cdot\xi)(|\nabla H|^2-4H^2)\,d\sigma=2\sum_{j\geq1}[j(j+3)-4]\int_{\mathbb{S}^3}(v\cdot\xi)H_jH_{j+1}\,d\sigma =2\sum_{j\geq2}[j(j+3)-4]\int_{\mathbb{S}^3}(v\cdot\xi)H_jH_{j+1}\,d\sigma.
\end{equation*}
Using
\begin{equation*}
[j(j+2)-2][(j+1)(j+3)-8]-[j(j+3)-4]^2=6(j-1)>0\qquad\text{if }j\geq2,
\end{equation*}
and the Cauchy--Schwarz inequality, we have
\begin{equation*}
\begin{aligned}
\left|\int_{\mathbb{S}^3}(v\cdot\xi)(|\nabla H|^2-4H^2)\,d\sigma\right|\leq&\sum_{j\geq2}[j(j+3)-4]\|H_j\|_2\|H_{j+1}\|_2\\
\leq& \left[\sum_{j\geq2}[j(j+2)-2]\|H_j\|_2^2\right]^{\frac{1}{2}} \left[\sum_{j\geq3}[j(j+2)-8]\|H_j\|_2^2\right]^{\frac{1}{2}}.
\end{aligned}
\end{equation*}
By \eqref{eq:normalized-deviation} and \eqref{eq:harmonic-variance},
\begin{equation*}
\sum_{j\geq2}[j(j+2)-2]\|H_j\|_2^2=\beta^{-1}-2\|H\|_2^2-\|H_1\|_2^2,\quad \sum_{j\geq3}[j(j+2)-8]\|H_j\|_2^2=8\left[\frac{1}{8\beta}+\frac{5}{8}\|H_1\|_2^2-\|H\|_2^2\right].
\end{equation*}
Taking the supremum over unit vectors $v$ gives \eqref{eq:adjacent-harmonics}.

Combining the two inequalities in \eqref{eq:multiplier-first-harmonic}, we have
\begin{equation*}
\begin{aligned}
\|H_1\|_2\leq&\frac{3\beta(2\tau-3)\sqrt{1-r}}{\tau(\tau+1)+3\beta(2\tau-3)}\left|\int_{\mathbb{S}^3}\xi(|\nabla H|^2-4H^2)\,d\sigma\right|\\
&+\frac{2(2\tau-3)}{[\tau(\tau+1)+3\beta(2\tau-3)]\sqrt{1-r}}\left[1+\frac{12\beta(2\tau-3)}{\tau(\tau+1)+3\beta(2\tau-3)}\frac{\int_{\mathbb{S}^3} u\,d\sigma}{\sqrt a}\right]\left|\int_{\mathbb{S}^3}\xi f\,d\sigma\right|.
\end{aligned}
\end{equation*}
The coefficient in the first term satisfies
\begin{equation*}
\frac{3\beta(2\tau-3)\sqrt{8(1-r)}}{\tau(\tau+1)+3\beta(2\tau-3)}\left[\beta^{-1}-2\|H\|_2^2-\|H_1\|_2^2\right]^{\frac{1}{2}}<\frac{2}{3}.
\end{equation*}
Using \eqref{eq:adjacent-harmonics}, we obtain
\begin{equation*}
\frac{3\beta(2\tau-3)\sqrt{1-r}}{\tau(\tau+1)+3\beta(2\tau-3)}\left|\int_{\mathbb{S}^3}\xi(|\nabla H|^2-4H^2)\,d\sigma\right|\leq\frac{2}{3}\left[\frac{1}{8\beta}+\frac{5}{8}\|H_1\|_2^2-\|H\|_2^2\right]^{\frac{1}{2}}.
\end{equation*}

To estimate $\left|\int_{\mathbb{S}^3}\xi f\,d\sigma\right|$, recall from \eqref{eq:variance-bounds} that
\begin{equation*}
(1+r)\sqrt{\frac{r}{1+3r}}-\frac{\int_{\mathbb{S}^3} u\,d\sigma}{\sqrt a} =\frac{(1-r)[\|H\|_2^2-\frac{r^2}{1+3r}]}{(1+r)\sqrt{\frac{r}{1+3r}}+\frac{\int_{\mathbb{S}^3} u\,d\sigma}{\sqrt a}}.
\end{equation*}
We also have
\begin{equation*}
\frac{1+\frac{12\beta(2\tau-3)}{\tau(\tau+1)+3\beta(2\tau-3)} \frac{\int_{\mathbb{S}^3} u\,d\sigma}{\sqrt a}}{(1+r)\sqrt{\frac{r}{1+3r}}+\frac{\int_{\mathbb{S}^3} u\,d\sigma}{\sqrt a}}<\frac{8}{5}.
\end{equation*}
Combining this with \eqref{eq:contact-estimates}, we obtain
\begin{equation*}
\begin{aligned}
&\frac{2(2\tau-3)}{[\tau(\tau+1)+3\beta(2\tau-3)]\sqrt{1-r}}\left[1+\frac{12\beta(2\tau-3)}{\tau(\tau+1)+3\beta(2\tau-3)}\frac{\int_{\mathbb{S}^3} u\,d\sigma}{\sqrt a}\right]\left|\int_{\mathbb{S}^3}\xi f\,d\sigma\right|\\
\leq&\frac{21(2\tau-3)\sqrt{1-r}}{2[\tau(\tau+1)+3\beta(2\tau-3)]}\frac{8}{5}\left[\|H\|_2^2-\frac{r^2}{1+3r}\right]\\
\leq&\frac{3}{2}\left[\|H\|_2^2-\frac{r^2}{1+3r}\right].
\end{aligned}
\end{equation*}
The two contributions to $\|H_1\|_2$ thus satisfy \eqref{eq:first-harmonic-bound}.

Combining \eqref{eq:first-harmonic-gap} and \eqref{eq:first-harmonic-bound} with the Cauchy--Schwarz inequality, we obtain
\begin{equation*}
\|H_1\|_2^2\leq\left[\frac{4}{9}+\frac{9}{4}\left(\|H\|_2^2-\frac{r^2}{1+3r}\right)\right]\left[\frac{1}{8\beta}+\frac{5}{8}\|H_1\|_2^2-\frac{r^2}{1+3r}\right] <\frac{5}{8}\left(\frac{4}{9}+\frac{9}{4}\frac{2}{5}\right)\|H_1\|_2^2<\|H_1\|_2^2,
\end{equation*}
which contradicts $H_1\ne0$. Thus $8\beta r^2<1+3r$ and \eqref{eq:upper-cutoff} follows.
\end{proof}

By Lemma~\ref{n3:inputs:upper-cutoff}, $r<\frac{43}{64}$. We now use \eqref{eq:defect-variance}, \eqref{eq:hemisphere-bound} and \eqref{eq:mean-bounds} to estimate $r$ and $1-m$.

\begin{lemma}\label{n3:inputs:localization}
Every solution under consideration satisfies
\begin{equation}\label{eq:localization}
0<r<\frac{1}{64},\qquad \frac{5}{8}r<1-m<\frac{2}{3}r.
\end{equation}
\end{lemma}

\begin{proof}
We first estimate $1-m$ using \eqref{eq:hemisphere-bound}. write $s=1-m$. If $s\leq\frac{2}{5}$, the left side of \eqref{eq:hemisphere-bound} satisfies
\begin{equation*}
\frac{1}{2}\left[(1-\frac{s}{2})^{\frac{2}{3}}+(1-\frac{s}{2}-\frac{s^2}{4})^{\frac{1}{3}}\right] >1-\frac{s}{4}-\frac{s^2}{10}.
\end{equation*}
Here we used $(1-z)^{\frac{2}{3}}\geq1-\frac{2z}{3}-\frac{z^2}{6}$ for $0\leq z\leq\frac{1}{5}$ and $(1-z)^{\frac{1}{3}}\geq1-\frac{z}{3}-\frac{3z^2}{16}$ for $0\leq z\leq\frac{6}{25}$. It follows from \eqref{eq:hemisphere-bound} that $s+\frac{2}{5}s^2>4(1-\beta)r$ and hence
\begin{equation}\label{eq:width-lower}
1-m>\frac{4(1-\beta)r}{1+\frac{8}{5}(1-\beta)r}.
\end{equation}
If $s>\frac{2}{5}$, \eqref{eq:width-lower} also holds since its right-hand side is less than $\frac{2}{5}$.

Suppose that $r\geq\frac{1}{64}$. We will prove
\begin{equation*}
g(r)<\frac{1}{r}\left\{1-\left[\frac{2}{\pi}+\left(1-\frac{2}{\pi}\right)\tr M^2\right]^{\frac{2}{3}}\right\}\leq f(r)
\end{equation*}
for the functions $f$ and $g$ defined below and show that $g(r)>f(r)$ on this interval.

For the upper bound, define
\begin{equation}\label{eq:upper-comparison}
f(r)=1-\kappa_0-\left(1-\kappa_0+\frac{4r}{3}\right)\frac{r(1-r)}{1+3r}+\frac{(1+r)^2}{1+3r}\frac{(1-\kappa_0-\frac{r^2}{2})^2}{\frac{3}{7}+\kappa_0}.
\end{equation}
We claim that
\begin{equation}\label{eq:defect-upper}
\frac{1}{r}\left\{1-\left[\frac{2}{\pi}+\left(1-\frac{2}{\pi}\right)\tr M^2\right]^{\frac{2}{3}}\right\}\leq f(r).
\end{equation}
write $z=\sqrt a\int_{\mathbb{S}^3} u^5\,d\sigma-\frac{\int_{\mathbb{S}^3} u\,d\sigma}{\sqrt a}$ and $\tau=3+\frac{1-r}{r(3r+2)}$. If $z\leq0$, then $\sqrt a\int_{\mathbb{S}^3} u^5\,d\sigma\leq\frac{\int_{\mathbb{S}^3} u\,d\sigma}{\sqrt a}$. Using \eqref{eq:defect-variance}, \eqref{eq:variance-bounds} and \eqref{eq:mean-bounds}, we have
\begin{equation*}
\begin{aligned}
&\frac{1}{r}\left\{1-\left[\frac{2}{\pi}+\left(1-\frac{2}{\pi}\right)\tr M^2\right]^{\frac{2}{3}}\right\}\\
\leq&\left(1-\kappa_0+\frac{4r}{3}\right)\frac{(\int_{\mathbb{S}^3} u\,d\sigma)^2}{ar}-\frac{4r}{3}\\
\leq&\left(1-\kappa_0+\frac{4r}{3}\right)\frac{(1+r)^2}{1+3r}-\frac{4r}{3}=1-\kappa_0-\left(1-\kappa_0+\frac{4r}{3}\right)\frac{r(1-r)}{1+3r}\leq f(r).
\end{aligned}
\end{equation*}

If $z>0$, \eqref{eq:mean-bounds} gives
\begin{equation*}
\frac{\int_{\mathbb{S}^3} u\,d\sigma}{\sqrt a}\leq(1+r)\sqrt{\frac{r}{1+3r}}-\frac{z}{\tau+1}.
\end{equation*}
By \eqref{eq:defect-variance},
\begin{equation*}
\begin{aligned}
&\frac{1}{r}\left\{1-\left[\frac{2}{\pi}+\left(1-\frac{2}{\pi}\right)\tr M^2\right]^{\frac{2}{3}}\right\}\\
\leq&\frac{1-\kappa_0}{r}\left[(1+r)\sqrt{\frac{r}{1+3r}}+\frac{\tau z}{\tau+1}\right]^2-\frac{22-9r}{16r}z^2-\frac{4r}{3}+\frac{4}{3}\left[(1+r)\sqrt{\frac{r}{1+3r}}-\frac{z}{\tau+1}\right]^2.
\end{aligned}
\end{equation*}
The quadratic polynomial in $z$ has a negative leading coefficient. Hence
\begin{equation*}
\begin{aligned}
&\frac{1}{r}\left\{1-\left[\frac{2}{\pi}+\left(1-\frac{2}{\pi}\right)\tr M^2\right]^{\frac{2}{3}}\right\}\\
\leq&\left(1-\kappa_0+\frac{4r}{3}\right)\frac{(1+r)^2}{1+3r}-\frac{4r}{3}+\frac{(1+r)^2}{1+3r}\frac{[(1-\kappa_0)\tau-\frac{4r}{3}]_+^2}{\frac{22-9r}{16}(\tau+1)^2-(1-\kappa_0)\tau^2-\frac{4r}{3}}.
\end{aligned}
\end{equation*}
Since $\frac{1}{\tau}\geq\frac{2r}{1+6r}$, $1-\kappa_0-\frac{r^2}{2}>0$ and
\begin{equation*}
\frac{8r^2}{3(1+6r)}\geq\frac{r^2}{2},\qquad \frac{22-9r}{16}\frac{1+10r}{1+6r}\geq\frac{10}{7},
\end{equation*}
we have
\begin{equation*}
\frac{[(1-\kappa_0)\tau-\frac{4r}{3}]_+^2}{\frac{22-9r}{16}(\tau+1)^2-(1-\kappa_0)\tau^2-\frac{4r}{3}}\leq\frac{(1-\kappa_0-\frac{r^2}{2})^2}{\frac{3}{7}+\kappa_0}.
\end{equation*}
This proves \eqref{eq:defect-upper}.

If $m>\frac{1}{2}$, then $1-\tr M^2\geq2m(1-m)$. If $m\leq\frac{1}{2}$, then $1-\tr M^2\geq\frac{1}{2}$. In either case, \eqref{eq:width-lower} gives
\begin{equation*}
1-\tr M^2>\frac{8(1-\beta)r[1-\frac{12}{5}(1-\beta)r]}{[1+\frac{8}{5}(1-\beta)r]^2}.
\end{equation*}
Applying $1-(1-z)^{\frac{2}{3}}\geq\frac{2z}{3}+\frac{z^2}{9}$ gives
\begin{equation}\label{eq:defect-lower}
g(r)<\frac{1}{r}\left\{1-\left[\frac{2}{\pi}+\left(1-\frac{2}{\pi}\right)\tr M^2\right]^{\frac{2}{3}}\right\},
\end{equation}
where
\begin{equation}\label{eq:lower-comparison}
g(r)=\frac{16}{3}(1-\beta)\left(1-\frac{2}{\pi}\right)\frac{1-\frac{12}{5}(1-\beta)r}{[1+\frac{8}{5}(1-\beta)r]^2}+\frac{64}{9}(1-\beta)^2\left(1-\frac{2}{\pi}\right)^2r\frac{[1-\frac{12}{5}(1-\beta)r]^2}{[1+\frac{8}{5}(1-\beta)r]^4}.
\end{equation}

To prove $g(r)>f(r)$ on $[\frac{1}{64},\frac{43}{64}]$, it suffices to verify
\begin{equation*}
(g-f)'''<0,\qquad (g-f)'\left(\frac{1}{64}\right)>0,\qquad (g-f)\left(\frac{1}{64}\right)>0,\qquad (g-f)\left(\frac{43}{64}\right)>0.
\end{equation*}
Indeed, the concavity of $(g-f)'$ and its positivity at the left endpoint imply that $g-f$ attains its minimum at one of the endpoints.

write $\tau=\frac{8}{5}(1-\beta)r$ and $F(\tau)=\frac{1-\frac{3}{2}\tau}{(1+\tau)^2}$. By \eqref{eq:lower-comparison},
\begin{equation*}
g(r)=4(1-\beta)\left[\frac{4}{3}\left(1-\frac{2}{\pi}\right)F(\tau)+\frac{10}{9}\left(1-\frac{2}{\pi}\right)^2\tau F(\tau)^2\right].
\end{equation*}
Here $0\leq\tau<\frac{1}{5}$ and $F>0$, $F'<0$, $F''>0$, $F'''<0$ and
\begin{equation*}
-\frac{1}{2}F'''-(F')^2-FF''=\frac{5+15\tau(14-3\tau)}{4(1+\tau)^6}>0,\qquad (\tau F^2)'''\leq6[(F')^2+FF'']<-3F'''.
\end{equation*}
Consequently, we have
\begin{equation}\label{eq:lower-comparison-derivative}
g'''(r)<0.
\end{equation}
From \eqref{eq:upper-comparison}, we also have
\begin{equation*}
\left(\frac{3}{7}+\kappa_0\right)f'''(r)=5r^2+\frac{10r}{3}+2\kappa_0-\frac{16}{9}+\frac{2(3996\kappa_0-2563)}{63(1+3r)^4}>0.
\end{equation*}
Hence $(g-f)'''<0$.

From \eqref{eq:upper-comparison} and \eqref{eq:lower-comparison}, one can check that
\begin{equation*}
f\left(\frac{1}{64}\right)<g\left(\frac{1}{64}\right),\qquad f\left(\frac{43}{64}\right)<g\left(\frac{43}{64}\right),\qquad f'\left(\frac{1}{64}\right)<-\frac{7}{24}<g'\left(\frac{1}{64}\right).
\end{equation*}
Thus
\begin{equation*}
\min_{r\in[\frac{1}{64},\frac{43}{64}]}(g-f)(r)=\min\left\{(g-f)\left(\frac{1}{64}\right),(g-f)\left(\frac{43}{64}\right)\right\}>0,
\end{equation*}
contrary to \eqref{eq:defect-upper} and \eqref{eq:defect-lower}. Hence $0<r<\frac{1}{64}$.

It follows from \eqref{eq:width-lower} that
\begin{equation}\label{eq:width-ratio}
\frac{1-m}{r}>\frac{4(1-\beta)}{1+\frac{8}{5}(1-\beta)r}>\frac{5}{8}.
\end{equation}

On the other hand, $1-m<\frac{4r}{3}$ gives $m>\frac{47}{48}$. Combining $\tr M^2\leq m^2+(1-m)^2$ with \eqref{eq:defect-bound}, we obtain
\begin{equation*}
\frac{12}{25}m(1-m)<1-\left[\frac{2}{\pi}+\left(1-\frac{2}{\pi}\right)\tr M^2\right]^{\frac{2}{3}}<\frac{5}{16}r.
\end{equation*}
Consequently,
\begin{equation*}
1-m<\frac{125}{192m}r<\frac{125}{188}r<\frac{2}{3}r.
\end{equation*}
This proves \eqref{eq:localization}.
\end{proof}

\begin{proof}[Proof of Theorem~\ref{thm:stable-critical-point-classification}]
Suppose that $u$ is a nonconstant stable solution of \eqref{eq:critical-equation} and write $m=\lambda_1(M)$. For $n=3$, Proposition~\ref{new:weighted:proposition} contradicts stability if $m\leq\frac{3}{10}$. If $m>\frac{3}{10}$, Lemma~\ref{n3:inputs:localization} gives \eqref{eq:localization} and Lemma~\ref{n3:small:transverse-instability} gives \eqref{eq:transverse-instability}, which also contradicts stability.

For $n=4$, Propositions~\ref{new:weighted:proposition} and~\ref{new:four-principal} give a negative second variation for $m\leq1-\alpha$ and $m\geq1-\alpha$, respectively.

For $n=5,6$, we have $\frac{1+\frac{3}{n}}{6}\geq\frac{1}{4}$, so Propositions~\ref{new:weighted:proposition} and~\ref{new:five:proposition} cover all possible values of $m$.

For $n\geq7$, we have
\begin{equation*}
\frac{1+\frac{3}{n}}{6}-\frac{p}{12}=\frac{n-6}{6n(n-2)}>0,
\end{equation*}
so Propositions~\ref{new:weighted:proposition} and~\ref{new:high-principal} again give a negative second variation.

Each case contradicts stability. Thus $u$ is constant. Normalization and \eqref{eq:critical-equation} give $u=1$, $a=1$ and $b=0$.
\end{proof}

\section{Variational reduction and proof of the inequality}\label{sec:variational-reduction}

We now deduce Theorem~\ref{thm:sharp-improved-sobolev} from Theorem~\ref{thm:stable-critical-point-classification}. To allow for concentration at the critical exponent, for $u\in H^1(\Sn)$, distinct points $e_i\in\Sn$ and positive masses $r_i$, we define
\begin{equation*}
\Erel(u,\{(r_i,e_i)\})=\|u\|_2^2+\beta\|\nabla u\|_2^2+\alpha\sum_i r_i^{\frac{2}{p}}
\end{equation*}
and
\begin{equation*}
\Lrel=\left\{(u,\{(r_i,e_i)\}):\int_{\Sn}|u|^p\,d\sigma+\sum_i r_i=1,\quad \int_{\Sn}\xi|u|^p\,d\sigma+\sum_i r_ie_i=0\right\}.
\end{equation*}
The collection of atoms is at most countable and may be empty.

\begin{proposition}\label{prop:relaxed-minimizer-dichotomy}
For every $n\geq3$, then we have $\inf _{\Lrel}\Erel\leq $ is attained. 

Furthermore, if $(u,\{(r_i,e_i)\})\in\Lrel$ is a minimizer with $u\geq0$, then 
\begin{enumerate}
\item either
\begin{equation*}
(u,\{(r_i,e_i)\})=(0,\{(\frac{1}{2},e),(\frac{1}{2},-e)\})
\end{equation*}
for some $e\in\Sn$

\item or $\{(r_i,e_i)\}=\emptyset$ and $u$ is a smooth positive stable solution of \eqref{eq:critical-equation} with $Q\geq0$ on $\Tu$.
\end{enumerate}
\end{proposition}

\begin{proof}
Since $(1,\emptyset)\in \Lrel$, we have $\inf_{\Lrel}\Erel\leq1$. 

Let $(u_j,\{(r_{j,i},e_{j,i})\})$ be a minimizing sequence. Without loss of generality, we assume $u_j\geq0$. Otherwise, we replace $u_j$ by $|u_j|$, which preserves the constraints and the energy. Since $\{u_j\}$ is bounded in $H^1(\Sn)$, up to a subsequence, we have
\begin{equation*}
u_j\rightharpoonup u\text{ in }H^1(\Sn),\qquad u_j\to u\text{ in }L^2(\Sn)\text{ and a.e. on }\Sn.
\end{equation*}

Write $v_j=u_j-u$. Passing to a further subsequence, we have
\begin{equation*}
|v_j|^p\,d\sigma\rightharpoonup\nu,\qquad \beta|\nabla v_j|^2\,d\sigma\rightharpoonup\mu
\end{equation*}
weakly as measures for some nonnegative Radon measure $\mu$ and $\nu$. 

Applying \eqref{eq:sobolev} to $\varphi v_j$ for $\varphi\in C^\infty(\Sn)$ and using $v_j\to0$ in $L^2(\Sn)$, we obtain
\begin{equation*}
\alpha\left(\int_{\Sn}|\varphi|^p\,d\nu\right)^{\frac{2}{p}}\leq\int_{\Sn}\varphi^2\,d\mu.
\end{equation*}
By approximation, we have 
\begin{equation*}
\alpha\nu(A)^{\frac{2}{p}}\leq\mu(A)
\end{equation*}
for any Borel set $A\subset\Sn$. 

If the non-atomic part of $\nu$ has mass $s>0$, we can choose disjoint Borel sets $A_1,\dots,A_N$ outside the atoms of $\nu$ such that $\nu(A_i)=\frac{s}{N}$. Then we have
\begin{equation*}
\mu(\Sn)\geq\sum_{i=1}^N\mu(A_i)\geq\alpha s^{\frac{2}{p}}N^{1-\frac{2}{p}}\to\infty\quad\text{as }N\to\infty,
\end{equation*}
a contradiction. It follows that $\nu=\sum_l\gamma_l\delta_{\zeta_l}$ with distinct points $\zeta_l\in\Sn$ and $\gamma_l>0$. By the Br\'ezis--Lieb lemma and concentration-compactness \cite{BrezisLieb,Lions}, we have
\begin{equation*}
u_j^p\,d\sigma\rightharpoonup u^p\,d\sigma+\sum_l\gamma_l\delta_{\zeta_l}.
\end{equation*}
Moreover, the weak convergence in $H^1(\Sn)$ gives
\begin{equation*}
\liminf_{j\to\infty}(\|u_j\|_2^2+\beta\|\nabla u_j\|_2^2)=\|u\|_2^2+\beta\|\nabla u\|_2^2+\mu(\Sn)\geq\|u\|_2^2+\beta\|\nabla u\|_2^2+\alpha\sum_l\gamma_l^{\frac{2}{p}}.
\end{equation*}

Without loss of generality, we assume $r_{j,i}$ is nonincreasing in $i$ and extend each finite collection by zero masses. The energy bound then gives $i r_{j,i}^{\frac{2}{p}}\leq C$ for all $j,i$. Consequently, we have
\begin{equation*}
\sup_j\sum_{i>N}r_{j,i}\leq C^{\frac{p}{2}}\sum_{i>N}i^{-\frac{p}{2}}\to0\quad\text{as }N\to\infty.
\end{equation*}

By diagonal argument, we may assume that $r_{j,i}\to r_i$ and $e_{j,i}\to e_i$ for every $i$. The uniform bound above implies
\begin{equation*}
\sum_i r_{j,i}\delta_{e_{j,i}}\rightharpoonup\sum_i r_i\delta_{e_i}.
\end{equation*}
Together with the weak convergence of $u_j^p\,d\sigma$, this gives
\begin{equation*}
\int_{\Sn} u^p\,d\sigma+\sum_i r_i+\sum_l\gamma_l=1,\qquad \int_{\Sn}\xi u^p\,d\sigma+\sum_i r_i e_i+\sum_l\gamma_l\zeta_l=0.
\end{equation*}

After merging coincident atoms and discarding zero masses, denote the resulting collection by $\{(\widehat r_i,\widehat e_i)\}$. Thus $(u,\{(\widehat r_i,\widehat e_i)\})\in\Lrel$. Since $(s+t)^{\frac{2}{p}}\leq s^{\frac{2}{p}}+t^{\frac{2}{p}}$ for $s,t\geq0$, Fatou's lemma and the preceding energy estimate yield
\begin{equation*}
\Erel(u,\{(\widehat r_i,\widehat e_i)\})\leq\|u\|_2^2+\beta\|\nabla u\|_2^2+\alpha\sum_i r_i^{\frac{2}{p}}+\alpha\sum_l\gamma_l^{\frac{2}{p}}\leq\liminf_{j\to\infty}\Erel(u_j,\{(r_{j,i},e_{j,i})\})=\inf_{\Lrel}\Erel.
\end{equation*}
Hence the infimum is attained.

Let $(u,\{(r_i,e_i)\})$ be a minimizer with $u\geq0$. If $u=0$, then $\sum_i r_i=1$ and $\sum_i r_ie_i=0$. In particular,
\begin{equation*}
r_i=\left|\sum_{l\neq i}r_le_l\right|\leq1-r_i.
\end{equation*}
It follows that
\begin{equation*}
1\geq\alpha\sum_i r_i^{\frac{2}{p}}=\sum_i r_i(2r_i)^{-\frac{2}{n}}\geq\sum_i r_i=1.
\end{equation*}
Thus the collection consists of two equally weighted atoms at antipodal points, each of mass $\frac{1}{2}$.

Suppose now that $u\neq0$. The derivatives of the constraints with respect to $u$ are linearly independent. Indeed, if $c\in\mathbb R$ and $v\in\mathbb R^{n+1}$ satisfy
\begin{equation*}
\int_{\Sn}(c+v\cdot\xi)u^{p-1}w\,d\sigma=0\quad\text{for every }w\in H^1(\Sn).
\end{equation*}
Taking $w=(c+v\cdot\xi)u$ gives $\int_{\Sn}(c+v\cdot\xi)^2u^p\,d\sigma=0$. It follows that $c=0$ and $v=0$. Then there exist Lagrange multipliers $a\in\mathbb R$ and $b\in\mathbb R^{n+1}$ such that
\begin{equation}\label{eq:relaxed-equation}
(1-\beta\Delta)u=(a+b\cdot\xi)u^{p-1}.
\end{equation}

Varying the masses and positions of any finite collection of atoms together with $u$ gives the same multipliers $a,b$ and
\begin{equation}\label{eq:atom-stationarity}
a+b\cdot e_i=\alpha r_i^{-\frac{2}{n}}=(2r_i)^{-\frac{2}{n}},\qquad b=(b\cdot e_i)e_i.
\end{equation}
In particular, $r_i\geq\frac{1}{2}(a+|b|)^{-\frac{n}{2}}$ if an atom is present, so there are only finitely many atoms. 

Testing \eqref{eq:relaxed-equation} against $u$ and using \eqref{eq:atom-stationarity}, we obtain
\begin{equation*}
\begin{aligned}
\Erel(u,\{(r_i,e_i)\})&=\int_{\Sn}(a+b\cdot\xi)u^p\,d\sigma+\sum_i(a+b\cdot e_i)r_i\\
&=a\left(\int_{\Sn} u^p\,d\sigma+\sum_i r_i\right)+b\cdot\left(\int_{\Sn}\xi u^p\,d\sigma+\sum_i r_ie_i\right)=a.
\end{aligned}
\end{equation*}
Hence, we have $0<a\leq1$.

We next prove the regularity of $u$. Write $-\Delta u=Vu$, where
\begin{equation*}
V=\beta^{-1}[(a+b\cdot\xi)u^{p-2}-1]\in L^{\frac{n}{2}}(\Sn).
\end{equation*}

For $s\geq1$ and $L>0$, let $h_L=u\min(u,L)^{s-1}$. Testing the equation against $u\min(u,L)^{2s-2}$ gives
\begin{equation*}
\int_{\Sn}|\nabla h_L|^2\,d\sigma\leq s^2\int_{\Sn}\nabla u\cdot\nabla(u\min(u,L)^{2s-2})\,d\sigma=s^2\int_{\Sn} Vh_L^2\,d\sigma.
\end{equation*}
For fixed $s$, choose $R_s>0$ such that
\begin{equation*}
\frac{4s^2}{n(n-2)}\|V\mathbf1_{\{|V|>R_s\}}\|_{\frac{n}{2}}\leq\frac{1}{2}.
\end{equation*}
By \eqref{eq:sobolev} and H\"older's inequality, we have
\begin{equation*}
\|h_L\|_p^2\leq\|h_L\|_2^2+\frac{4s^2}{n(n-2)}\int_{\Sn}|V|h_L^2\,d\sigma\leq\left(1+\frac{4s^2R_s}{n(n-2)}\right)\|h_L\|_2^2+\frac{1}{2}\|h_L\|_p^2.
\end{equation*}

If $u\in L^{2s}(\Sn)$, letting $L\to\infty$ gives
\begin{equation*}
\|u\|_{sp}^{2s}\leq 2\left(1+\frac{4s^2R_s}{n(n-2)}\right)\|u\|_{2s}^{2s}.
\end{equation*}

Starting with $2s=p$ and iterating, we have $u\in L^q(\Sn)$ for every finite $q$. Elliptic regularity and the strong maximum principle give $u>0$ and $u\in C^\infty(\Sn)$.

If the collection of atoms is empty, then $u$ satisfies \eqref{eq:critical-equation}. We prove that $u$ is stable. For $f\in C^\infty(\Sn)\cap\Tu$, consider
\begin{equation*}
u_\tau=u+\tau f+(c(\tau)+v(\tau)\cdot\xi)u.
\end{equation*}
The derivative of the constraints with respect to $(c,v)$ at $(\tau,c,v)=(0,0,0)$ is $p\operatorname{diag}(1,M)$, which is invertible since $u>0$. Since $f\in\Tu$, the implicit function theorem gives smooth functions $c(\tau)$ and $v(\tau)$ with $c(\tau),v(\tau)=O(\tau^2)$ such that $u_\tau>0$ and $(u_\tau,\emptyset)\in\Lrel$ for $|\tau|$ sufficiently small. In particular, $u_0=u$ and $u'_0=f$. The constraints give
\begin{equation*}
\int_{\Sn} u^{p-1}u''_0\,d\sigma=-(p-1)\int_{\Sn} u^{p-2}f^2\,d\sigma,\qquad \int_{\Sn}\xi u^{p-1}u''_0\,d\sigma=-(p-1)\int_{\Sn}\xi u^{p-2}f^2\,d\sigma.
\end{equation*}

By minimality and \eqref{eq:relaxed-equation}, we have
\begin{equation*}
\begin{aligned}
0\leq&\frac{1}{2}\frac{d^2}{d\tau^2}\Big|_{\tau=0}(\|u_\tau\|_2^2+\beta\|\nabla u_\tau\|_2^2)\\
=&\|f\|_2^2+\beta\|\nabla f\|_2^2+\int_{\Sn}(a+b\cdot\xi)u^{p-1}u''_0\,d\sigma\\
=&\|f\|_2^2+\beta\|\nabla f\|_2^2-(p-1)\int_{\Sn}(a+b\cdot\xi)u^{p-2}f^2\,d\sigma=Q(f,f).
\end{aligned}
\end{equation*}

For any $f\in\Tu$, choose $f_j\in C^\infty(\Sn)$ such that $f_j\to f$ in $H^1(\Sn)$. Then
\begin{equation*}
f_j-\left(\int_{\Sn} u^{p-1}f_j\,d\sigma\right)u-\left(M^{-1}\int_{\Sn}\xi u^{p-1}f_j\,d\sigma\right)\cdot\xi\,u\in\Tu
\end{equation*}
and this sequence converges to $f$ in $H^1(\Sn)$. By continuity of $Q$, we have $Q(f,f)\geq0$.

It remains to exclude the case where $u\neq0$ and at least one atom is present. For every atom, the constraints give
\begin{equation*}
2r_i=1-\int_{\Sn}(1+e_i\cdot\xi)u^p\,d\sigma-\sum_{l\neq i}r_l(1+e_i\cdot e_l)<1.
\end{equation*}
Using \eqref{eq:atom-stationarity} and $a\leq1$, we obtain $b\cdot e_i>0$ and hence $e_i=\frac{b}{|b|}$. Since the points $e_i$ are distinct, there is exactly one atom, of mass $0<\rho<\frac{1}{2}$ at $e=\frac{b}{|b|}$. Thus
\begin{equation*}
\int_{\Sn} u^p\,d\sigma=1-\rho,\qquad \int_{\Sn}\xi u^p\,d\sigma=-\rho e,\qquad a+|b|=\alpha\rho^{-\frac{2}{n}}.
\end{equation*}
write $x=e\cdot\xi$ and
\begin{equation*}
m=\int_{\Sn} x^2u^p\,d\sigma+\rho,\qquad m_3=\int_{\Sn} x^3u^p\,d\sigma+\rho,\qquad \xi_\perp=\int_{\Sn} x(\xi-xe)u^p\,d\sigma.
\end{equation*}
We first prove that
\begin{equation}\label{eq:atom-moments}
\frac{\rho}{1-\rho}\leq m<1,\qquad |\xi_\perp|^2\leq\left(m-\frac{\rho}{1-\rho}\right)(1-m),\qquad m_3\leq m(1-m).
\end{equation}

Since $\int_{\Sn} xu^p\,d\sigma=-\rho$, the Cauchy--Schwarz inequality gives $\rho^2\leq(1-\rho)(m-\rho)$ and hence $m\geq\frac{\rho}{1-\rho}$. Also, $1-m=\int_{\Sn}(1-x^2)u^p\,d\sigma>0$. Using $\int_{\Sn}(\xi-xe)u^p\,d\sigma=0$, we have
\begin{equation*}
\begin{aligned}
|\xi_\perp|^2&=\left|\int_{\Sn}\left(x+\frac{\rho}{1-\rho}\right)(\xi-xe)u^p\,d\sigma\right|^2\\
&\leq\int_{\Sn}\left(x+\frac{\rho}{1-\rho}\right)^2u^p\,d\sigma\int_{\Sn}(1-x^2)u^p\,d\sigma=\left(m-\frac{\rho}{1-\rho}\right)(1-m).
\end{aligned}
\end{equation*}

Finally, the inequality $0\leq\int_{\Sn}(1-x)(x+m)^2u^p\,d\sigma=m(1-m)-m_3$ gives the third inequality in \eqref{eq:atom-moments}.

For any $w\in C^\infty(\Sn)$, we consider the admissible curve $(u_\tau,\{(\rho_\tau,e_\tau)\})$ defined by
\begin{equation*}
V_\tau=\int_{\Sn}\xi(u+\tau w)^p\,d\sigma,\qquad S_\tau=\int_{\Sn}(u+\tau w)^p\,d\sigma+|V_\tau|,
\end{equation*}
\begin{equation*}
u_\tau=\frac{u+\tau w}{S_\tau^{\frac{1}{p}}},\qquad \rho_\tau=\frac{|V_\tau|}{S_\tau},\qquad e_\tau=-\frac{V_\tau}{|V_\tau|}.
\end{equation*}

Since $u>0$, $V_0=-\rho e$ and $S_0=1$, this curve is smooth near $\tau=0$. Its energy is
\begin{equation*}
\Erel(u_\tau,\{(\rho_\tau,e_\tau)\})=S_\tau^{-\frac{2}{p}}\left(\|u+\tau w\|_2^2+\beta\|\nabla(u+\tau w)\|_2^2+\alpha|V_\tau|^{\frac{2}{p}}\right).
\end{equation*}
We have
\begin{equation*}
\begin{aligned}
\frac{d}{d\tau}\Big|_{\tau=0}|V_\tau|&=-p\int_{\Sn} xu^{p-1}w\,d\sigma,\\
\frac{d^2}{d\tau^2}\Big|_{\tau=0}|V_\tau|&=-p(p-1)\int_{\Sn} xu^{p-2}w^2\,d\sigma+\frac{p^2}{\rho}\left|\int_{\Sn}(\xi-xe)u^{p-1}w\,d\sigma\right|^2.
\end{aligned}
\end{equation*}
Consequently,
\begin{equation*}
\begin{aligned}
S'_0&=p\int_{\Sn}(1-x)u^{p-1}w\,d\sigma,\\
S''_0&=p(p-1)\int_{\Sn}(1-x)u^{p-2}w^2\,d\sigma+\frac{p^2}{\rho}\left|\int_{\Sn}(\xi-xe)u^{p-1}w\,d\sigma\right|^2.
\end{aligned}
\end{equation*}

Let $N_\tau=\|u+\tau w\|_2^2+\beta\|\nabla(u+\tau w)\|_2^2+\alpha|V_\tau|^{\frac{2}{p}}$. By \eqref{eq:relaxed-equation} and \eqref{eq:atom-stationarity},
\begin{equation*}
N_0=a,\qquad N'_0=2\int_{\Sn}(a+|b|x)u^{p-1}w\,d\sigma-2(a+|b|)\int_{\Sn} xu^{p-1}w\,d\sigma=\frac{2a}{p}S'_0
\end{equation*}
and
\begin{equation*}
\begin{aligned}
\frac{1}{2} N''_0=&\|w\|_2^2+\beta\|\nabla w\|_2^2-(p-1)(a+|b|)\int_{\Sn} xu^{p-2}w^2\,d\sigma\\
&-\frac{(p-2)(a+|b|)}{\rho}\left(\int_{\Sn} xu^{p-1}w\,d\sigma\right)^2+\frac{p(a+|b|)}{\rho}\left|\int_{\Sn}(\xi-xe)u^{p-1}w\,d\sigma\right|^2.
\end{aligned}
\end{equation*}
Combining these identities, we obtain
\begin{equation}\label{eq:mixed-variation}
\begin{aligned}
&0\leq\Hc(w):=\frac{1}{2}\frac{d^2}{d\tau^2}\Big|_{\tau=0}\Erel(u_\tau,\{(\rho_\tau,e_\tau)\})\\
=&\frac{1}{2} N''_0-\frac{a}{p}S''_0+\frac{(p-2)a}{p^2}(S'_0)^2\\
=&Q(w,w)+(p-2)a\left(\int_{\Sn}(1-x)u^{p-1}w\,d\sigma\right)^2-\frac{(p-2)(a+|b|)}{\rho}\left(\int_{\Sn} xu^{p-1}w\,d\sigma\right)^2\\
&\quad+\frac{p|b|}{\rho}\left|\int_{\Sn}(\xi-xe)u^{p-1}w\,d\sigma\right|^2.
\end{aligned}
\end{equation}

We choose $w=Z_e-\frac{1}{2}xu$. The identity \eqref{eq:conformal-integration} gives
\begin{equation*}
\begin{aligned}
\int_{\Sn}(1-x)u^{p-1}w\,d\sigma&=\frac{1-m}{p}+\frac{m}{2}=\frac{n-2+2m}{2n},\\
\int_{\Sn} xu^{p-1}w\,d\sigma&=-\frac{1-m}{p}-\frac{m-\rho}{2}=\frac{\rho}{2}-\frac{n-2+2m}{2n},\\
\int_{\Sn}(\xi-xe)u^{p-1}w\,d\sigma&=\left(\frac{1}{p}-\frac{1}{2}\right)\xi_\perp=-\frac{\xi_\perp}{n}.
\end{aligned}
\end{equation*}
Using \eqref{eq:linearized-equation}, \eqref{eq:conformal-integration}, \eqref{eq:product-test} and \eqref{eq:conformal-pairing}, we have
\begin{equation*}
\begin{aligned}
Q(Z_e,Z_e)&=(1-\alpha)\int_{\Sn}(3x^2-1)u^2\,d\sigma-\frac{2|b|}{p}m_3,\\
Q(Z_e,xu)&=2(1-\alpha)\int_{\Sn} x^2u^2\,d\sigma-|b|m_3,\\
Q(xu,xu)&=\beta\int_{\Sn}(1-x^2)u^2\,d\sigma-(p-2)[a(m-\rho)+|b|(m_3-\rho)].
\end{aligned}
\end{equation*}
It follows that
\begin{equation*}
\begin{aligned}
Q(w,w)&=Q(Z_e,Z_e)-Q(Z_e,xu)+\frac{1}{4}Q(xu,xu)\\
&=-\left(1-\alpha-\frac{\beta}{4}\right)\int_{\Sn}(1-x^2)u^2\,d\sigma+\frac{n-4}{n(n-2)}|b|m_3-\frac{p-2}{4}am+\frac{p-2}{4}(a+|b|)\rho.
\end{aligned}
\end{equation*}
Substituting these identities into \eqref{eq:mixed-variation}, we obtain
\begin{equation}\label{eq:mixed-conformal-variation}
\begin{aligned}
\Hc(w)=&-\left(1-\alpha-\frac{\beta}{4}\right)\int_{\Sn}(1-x^2)u^2\,d\sigma+\frac{n-4}{n(n-2)}|b|m_3+(p-2)a\left[\frac{(n-2+2m)^2}{4n^2}-\frac{m}{4}\right]\\
&-(p-2)(a+|b|)\left[\frac{(n-2+2m)^2}{4n^2\rho}-\frac{n-2+2m}{2n}\right]+\frac{2|b|}{n(n-2)\rho}|\xi_\perp|^2.
\end{aligned}
\end{equation}
Since $2^{\frac{2}{n}}-1>\frac{1}{n(n-2)}$ for $n\geq3$, we have $1-\alpha-\frac{\beta}{4}>0$.

If $n\geq4$, we use the upper bound for $m_3$ in \eqref{eq:atom-moments}. Since $\rho\leq\frac{m}{1+m}$, we also have
\begin{equation*}
\frac{m}{\rho}-\frac{1}{1-\rho}=\frac{m}{1-\rho}\left(\frac{1}{\rho}-\frac{1+m}{m}\right)\leq m(1+m)\left(\frac{1}{\rho}-\frac{1+m}{m}\right).
\end{equation*}
Combining \eqref{eq:atom-moments} with \eqref{eq:mixed-conformal-variation}, we obtain
\begin{equation*}
\begin{aligned}
\Hc(w)\leq&\frac{n-4}{n(n-2)}|b|m(1-m)+(p-2)a\left[\frac{(n-2+2m)^2}{4n^2}-\frac{m}{4}\right]\\
&-(p-2)(a+|b|)\left[\frac{(n-2+2m)^2}{4n^2\rho}-\frac{n-2+2m}{2n}\right]+\frac{2|b|m(1-m^2)}{n(n-2)}\left(\frac{1}{\rho}-\frac{1+m}{m}\right)\\
=&-\frac{(n-2)(1-m)^2}{n^2}\left(|b|+\frac{a+|b|}{m}\right)-\frac{(a+|b|)(n-2+2m)^2-2n|b|m(1-m^2)}{n^2(n-2)}\left(\frac{1}{\rho}-\frac{1+m}{m}\right).
\end{aligned}
\end{equation*}
Since
\begin{equation*}
(n-2+2m)^2-2nm(1-m^2)=4+4m^2+8m^3+(n-4)(n+2m+2m^3)>0,
\end{equation*}
we have $\Hc(w)<0$, which contradicts \eqref{eq:mixed-variation}.

If $n=3$, the coefficient of $m_3$ in \eqref{eq:mixed-conformal-variation} is negative. Taking $k=\frac{m-2\rho}{1-2\rho}$ in
\begin{equation*}
0\leq\int_{\mathbb{S}^3}(1+x)(x-k)^2u^6\,d\sigma=m_3+m-2\rho-2k(m-2\rho)+k^2(1-2\rho),
\end{equation*}
we obtain
\begin{equation}\label{unified:n3-atom-third}
m_3\geq\frac{(1-m)(2\rho-m)}{1-2\rho}.
\end{equation}
Using \eqref{eq:atom-moments} and \eqref{unified:n3-atom-third} in \eqref{eq:mixed-conformal-variation}, we have
\begin{equation*}
\begin{aligned}
\Hc(w)\leq&-a\left[\frac{(1-m)^2}{9m}+\frac{(2m+1)^2[m-(1+m)\rho]}{9m\rho}\right]\\
&+|b|\left[-\frac{(1-m)(2\rho-m)}{3(1-2\rho)}-\frac{(1+2m)^2}{9\rho}+\frac{2(1+2m)}{3}+\frac{2(1-m)}{3\rho}\left(m-\frac{\rho}{1-\rho}\right)\right].
\end{aligned}
\end{equation*}
The coefficient of $a$ is strictly negative by \eqref{eq:atom-moments}. To determine the sign of the coefficient of $|b|$, write $y=1-2\rho\in(0,1)$ and $z=\frac{m(1-\rho)-\rho}{1-2\rho}\in[0,1)$. We have
\begin{equation*}
\begin{aligned}
&-\frac{(1-m)(2\rho-m)}{3(1-2\rho)}-\frac{(1+2m)^2}{9\rho}+\frac{2(1+2m)}{3}+\frac{2(1-m)}{3\rho}\left(m-\frac{\rho}{1-\rho}\right)\\
=&-\frac{2y\{9y^2z+y[(5z-2)^2+9z^2]+3z(2z+1)\}}{9(1-y)(1+y)^2}<0.
\end{aligned}
\end{equation*}
Thus $\Hc(w)<0$, which again contradicts \eqref{eq:mixed-variation}. This completes the proof.
\end{proof}

\begin{proof}[Proof of Theorem~\ref{thm:sharp-improved-sobolev}]
By Proposition~\ref{prop:relaxed-minimizer-dichotomy} and Theorem~\ref{thm:stable-critical-point-classification}, every nonnegative minimizer of $\Erel$ on $\Lrel$ is either $(0,\{(\frac{1}{2},e),(\frac{1}{2},-e)\})$ for some $e\in\Sn$ or $(1,\emptyset)$. In either case, the minimum is one. 

For every nonzero admissible $f$, we have $(\frac{|f|}{\|f\|_p},\emptyset)\in\Lrel$. Hence
\begin{equation*}
1\leq\Erel\left(\frac{|f|}{\|f\|_p},\emptyset\right)=\frac{\|f\|_2^2+\beta\|\nabla f\|_2^2}{\|f\|_p^2}.
\end{equation*}

If the equality holds, then $(\frac{|f|}{\|f\|_p},\emptyset)$ is a minimizer and hence $|f|$ is a positive constant. It follows that $\nabla f=\operatorname{sgn}(f)\nabla|f|=0$ a.e. on $\Sn$. Thus $f$ is constant. Conversely, equality holds for all constants.

Constants show that the coefficient $1$ of $\|u\|_2^2$ is optimal. The optimality of the coefficient $\beta$ of $\|\nabla u\|_2^2$ follows from antipodal atoms as in \cite[Remark~2.4]{WYY}.
\end{proof}


\begin{thebibliography}{99}

\bibitem{Aubin76} T. Aubin, \emph{Probl\`emes isop\'erim\'etriques et espaces de Sobolev}, J. Differential Geom. \textbf{11} (1976), no. 4, 573--598.

\bibitem{Aubin79} T. Aubin, \emph{Meilleures constantes dans le th\'eor\`eme d'inclusion de Sobolev et un th\'eor\`eme de Fredholm non lin\'eaire pour la transformation conforme de la courbure scalaire}, J. Funct. Anal. \textbf{32} (1979), no. 2, 148--174.

\bibitem{AubinLi} T. Aubin and Y. Y. Li, \emph{On the best Sobolev inequality}, J. Math. Pures Appl. (9) \textbf{78} (1999), no. 4, 353--387.

\bibitem{Beckner} W. Beckner, \emph{Sharp Sobolev inequalities on the sphere and the Moser--Trudinger inequality}, Ann. of Math. (2) \textbf{138} (1993), no. 1, 213--242.

\bibitem{BianchiEgnell} G. Bianchi and H. Egnell, \emph{A note on the Sobolev inequality}, J. Funct. Anal. \textbf{100} (1991), no. 1, 18--24.

\bibitem{BrezisLieb} H. Brezis and E. Lieb, \emph{A relation between pointwise convergence of functions and convergence of functionals}, Proc. Amer. Math. Soc. \textbf{88} (1983), no. 3, 486--490.

\bibitem{Case} J. S. Case, \emph{The Frank--Lieb approach to sharp Sobolev inequalities}, Commun. Contemp. Math. \textbf{23} (2021), no. 3, 2050015.

\bibitem{ChangYang91} S.-Y. A. Chang and P. C. Yang, \emph{A perturbation result in prescribing scalar curvature on $\mathbb{S}^n$}, Duke Math. J. \textbf{64} (1991), no. 1, 27--69.

\bibitem{CFM} G. Ciraolo, A. Figalli and F. Maggi, \emph{A quantitative analysis of metrics on $\mathbb{R}^n$ with almost constant positive scalar curvature, with applications to fast diffusion flows}, Int. Math. Res. Not. IMRN \textbf{2018} (2018), no. 21, 6780--6797.

\bibitem{DSW} B. Deng, L. Sun and J.-C. Wei, \emph{Sharp quantitative estimates of Struwe's decomposition}, Duke Math. J. \textbf{174} (2025), no. 1, 159--228.

\bibitem{Dolbeault} J. Dolbeault, \emph{Functional inequalities: nonlinear flows and entropy methods as a tool for obtaining sharp and constructive results}, Milan J. Math. \textbf{89} (2021), no. 2, 355--386.

\bibitem{DEFFL} J. Dolbeault, M. J. Esteban, A. Figalli, R. L. Frank and M. Loss, \emph{Sharp stability for Sobolev and log-Sobolev inequalities, with optimal dimensional dependence}, Camb. J. Math. \textbf{13} (2025), no. 2, 359--430.

\bibitem{DEL} J. Dolbeault, M. J. Esteban and M. Loss, \emph{Interpolation inequalities on the sphere: linear vs. nonlinear flows}, Ann. Fac. Sci. Toulouse Math. (6) \textbf{26} (2017), no. 2, 351--379.

\bibitem{Druet99} O. Druet, \emph{The best constants problem in Sobolev inequalities}, Math. Ann. \textbf{314} (1999), no. 2, 327--346.

\bibitem{FG} A. Figalli and F. Glaudo, \emph{On the sharp stability of critical points of the Sobolev inequality}, Arch. Ration. Mech. Anal. \textbf{237} (2020), no. 1, 201--258.

\bibitem{FigalliZhang} A. Figalli and Y. R.-Y. Zhang, \emph{Sharp gradient stability for the Sobolev inequality}, Duke Math. J. \textbf{171} (2022), no. 12, 2407--2459.

\bibitem{FJWX} R. L. Frank, T. Jin, H.-Y. Wang and J. Xiong, \emph{Sharp Sobolev inequalities under first moment constraints}, 2026, arXiv:2610.20001.

\bibitem{FKK20} R. L. Frank, T. K\"onig and H. Kova\v{r}\'ik, \emph{Energy asymptotics in the Brezis--Nirenberg problem: The higher-dimensional case}, Math. Eng. \textbf{2} (2020), no. 1, 119--140.

\bibitem{FKK21} R. L. Frank, T. K\"onig and H. Kova\v{r}\'ik, \emph{Energy asymptotics in the three-dimensional Brezis--Nirenberg problem}, Calc. Var. Partial Differential Equations \textbf{60} (2021), Paper No. 58.

\bibitem{FrankLieb} R. L. Frank and E. H. Lieb, \emph{A new, rearrangement-free proof of the sharp Hardy--Littlewood--Sobolev inequality}, Spectral Theory, Function Spaces and Inequalities, Oper. Theory Adv. Appl., vol. 219, Birkh\"auser/Springer Basel, Basel, 2012, pp. 55--67.

\bibitem{GJMS} C. R. Graham, R. Jenne, L. J. Mason and G. A. J. Sparling, \emph{Conformally invariant powers of the Laplacian. I. Existence}, J. London Math. Soc. (2) \textbf{46} (1992), no. 3, 557--565.

\bibitem{GHLWXY} C. Gui, Y. Hu, T. Li, J. Wei, Z. Xie and Z. Ye, \emph{Rigidity of a $q$-curvature-type problem on $\mathbb{S}^N$ in every dimension $N\geq3$}, 2026, arXiv:2610.11338.

\bibitem{GLWY} C. Gui, T. Li, J. Wei and Z. Ye, \emph{A positive answer to the generalized Chang--Yang conjecture on $\mathbb{S}^N$}, 2026, arXiv:2608.13497.

\bibitem{GuiMoradifam} C. Gui and A. Moradifam, \emph{The sphere covering inequality and its applications}, Invent. Math. \textbf{214} (2018), no. 3, 1169--1204.

\bibitem{HangWang} F. Hang and X. Wang, \emph{Improved Sobolev inequality under constraints}, Int. Math. Res. Not. IMRN \textbf{2022} (2022), no. 14, 10822--10857.

\bibitem{HebeyVaugon} E. Hebey and M. Vaugon, \emph{Meilleures constantes dans le th\'eor\`eme d'inclusion de Sobolev}, Ann. Inst. H. Poincar\'e Anal. Non Lin\'eaire \textbf{13} (1996), no. 1, 57--93.

\bibitem{JXZ} T. Jin, J. Xiong and N. Zhou, \emph{\L{}ojasiewicz--Simon inequalities near bubbling configurations for the Yamabe functional on bounded domains}, 2026, arXiv:2608.30173.

\bibitem{Konig23} T. K\"onig, \emph{On the sharp constant in the Bianchi--Egnell stability inequality}, Bull. Lond. Math. Soc. \textbf{55} (2023), no. 4, 2070--2075.

\bibitem{Konig25} T. K\"onig, \emph{Stability for the Sobolev inequality: Existence of a minimizer}, J. Eur. Math. Soc. (2025), published online, doi:10.4171/JEMS/1582.

\bibitem{Li96} Y. Y. Li, \emph{Prescribing scalar curvature on $\mathbb{S}^n$ and related problems, Part II: Existence and compactness}, Comm. Pure Appl. Math. \textbf{49} (1996), no. 6, 541--597.

\bibitem{Lieb} E. H. Lieb, \emph{Sharp constants in the Hardy--Littlewood--Sobolev and related inequalities}, Ann. of Math. (2) \textbf{118} (1983), no. 2, 349--374.

\bibitem{Lions} P.-L. Lions, \emph{The concentration-compactness principle in the calculus of variations. The limit case. I}, Rev. Mat. Iberoamericana \textbf{1} (1985), no. 1, 145--201.

\bibitem{Putterman} E. Putterman, \emph{Cubature formulas and Sobolev inequalities}, 2020, arXiv:2012.08109.

\bibitem{SunZou} J. Sun and W. Zou, \emph{Some improved almost sharp inequalities under constraints}, Int. Math. Res. Not. IMRN \textbf{2025} (2025), no. 16, Paper No. rnaf242, 27 pp.

\bibitem{Talenti} G. Talenti, \emph{Best constant in Sobolev inequality}, Ann. Mat. Pura Appl. (4) \textbf{110} (1976), 353--372.

\bibitem{WZ} G. Wang and M. Zhang, \emph{Another sharp criterion for Dirac zero modes}, 2026, arXiv:2609.34958.

\bibitem{WYY} J. Wei, Z. Yan and Z. Ye, \emph{On Dolbeault's conjecture}, 2026, arXiv:2610.14359.

\end{thebibliography}
\end{document}